\documentclass{article}[11pt]

\usepackage[
left=1.25in, right=1.25in]{geometry}
\usepackage[all]{xy}
\usepackage{graphicx}
\usepackage{indentfirst}
\usepackage{bm}
\usepackage{amsthm}
\usepackage{mathrsfs}
\usepackage{latexsym}
\usepackage{amsmath}
\usepackage{amssymb}
\usepackage{color}
\usepackage{hyperref}
\usepackage{microtype}

\newcommand{\grs}[1]{\raisebox{-.16cm}{\includegraphics[height=.5cm]{UP#1.pdf}}}
\newcommand{\gra}[1]{\raisebox{-.4cm}{\includegraphics[height=1cm]{UP#1.pdf}}}
\newcommand{\graa}[1]{\raisebox{-.6cm}{\includegraphics[height=1.5cm]{UP#1.pdf}}}
\newcommand{\grb}[1]{\raisebox{-.8cm}{\includegraphics[height=2cm]{UP#1.pdf}}}
\newcommand{\grc}[1]{\raisebox{-1.3cm}{\includegraphics[height=3cm]{UP#1.pdf}}}

\begin{document}
\theoremstyle{plain}
\newtheorem{theorem}{Theorem~}[section]
\newtheorem{main}{Main Theorem~}
\newtheorem{lemma}[theorem]{Lemma~}
\newtheorem{proposition}[theorem]{Proposition~}
\newtheorem{corollary}[theorem]{Corollary~}
\newtheorem{definition}[theorem]{Definition~}
\newtheorem{notation}[theorem]{Notation~}
\newtheorem{example}[theorem]{Example~}
\newtheorem*{remark}{Remark~}
\newtheorem*{cor}{Corollary~}
\newtheorem*{question}{Question}
\newtheorem*{claim}{Claim}
\newtheorem*{conjecture}{Conjecture~}
\newtheorem*{fact}{Fact~}
\renewcommand{\proofname}{\bf Proof}



\title{\textbf{Noncommutative Uncertainty Principles}\\
} 

\author{Chunlan Jiang \thanks{Chunlan Jiang was supported in part by NSFC (Grant no. A010602) and Foundation for the Author of National Excellent Doctoral Dissertation of China (Grant no. 201116).}\\
{\small Department of Mathematics}\\
{\small Hebei Normal University}\\
{\small cljiang@mail.hebtu.edu.cn}
\and Zhengwei Liu \thanks{ Zhengwei Liu was supported by DOD-DARPA grant HR0011-12-1-0009.}\\
{\small Department of Mathematics}\\
{\small Vanderbilt University}\\
{\small zhengwei.liu@vanderbilt.edu}\\
\and Jinsong Wu\thanks{Jinsong Wu was supported by the project-sponsored by OATF, USTC and supported by "PCSIRT" and the Fundamental Research Funds for the Central Universities.(WK0010000024)} \\
{\small School of Mathematical Sciences}\\
{\small University of Science and Technology of China}\\
{\small wjsl@ustc.edu.cn} \\
{}}


\date{\today}

\maketitle

\hspace*{3,6mm}\textit{Keywords:}
Uncertainty principle, Hausdorff-Young inequality, Young's inequality, Subfactor, Planar algebra, finite group, Kac algebra, quantum group.
\vspace{30pt} 

{\small} {\small \textbf{2010 MR Subject Classification 46L37,43A30,94A15} }
\begin{abstract}
The classical uncertainty principles deal with functions on abelian groups. In this paper, we discuss the uncertainty principles for finite index subfactors which include the cases for finite groups and finite dimensional Kac algebras.
We prove the Hausdorff-Young inequality, Young's inequality, the Hirschman-Beckner uncertainty principle, the Donoho-Stark uncertainty principle.
We characterize the minimizers of the uncertainty principles. We also prove that the minimizer is uniquely determined by the supports of itself and its Fourier transform.
The proofs take the advantage of the analytic and the categorial perspectives of subfactor planar algebras. Our method to prove the uncertainty principles also works for more general cases, such as Popa's $\lambda$-lattices, modular tensor categories etc.
\end{abstract}

\section{Introduction}
In quantum mechanics, Heisenberg's uncertainty principle says that the more precisely the position of some particle is determined, the less precisely its momentum can be known. This uncertainty principle is a fundamental result in quantum mechanics. By using the Shannon entropy of the measurement, entropic uncertainty principles were established, which strengthen and generalize Heisenberg's uncertainty principle. The phenomenon of ``uncertainty'' also appears in mathematics. In Harmonic Analysis, various uncertainty principles are found by mathematicians such as Hardy's uncertainty principle \cite{Hardy}, Hirschman's uncertainty principle \cite{Hirsch} etc. In information theory, Donoho and Stark \cite{DoSt} introduced an uncertainty principle. They proved that a signal may be reconstructed with a few samples. This idea was further developed by Candes, Romberg and Tao \cite{CRT} in the theory of compressed sensing.

The classical uncertainty principles deal with functions on an abelian group and its dual. A natural question is to ask whether there exist uncertainty principles for a non-abelian group. Recently the uncertainty principle of finite groups \cite{GGI} and Kac algebras ($C^*$ Hopf algebras or quantum groups in von Neumann algebraic setting) \cite{CrKa14} were discussed from different perspectives.
It is natural to consider the uncertainty principles of Kac algebras due to the existense of a proper pair of observables.
There is a one to one correspondence between Kac algebras and a special family of subfactors, namely irreducible \emph{depth-2} subfactors \cite{Sz94}.
Subfactors naturally provide vector spaces with two or more observables.
In this paper, we are going to talk about the uncertainty principles of subfactors.

Modern subfactor theory was initiated by Jones. Subfactors generalize the symmetries of groups and their duals. The index of a subfactor generalizes the order of a group. All possible indices of subfactors,
$$\{4\cos^2\frac{\pi}{n}, n=3,4,\cdots\}\cup[4,\infty],$$ were found by Jones in his remarkable rigidity result \cite{Jon83}.
A deep theorem of Popa \cite{Pop94} showed that the \emph{standard invariant}, a $\mathbb{Z}_2$ graded filtered algebra, is a complete invariant of \emph{finite depth} subfactors of the hyperfinite factor of type II$_1$.
There are three axiomatizations of the standard invariant: Ocneanu's paragroups \cite{Ocn88}; Popa's standard $\lambda$-lattices \cite{Pop95}; Jones' subfactor planar algebras \cite{Jon}.

In terms of paragroups, the $n$-box space of the standard invariant of a (finite index, finite depth, irreducible) subfactor consists of linear sums of certain length-$2n$ loops of a bipartite graph, the \emph{principal graph} of the subfactor.
The Fourier transform of paragroups was introduced by Ocneanu \cite{Ocn88} as the 1-click rotation of loops. It generalizes the Fourier transform of finite groups.
Note that the Fourier transform of the $n$-box space has periodicity $2n$. Up to an anti-isomorphism, the n-box space admits $n$ *-algebraic structures and a proper $n$-tuple of observables.
We will prove the uncertainty principles of $n$-box spaces.
In particular, if we start with a subfactor arising from an outer action of a finite abelian group, then we will recover the group and its dual from the 2-box spaces of the standard invariant of the subfactor. Moreover, the classical uncertainty principles of a finite abelian group will be recovered from those of 2-box spaces.
Hence the uncertainty principles of 2-box spaces are our main interest.

For 2-box spaces, we will prove the Hausdorff-Young inequality, Young's inequality, the Hirschman-Beckner uncertainty principle and the Donoho-Stark uncertainty principle and characterize the minimizers of the uncertainty principles. Our proof of the Hausdorff-Young inequality also works for $n$-box spaces and more general cases, such as Popa's $\lambda$-lattices, modular tensor categories etc. The uncertainty principles follow directly from the Hausdorff-Young inequality. So far we could not find Young's inequality for Kac algebras in the literature.
Our characterization of the minimizers of the uncertainty principles is new for finite groups.
The proofs of these results benefit a lot from the analytic and the categorial perspectives of subfactor planar algebras.
We could not find an alternative proof without the knowledge of subfactor planar algebras.

The minimizers of the uncertainty principles of finite abelian groups are given by translations and modulations of indicator functions of its subgroups \cite{DoSt,OzPr}.
Unlike abelian groups, the 2-box space forms a noncommutative algebra in general. This makes extra difficulties to characterize the minimizers of the uncertainty principles.
Fortunately Bisch and Jones introduced \emph{biprojections} \cite{Bis94,BJ97a} in subfactors which generalize the indicator functions of subgroups.
We will introduce a notion of \emph{bi-shifts of biprojections} in terms of subfactor planar algebras as an equivalent characterization of the minimizers.
Furthermore, we will prove that a bi-shift of a biprojection is uniquely determined by the \emph{supports} of itself and its Fourier transform. We would like to see the application of  this result in signal recovery.

It will be interesting to understand the uncertainty principle of infinite index subfactors. But the Fourier transform is not clear in general. We would like to refer to \cite{Burn03, EnNe96, JonPen} for some known results of infinite index subfactors. There is no canonical choice of the measurement in general. Connes' spatial theory \cite{Con80} should be involved. These will lead us to the world of Abstract Harmonic Analysis and Noncommutative Geometry.

This paper is organized as follows.
In Section 2, we briefly introduce classical uncertainty principles, subfactors and our main results.
In Section 3, we review subfactor planar algebras and some notations, such as the Fourier transform, Wenzl's formula, the local relation etc.
In Section 4, we find the Hausdorff-Young inequality and Young's inequality for irreducible subfactor planar algebras.
In Section 5, we show the Donoho-Stark and the Hirschman-Beckner uncertainty principles for irreducible subfactor planar algebras.
In Section 6, we show that the two uncertainty principles found in the last section have the same minimizers and determine the minimizers in terms of biprojections.
In Section 7, we generalize the two uncertainty principles to general cases.
In Sections 8 and 9, we show some applications of our uncertainty principles for groups and group actions.

\subsection*{Acknowledgements}
This work was initiated while Jinsong Wu was visiting Department of Mathematics, Vanderbilt University, Nashville, TN, US in the fall of 2013 and was completed during one-month-long research visit of Zhengwei Liu and Jinsong Wu to Hebei Normal University. Zhengwei Liu would like to thank Vaughan F. R. Jones for a fruitful discussion. Jinsong Wu is grateful for Jesse Peterson's hospitality and suggestion, and he is also grateful to Dechao Zheng.

\section{Overview}

In 1927, Heisenberg discovered the uncertainty principle
$$\Delta x\Delta p\geq \frac{\hbar}{2},$$
where $x$ is the position of a particle, $p$ is its momentum, $\hbar$ is the reduced Planck constant. It has a mathematic fomulation
$$(\int_{-\infty}^{\infty}x^2|f(x)|^2dx)(\int_{-\infty}^{\infty}\xi^2|\hat{f}(\xi)|^2d\xi)\geq\frac{\|f\|_2^4}{16\pi^2},$$
where $\displaystyle\hat{f}(\xi)=\int_{-\infty}^{\infty}f(x)e^{-2\pi i x\xi}dx$.

In 1957, Hirschman \cite{Hirsch} proved a stronger uncertainty principle for the Schwartz space on $\mathbb{R}$,
$$H_s(|f|^2)+H_s(|\hat{f}|^2)\geq 0,\quad \|f\|_2=1,$$
where $H_s$ is the Shannon Entropy
$\displaystyle H_s(|f^2|)=-\int_{-\infty}^{\infty}|f|^2\log|f|^2dx$.
He conjectured that
$$H_s(|f|^2)+H_s(|\hat{f}|^2)\geq \log \frac{e}{2}, \quad \|f\|_2=1.$$
This was proved by Beckner \cite{Be75} in 1975.

In 1989, Donoho and Stark \cite{DoSt} established an uncertainty principle for finite cyclic groups $\mathbb{Z}_n$.

Suppose $f$ is a nonzero function on $\mathbb{Z}_n$ and $\hat{f}$ is its Fourier transform,
$$\hat{f}(k)=\sum_{j=0}^{n-1}f(j)e^{2\pi ijk/n}.$$
The support of $f$ is denoted by $\text{supp}(f)=\{x\in \mathbb{Z}_n: f(x)\neq 0\}$. Then
\begin{equation}\label{DS}
|\text{supp}(f)||\text{supp}(\hat{f})|\geq n.
\end{equation}

They proved that the equality of (\ref{DS}) holds if and only if $f=c\sum_{h\in H}\chi(h)\delta_{h+k}$, where $c$ is a nonzero constant, $H$ is a subgroup of $\mathbb{Z}_n$, $\chi$ is a character of $\mathbb{Z}_n$, $k$ is an element in $\mathbb{Z}_n$, and $\delta_{h+k}$ is the dirac function at $h+k$.


In 1990, K. Smith \cite{Sm90} generalized the Donoho-Stark uncertainty principle to finite abelian groups.

In 2004, \"{O}zaydm and Przebinda \cite{OzPr} generalized the Hischman-Beckner uncertainty principle and the Donoho-Stark uncertainty principle to abelian groups. They characterized all minimizers
of the uncertainty principles and noticed that the minimizers of two uncertainty principles coincide.

In 2005, Tao \cite{Tao05} proved a stronger uncertainty principle for the cyclic group $\mathbb{Z}_p$, $p$ prime.


In this paper, we discuss the uncertainty principles for finite index subfactors which cover the case for finite groups and finite dimensional Kac algebras.
We prove the Haussdorff-Young inequality (see Theorem \ref{Young1}), Young's inequality (see Theorem \ref{Young2}),
the Hirshman-Beckner uncertainty principle (see Theorem \ref{entropybase}), the Donoho-Stark uncertainty principle (see Theorem \ref{UnP}), and characterize the minimizers of the uncertainty principles (see Theorem \ref{minmain1},\ref{minmain2}).

In order to state the results, let us review some basic concepts and notations in Subfactor theory.
Suppose $\mathcal{N}\subset\mathcal{M}$ is an irreducible subfactor of type II$_1$ with finite index, denoted as $[\mathcal{M}:\mathcal{N}]$, and $\delta=\sqrt{[\mathcal{M}:\mathcal{N}]}$.
Let $L^2(\mathcal{M})$ be the Gelfand-Naimark-Segal representation of $\mathcal{M}$ derived from the unique trace of $\mathcal{M}$. The Jones projection $e_1\in B(L^2(\mathcal{M}))$ is the projection onto the subspace $L^2(\mathcal{N})$ of $L^2(\mathcal{M})$. Then we have the Jones basic construction $\mathcal{N}\subset\mathcal{M}\subset\mathcal{M}_1$, where $\mathcal{M}_1=(\mathcal{M}\cup \{e_1\})''$.

Repeating the Jones basic construction, we have the Jones tower

$$\begin{array}{ccccccccc}
&&&e_1&&e_2&&e_3&\\
\mathcal{N}&\subset&\mathcal{M}&\subset&\mathcal{M}_1&\subset&\mathcal{M}_2&\subset&\cdots\\
\end{array}$$

Taking the higher relative commutants, we have the standard invariant of the subfactor

$$\begin{array}{ccccccccc}
\mathcal{N}'\cap\mathcal{N}&\subset&\mathcal{N}'\cap\mathcal{M}&\subset&\textcolor{red}{\mathcal{N}'\cap\mathcal{M}_1}&\subset&\mathcal{N}'\cap\mathcal{M}_2&\subset&\cdots\\
&&\cup&&\cup&&\cup&&\\
&&\mathcal{M}'\cap\mathcal{M}&\subset&\mathcal{M}'\cap\mathcal{M}_1&\subset&\textcolor{red}{\mathcal{M}'\cap\mathcal{M}_2}&\subset&\cdots\\
\end{array}$$

Jones introduced subfactor planar algebras as an axiomatization of the standard invariant.
The subfactor planar algebra associated with the subfactor $\mathcal{N}\subset\mathcal{M}$ is a collection of $\mathbb{Z}_2$ graded vertor spaces $\{\mathscr{P}_{n,\pm}\}_{n\in\mathbb{N}\cup \{0\}}$ with actions labeled by \emph{planar tangles}. Precisely $\mathscr{P}_{n,+}=\mathcal{N}'\cap\mathcal{M}_{n-1}$, $\mathscr{P}_{n,-}=\mathcal{M}'\cap\mathcal{M}_{n}$, called the $n$-box space of the subfactor planar algebras. An element in the $n$-box space is called an $n$-box. The action of a planar tangle is a composition of elementary operations of the standard invariant, such as multiplications, inclusions and conditional expectations.

If $\mathcal{M}=\mathcal{N}\rtimes G$, for an outer action a finite abelian group $G$, then the index $[\mathcal{M}:\mathcal{N}]$ is the order of the group $|G|$.
We have the Jones tower

$$\begin{array}{ccccccccc}
\mathcal{N}&\subset&\mathcal{N}\rtimes G&\subset&\mathcal{N}\rtimes G \rtimes \hat{G}&\subset&\mathcal{N}\rtimes G \rtimes \hat{G} \rtimes G&\subset&\cdots\\
\end{array}$$
and the standard invariant
$$\begin{array}{ccccccccc}
\mathbb{C}&\subset&\mathbb{C}&\subset&L(\textcolor{red}{\hat{G}})&\subset&L(\textcolor{blue}{\hat{G}\rtimes G)}&\subset&\cdots\\
&&\cup&&\cup&&\cup&&\\
&&\mathbb{C}&\subset&\mathbb{C}&\subset&L(\textcolor{red}{G})&\subset&\cdots\\
\end{array}$$

The \textcolor{red}{2-box spaces} of the standard invariant ($\textcolor{red}{\mathcal{M}'\cap\mathcal{M}_2}$ and $\textcolor{red}{\mathcal{N}'\cap\mathcal{M}_1}$) recover the group \textcolor{red}{$G$} and its dual \textcolor{red}{$\hat{G}$}.
Hence the uncertainty principles of 2-box spaces are our main interest.
Moreover, the 3-box space $\textcolor{blue}{\mathcal{N}'\cap \mathcal{M}_2}$ naturally provides an algebra to consider $G$ and $\hat{G}$ simultaneously.

There is an (unnormalized) trace of the $2$-box space, denoted by $tr_2$.
For an $2$-box $x$, its p-norm is given by
$$\displaystyle{\|x\|_p=tr_2(|x|^p)^\frac{1}{p}}, \quad p\geq1.$$

The Fourier transform $\mathcal{F}$ was introduced by Ocneanu. In terms of planar algebras, it is a 1-click rotation diagrammatically.

We have the Hausdorff-Young inequality for a 2-box $x$ (see Theorem \ref{Young1}),
$$\|\mathcal{F}(x)\|_p\leq \left(\frac{1}{\delta}\right)^{1-\frac{2}{p}}\|x\|_q,\quad 2\leq p\leq\infty,\frac{1}{p}+\frac{1}{q}=1.$$
Recall that $\delta$ is the square root of the index.

The Hausdorff-Young inequality becomes non-trivial for subfactor planar algebras, since the Fourier transform is no longer an integral.
When $p=\infty$, we will apply the local relation, see Relation (\ref{liu}), and the $C^*$ structure of the 3-box space to give a categorial proof.
When $p=2$, the spherical isotopy of subfactor planar algebras implies the Plancherel's formula of the Fourier transform,
$$\|\mathcal{F}(x)\|_2=\|x\|_2.$$
With the help of the interpolation theorem \cite{Ko84}, we obtain the general case.

Applying a similar method in the 4-box space, we obtain Young's inequality for 2-boxes $x$ and $y$ (see Theorem \ref{Young2}),
$$\|x*y\|_r\leq\frac{\|x\|_p\|y\|_q}{\delta}, \quad \frac{1}{p}+\frac{1}{q}=\frac{1}{r}+1.$$
For the group case, $\delta\ x*y$ is the convolution of functions of the group.

By the Hausdorff-Young inequality in subfactor planar algebras, we are able to give the following uncertainty principles.
\begin{main}\label{mainthm1}[See Theorem \ref{UnP}, \ref{entropybase}]
Suppose $x$ is a nonzero 2-box in an irreducible subfactor planar algebra. Then we have the Hirschman-Beckner uncertainty principle,
$$H(|x|^2)+H(|\mathcal{F}(x)|^2)\geq \|x\|_2(2\log\delta-4\log\|x\|_2),$$
where $H(\cdot)$ is the von Neumann entropy, $H(|x|^2)=-tr_2(|x|^2\log|x|^2)$;
and the Donoho-Stark uncertainty principle,
$$\mathcal{S}(x)\mathcal{S}(\mathcal{F}(x))\geq \delta^2,$$
where $\mathcal{S}(x)$ is the trace of the range projection of $x$.
\end{main}

Taking the derivative of the Hausdorff-Young inequality with respect to $p$ at $p=2$,
we derive Hirschman-Beckner uncertainty principle.
By the concavity of $-t\log t$, we have $\log{\mathcal{S}(x)}\geq\log{H(|x|^2)}$, when $\|x\|_2=1$. Then we obtain the Donoho-Stark uncertainty principle.

We apply the noncommutative version of \"{O}zaydm and Przebinda \cite{OzPr} to show the minimizers of the two uncertainty principles coincide.
To characterize the minimizers of Donoho-Stark uncertainty principle, we apply the noncommutative version of Tao's proof in \cite{Tao05}.
We show that the minimizers are \emph{extremal bi-partial isometries} (see Definition \ref{ebipartial} and Theorem \ref{minmain1}).
Extremal bi-partial isometries are convenient to work with theoretically, but not convenient to construct.

Recall that the minimizer of uncertainty principles of finite abelian groups is given by
a nonzero constant multiple of a translation and a modulation of the indicator function of its subgroup.
The original technique for finite abelian groups is not enough to give this kind of characterization for the 2-box space, because the 2-box space may not be a commutative algebra and there is no translation or modulation to shift the minimizer to a good position. Fortunately Bisch and Jones introduced biprojections \cite{Bis94, BJ97a} which generalized the indicator functions of subgroups. We expect to construct minimizers with the knowledge of biprojections.
We prove that an extremal bi-partial isometry $x$ is uniquely determined by the range projections of $x$ and its Fourier transform $\mathcal{F}(x)$ (see Theorem \ref{uniqueness}). Moreover, two range projections are \emph{shifts} of a pair of biprojections (see Definition \ref{shifts}), a generalization of the translation or the modulataion of the indicator function of a subgroup.
To construct such an extremal bi-partial isometry, we introduce the notion of a \emph{bi-shift} of a biprojection (see Definition \ref{bishifts}), a generalization of a translation and a modulation of the indicator function of a subgroup. Bi-shifts of biprojections are easy to construct.

\begin{main}\label{mainthm2}[See Theorem \ref{minmain1},\ref{minmain2} and Proposition \ref{bipartial}]
For a nonzero 2-box $x$ in an irreducible subfactor planar algebra, the following statements are equivalent:
\begin{itemize}
\item[(1)] $\mathcal{S}(x)\mathcal{S}(\mathcal{F}(x))= \delta^2;$
\item[(2)] $H(|x|^2)+H(|\mathcal{F}(x)|^2)=\|x\|_2(2\log\delta-4\log\|x\|_2);$
\item[(3)] $x$ is an extremal bi-partial isometry;
\item[(3')] $x$ is a partial isometry and $\mathcal{F}^{-1}(x)$ is extremal;
\item[(4)] $x$ is a bi-shift of a biprojection,
\end{itemize}
\end{main}
As a corollary, we obtain a new characterization of biprojections.
\begin{corollary}[Corollary \ref{positivebiprojection}]
For a nonzero 2-box $x$ in an irreducible subfactor planar algebra. If $x$ and $\mathcal{F}(x)$ are positive operators and
$\mathcal{S}(x)\mathcal{S}(\mathcal{F}(x))= \delta^2$, then $x$ is a biprojection.
\end{corollary}

As mentioned above, we have the uncertainty principle for 2-boxes of an irreducible subfactor planar algebra. The irreducibility is not necessary, in Section \ref{nbox}, we have the uncertainty principle for 2-boxes of arbitrary subfactor planar algebra. Moreover, we obtain the uncertainty principles for $n$-boxes.

\begin{main}\label{mainthm3}[Theorem \ref{UPnbox}]
For a nonzero n-box $x$ in an irreducible subfactor planar algebra, we have
$$\sum_{k=0}^{n-1} H(|\mathcal{F}^k(x)|^2)\geq \|x\|_2(n\log\delta-2n\log\|x\|_2);$$
$$\prod_{k=0}^{n-1} \mathcal{S}(\mathcal{F}^k(x))\geq\delta^n.$$
\end{main}
The equalities can be achieved on a special $n$-box.

When the subfactor planar algebra is the group subfactor planar algebra associated with a finite group $G$, we will have the uncertainty principle for $G$. It will recover the classical uncertainty principle when $G$ is abelian. Moreover, its minimizer are given as follows.
\begin{proposition}[Proposition \ref{mingroup}]
Suppose $G$ is a finite group. Take a subgroup $H$, a one dimensional representation $\chi$ of $H$, an element $g\in G$, a nonzero constant $c\in\mathbb{C}$. Then
$$x=c\sum_{h\in H}\chi(h)hg$$
is a bi-shift of a biprojection. Conversely any bi-shift of a biprojection is of this form.
\end{proposition}

Let a group $G$ act on a finite set $S$.
The fixed point algebra of the spin model associated with $S$ under the group action of $G$ is a subfactor planar algebra.
The 2-box space of this subfactor planar algebra consists of $S\times S$ matrixes which commute with the action of $G$.
We may apply all the results above to these $S\times S$ matrixes. For example, we have an inequality for the Hadamand product $\circ$ as follows.
\begin{corollary}[Corollary \ref{youngmatrix}]
If $A,B$ are two $S\times S$ matrixes commuting with a group action of $G$ on $S$, then
$$\|A\circ B\|_r\leq \frac{1}{n_0^2}\|A\|_p\|B\|_q,$$
where $\frac{1}{p}+\frac{1}{q}=\frac{1}{r}+1$ and $n_0$ is the minimal cardinality of $G$-orbits of $S$.
\end{corollary}

In particular, when $S=G$ and the action is the group multiplication, we will recover the group subfactor planar algebra and related results.

We would like to emphasize the importance of irreducibility of subfactors while studying its Fourier analysis, although our method also works without this condition. (It even works for planar algebras with multiple kinds of regions and strings.) The inequalities and the uncertainty principles are very sensitive to this condition.

We also consider the maximizer of the Hirschman-Beckner uncertainty principle. If there exists a biunitary in a subfactor planar algebra, then the maximizer is a biunitary and biunitaries are maximizers. But the subfactor planar algebra may not have any biunitary.

\section{Preliminaries}
\subsection{Planar Algebras}
We refer the reader to \cite{Jon12} for the definition of subfactor planar algebras. Now let us recall some notations in planar algebras.

\begin{notation}
In planar tangles, we use a thick string with a number $n$ to indicate $n$ parallel strings.
\end{notation}

Suppose $\mathscr{P}$ is a subfactor planar algebra, then the $n$-box space $\mathscr{P}_{n,\pm}$ forms a $C^*$ algebra.
For $a, b\in \mathscr{P}_{n,\pm}$, the product of $a$ and $b$ is
$ab=\grb{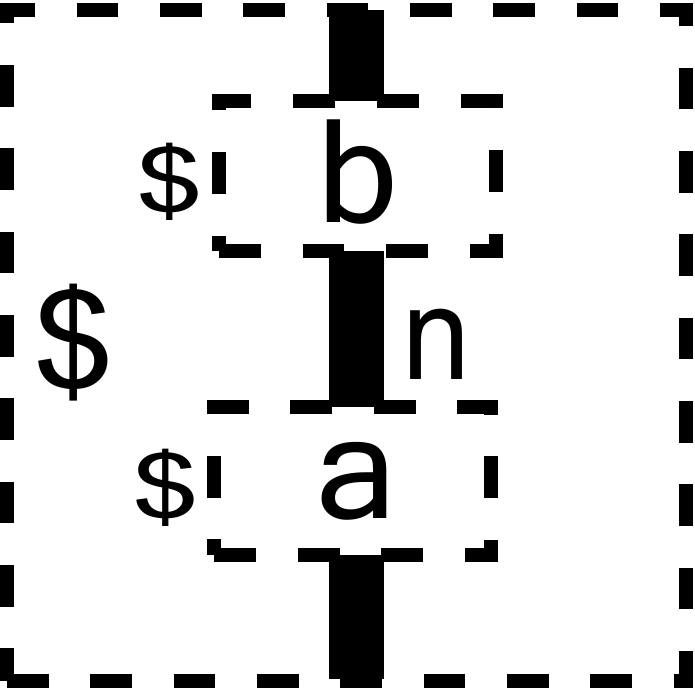}$.
The adjoint operation $*$ is a vertical reflection.
The unnormalized Markov trace $tr_n$ of the $n$-box space is given by the tangle $\grb{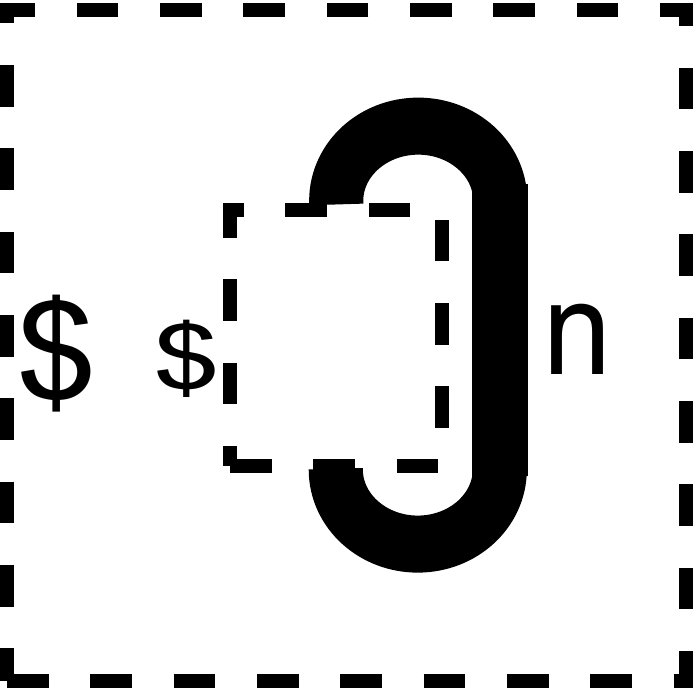}$.
It induces a positive definite inner product of $n$-boxes
$$\langle a,b \rangle=tr_n(b^*a).$$

Note that $\dim(\mathscr{P}_{0,\pm})=1$, then $\mathscr{P}_{0,\pm}$ is isomorphic to $\mathbb{C}$ as a field. The sphericality of the planar algebra means
$$\gra{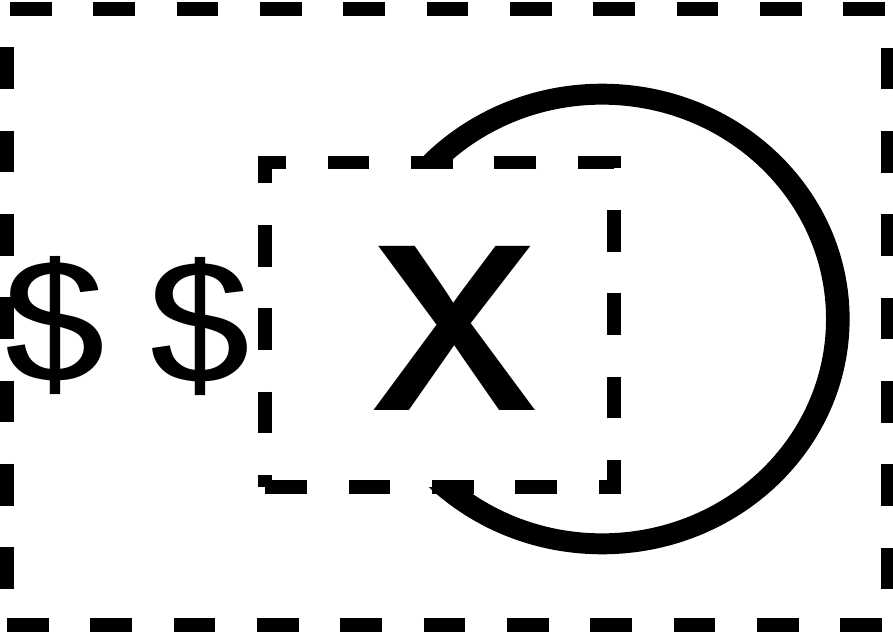}=\gra{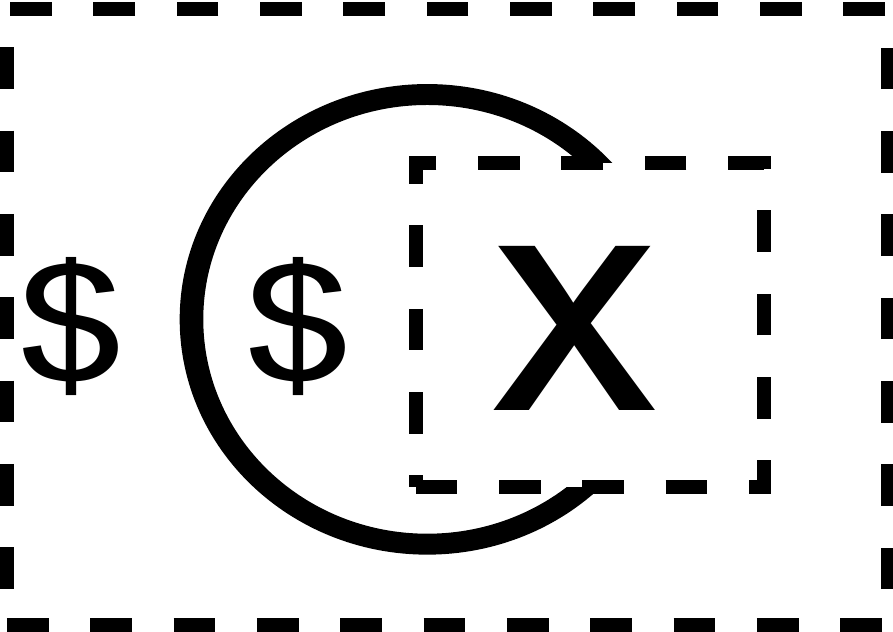}$$
as a scalar in $\mathbb{C}$, for any $x\in\mathscr{P}_{1,\pm}$.

Note that $\mathscr{
P}_{0,\pm}$ is isomorphic to $\mathbb{C}$, the (shaded or unshaded) empty diagram can be identified as the number $1$ in $\mathbb{C}$. The value of a (shaded or unshaded) closed string is $\delta$, the square root of the index.
Moreover, $\delta^{-1}\gra{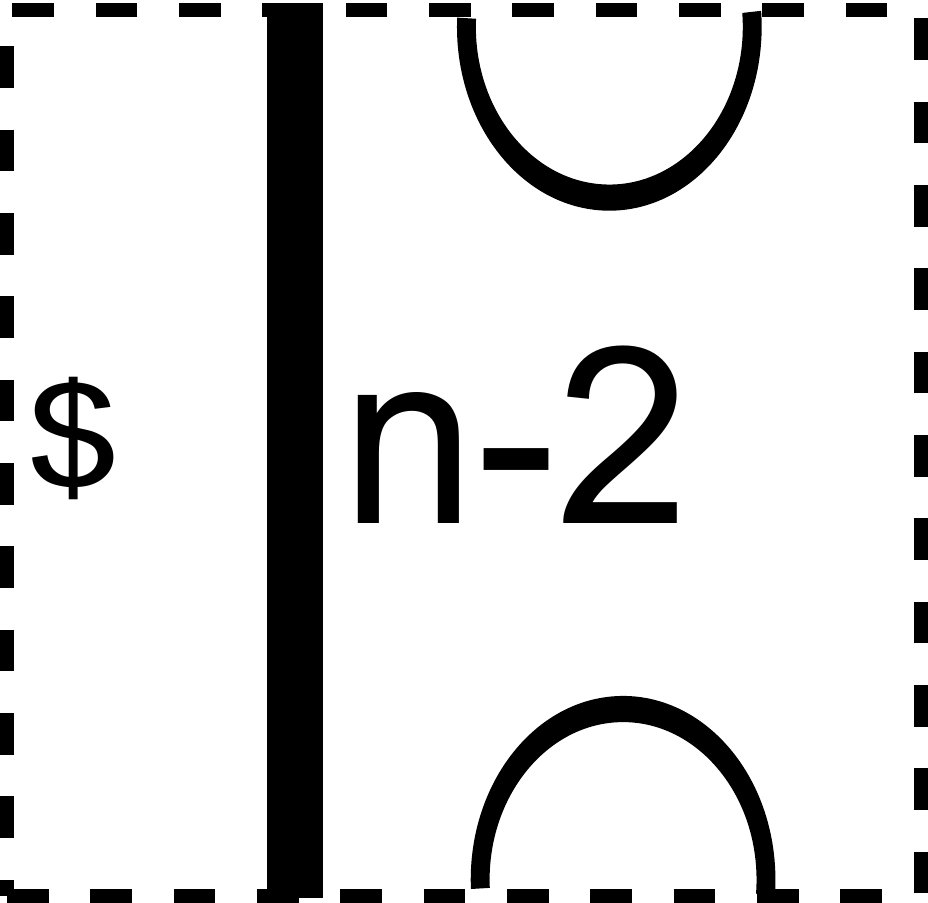}$ in $\mathscr{P}_{n,+}$, denoted by $e_{n-1}$, for $n\geq2$, is the sequences of Jones projections.
The filtered algebra generated by Jones projections is the smallest subfactor planar algebra, well known as the Temperley-Lieb algebra, denoted by $TL(\delta)$.

The Fourier transform from $\mathscr{P}_{n,\pm}$ to $\mathscr{P}_{n,\mp}$ is the 1-click rotation
$$\grb{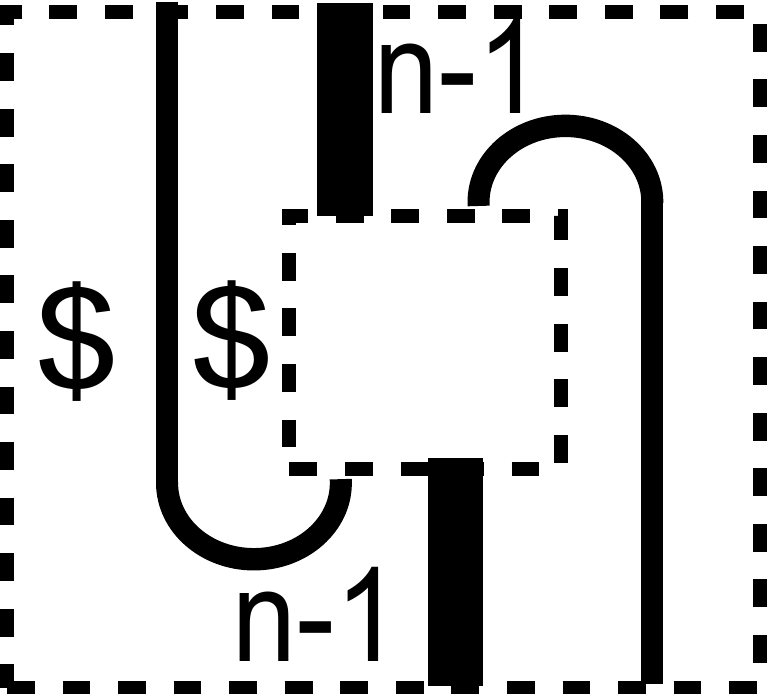}.$$

The Fourier transform and inverse Fourier transform has the following relation.
\begin{proposition}\label{fb}
For any $x$ in $\mathscr{P}_{n,\pm}$, we have $\mathcal{F}^{-1}(x^*)=(\mathcal{F}(x))^*$.
\end{proposition}

For an $n$-box $x$, its contragredient, denoted by $\overline{x}$, is defined as its $180^\circ$ rotation, i.e. $\mathcal{F}^n(x)$.

If $a, b\in \mathscr{P}_{2,\pm}$, then the pull back of the product of $\mathscr{P}_{2,\mp}$ by the Fourier transform gives another product of 2-boxes, called the coproduct, denoted by $a*b$, and $a*b=\mathcal{F}(\mathcal{F}^{-1}(a)\mathcal{F}^{-1}(b))$. The corresponding tangle is as follows
$$a*b=\grb{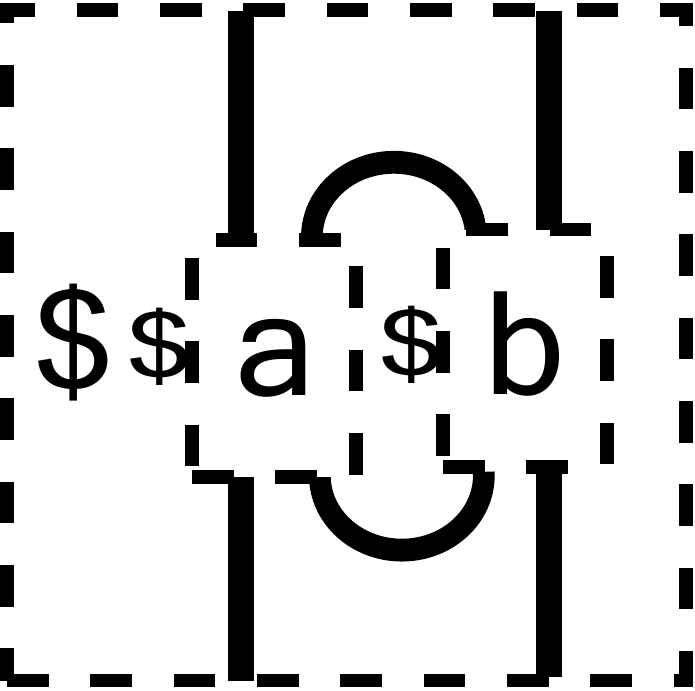}.$$

We collect some properties of coproduct in the following proposition and lemmas.

\begin{proposition}[Schur Product Theorem]\label{schur}
Let $\mathscr{P}$ be a subfactor planar algebra. If $a$, $b$ are positive operators in $\mathscr{P}_{2,\pm}$, then $a*b$ is positive.
\end{proposition}
\begin{proof}
This is the Theorem 4.1 in \cite{Liu13}.
\end{proof}

\begin{lemma}\label{change}
Let $\mathscr{P}$ be a subfactor planar algebra and $a$, $b$, $c$ $\in \mathscr{P}_{2,\pm}$. Then
\begin{eqnarray*}
&&tr_2((a*b)\overline{c})=tr_2((b*c)\overline{a})=tr_2((c*a)\overline{b})\\
&=&tr_2((\overline{c}*\overline{b})a)=tr_2((\overline{a}*\overline{c})b)=tr_2((\overline{b}*\overline{a})c).
\end{eqnarray*}
\end{lemma}
\begin{proof}
This is the Lemma 4.6 in \cite{Liu13}.
\end{proof}

\begin{lemma}\label{PQR=0}
For $x, y\in\mathscr{P}_{2,\pm}$, we have $\mathcal{R}(x*y)\leq \mathcal{R}(\mathcal{R}(x)*\mathcal{R}(y))$, where $\mathcal{R}(\cdot)$ is the range projection of the variable.
\end{lemma}
\begin{proof}
This is equivalent to Lemma 4.7 in \cite{Liu13}.
\end{proof}

\begin{notation}
When the tangle is written as a box and the dollar sign is on the left side, we may omit the boundary and the dollar sign.
\end{notation}

\begin{definition}[\cite{Bis94,BJ97a}]
For an irreducible subfactor planar algebra $\mathscr{P}$,
a projection $Q\in \mathscr{P}_{2,\pm}$ is called a biprojection if $\mathcal{F}(Q)$ is a multiple of a projection.
\end{definition}

It is shown that there is a one to one correspondence between biprojections and intermediate subfactors (see \cite{Bis94,BJ97a}).
If the subfactor is the group subfactor $\mathcal{N}\subset \mathcal{N}\rtimes G$, for an outer action of the group $G$ on the factor $\mathcal{N}$ of type II$_1$, then each intermediate subfactor is $\mathcal{N}\rtimes H$, for some subgroup $H$ of $G$.
Let $\mathscr{P}$ be the subfactor planar algebra of the subfactor $\mathcal{N}\subset \mathcal{N}\rtimes G$, then the corresponding biprojection in $\mathscr{P}_{2,+}$ is the indicator function of $H$.

\begin{definition}
Let $\mathscr{P}$ be a subfactor planar algebra. A biunitary $U$ in $\mathscr{P}_{2,\pm}$ is a unitary element in $\mathscr{P}_{2,\pm}$ such that $\mathcal{F}(U)$ is unitary.
\end{definition}

\subsection{Wenzl's formula and the local relation}
Suppose $\mathscr{P}$ is a subfactor planar algebra,
$\mathscr{I}_{n+1,\pm}$ is the two sided ideal of $\mathscr{P}_{n+1,\pm}$ generated by the Jones projection $e_n$, and $\mathscr{P}_{n+1,\pm}/\mathscr{I}_{n+1,\pm}$ is its orthogonal complement in $\mathscr{P}_{n+1,\pm}$, then $\mathscr{P}_{n+1,\pm}=\mathscr{I}_{n+1,\pm} \oplus \mathscr{P}_{n+1,\pm}/\mathscr{I}_{n+1,\pm}$.
Let $s_{n+1}$ be the support of $\mathscr{P}_{n+1,\pm}/\mathscr{I}_{n+1,\pm}$.

If $\mathscr{P}$ is Temperley-Lieb (with $\delta^2\geq 4$), then $s_n$ is called the $n^{th}$ Jones-Wenzl projection.
The following relation is called Wenzl's formula \cite{Wen87},
$$\grb{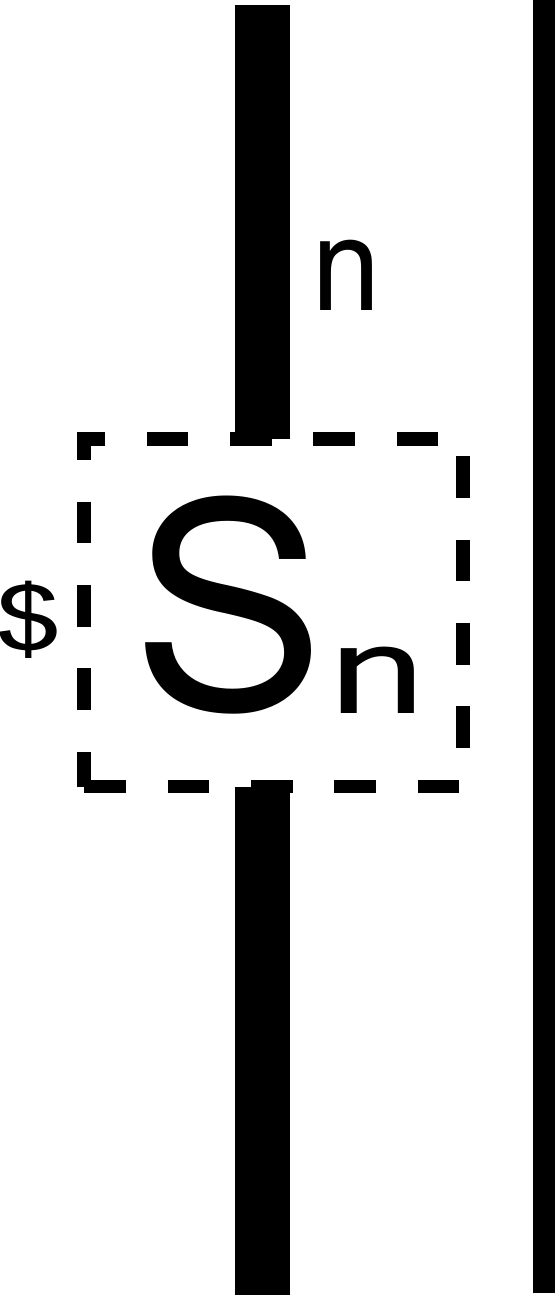}=\frac{tr_{n-1}(s_{n-1})}{tr_n(s_n)}\grb{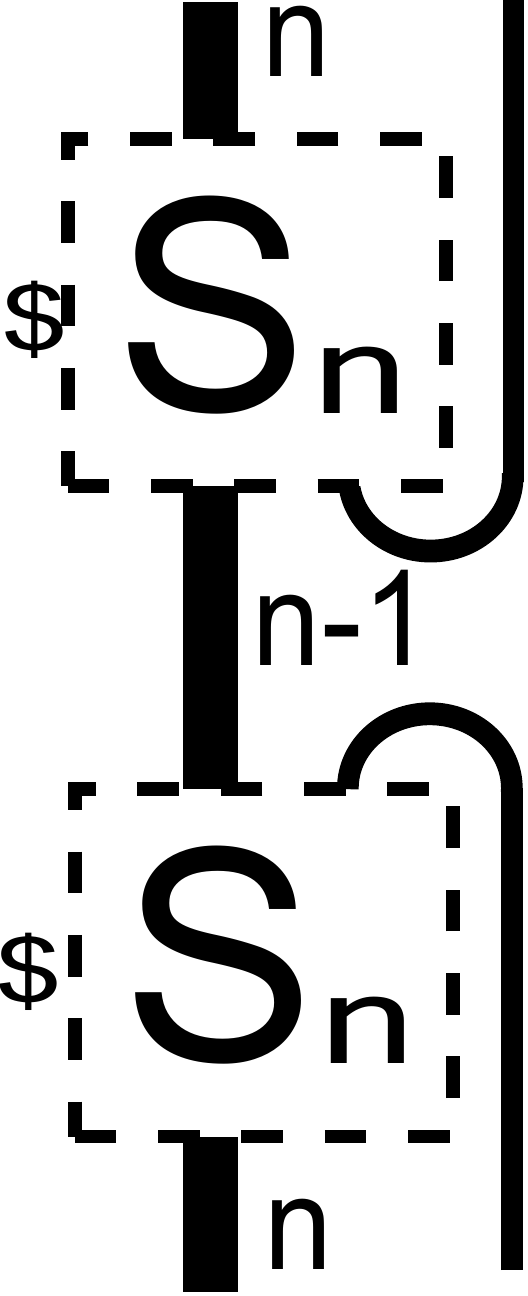}+\grb{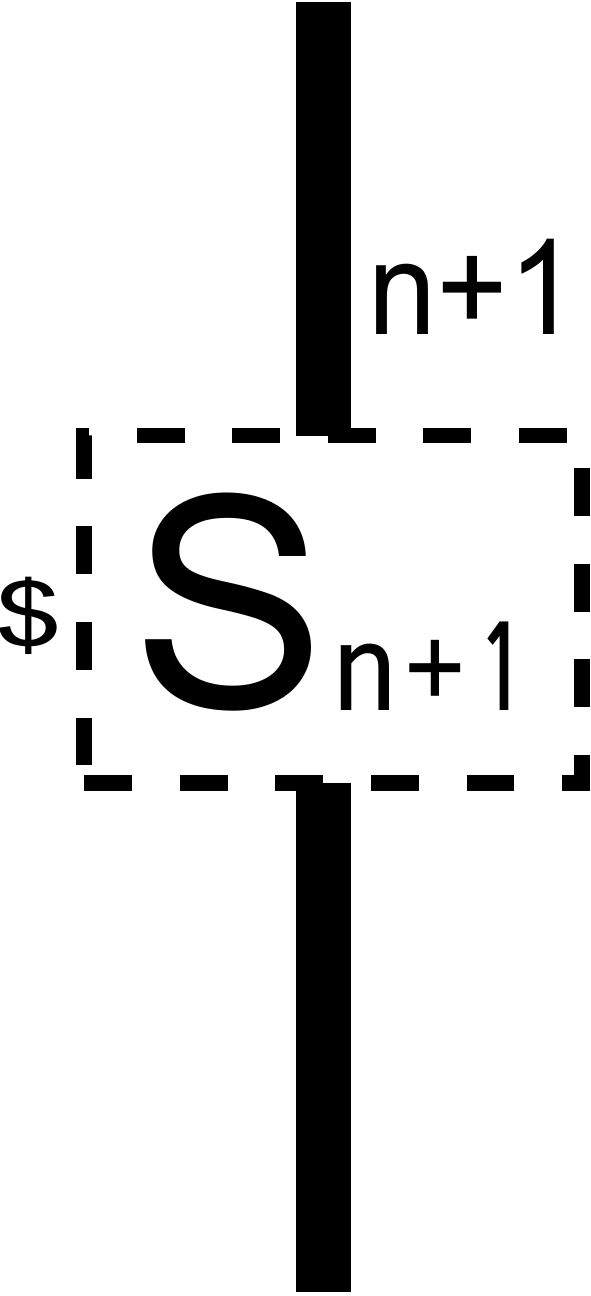}.$$
It tells how a minimal projection is decomposed after adding one string to the right.

In general, suppose $P$ is a minimal projection in $\mathscr{P}_{n,\pm}/\mathscr{I}_{n,\pm}$.
Note that $\mathscr{P}_{n+1,\pm}=\mathscr{I}_{n+1,\pm} \oplus \mathscr{P}_{n+1,\pm}/\mathscr{I}_{n+1,\pm}$. When $P$ is included in $\mathscr{P}_{n+1,\pm}$, it is decomposed as a sum of two projections $P=P_{old}+P_{new}$, such that
$P_{old}\in\mathscr{I}_{n+1,\pm}$ and $P_{new}\in\mathscr{P}_{n+1,\pm}/\mathscr{I}_{n+1,\pm}$. By the definition of $s_{n+1}$, we have $P_{new}=s_{n+1}P$. Now let us construct $P_{old}$.
Let $v$ be the depth $n$ point in the principal graph corresponding to $P$, $V$ be the central support of $P$.
Suppose $v_i$, $1\leq i \leq m$, are the depth $(n-1)$ points adjacent to $v$, the multiplicity of the edge between $v_i$ and $v$ is $m(i)$,
and $Q_i$ is a minimal projection in $\mathscr{P}_{n-1,\pm}$ corresponding to $v_i$.
For each $i$, take partial isometries $\{U_{ij}\}_{j=1}^{m(i)}$ in $\mathscr{P}_{n,\pm}$, such that $U_{ij}^*U_{ij}=P$, $\forall 1\leq j\leq m(i)$, and $\sum_{j=1}^{m(i)}U_{ij}U_{ij}^*=Q_iV$.
It is easy to check that $\frac{tr_{n-1}(Q_i)}{tr_n(P)}\grb{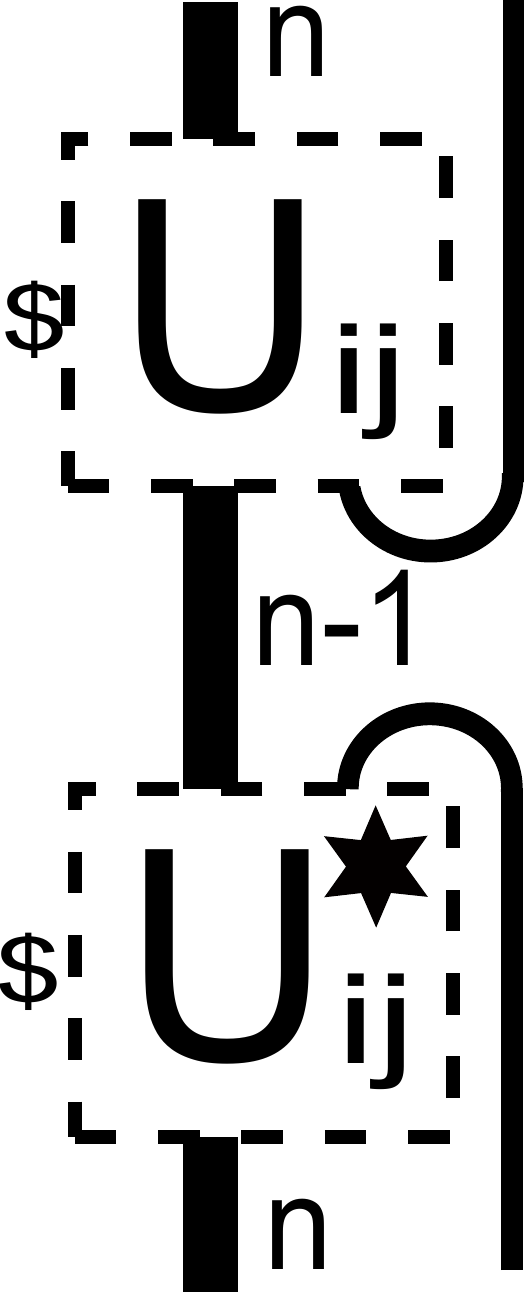}$ is a subprojection of $P$, and they are mutually orthogonal for all $i,j$.
By Frobenius reciprocity, their sum is $P_{old}$.
Then the general Wenzl's formula is
\begin{equation}\label{wenzl}
\grb{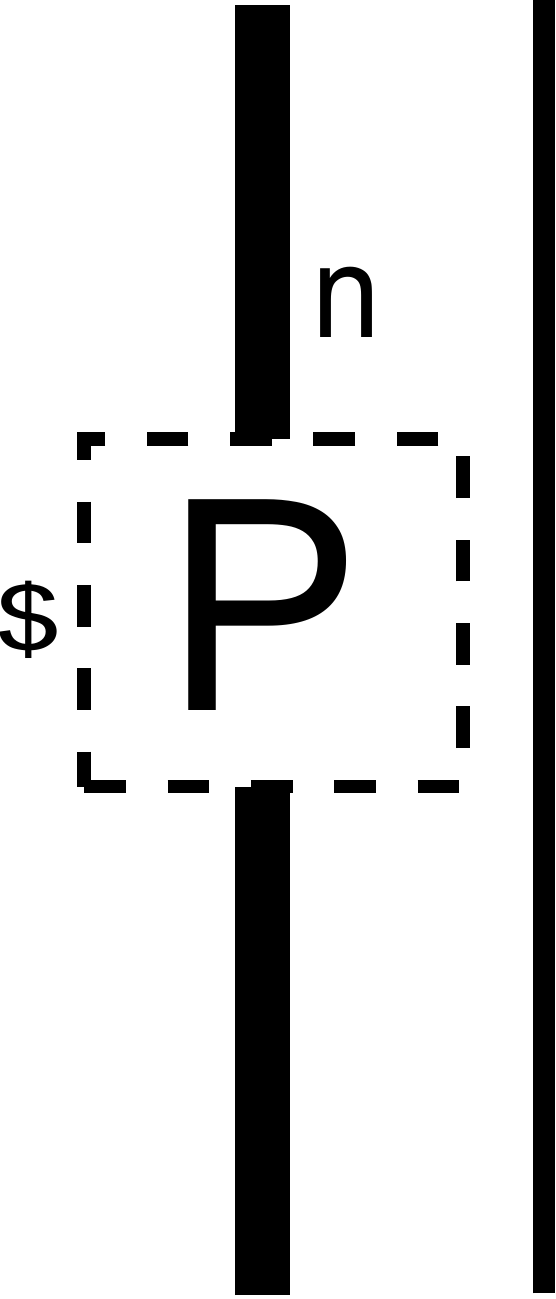}=\sum_{i=1}^m\sum_{j=1}^{n(i)} \frac{tr_{n-1}(Q_i)}{tr_n(P)}\grb{wen5}+s_{n+1}\grb{wen4}.
\end{equation}
Adding a cap to the left, we have the local relation
\begin{equation}\label{liu}
\grb{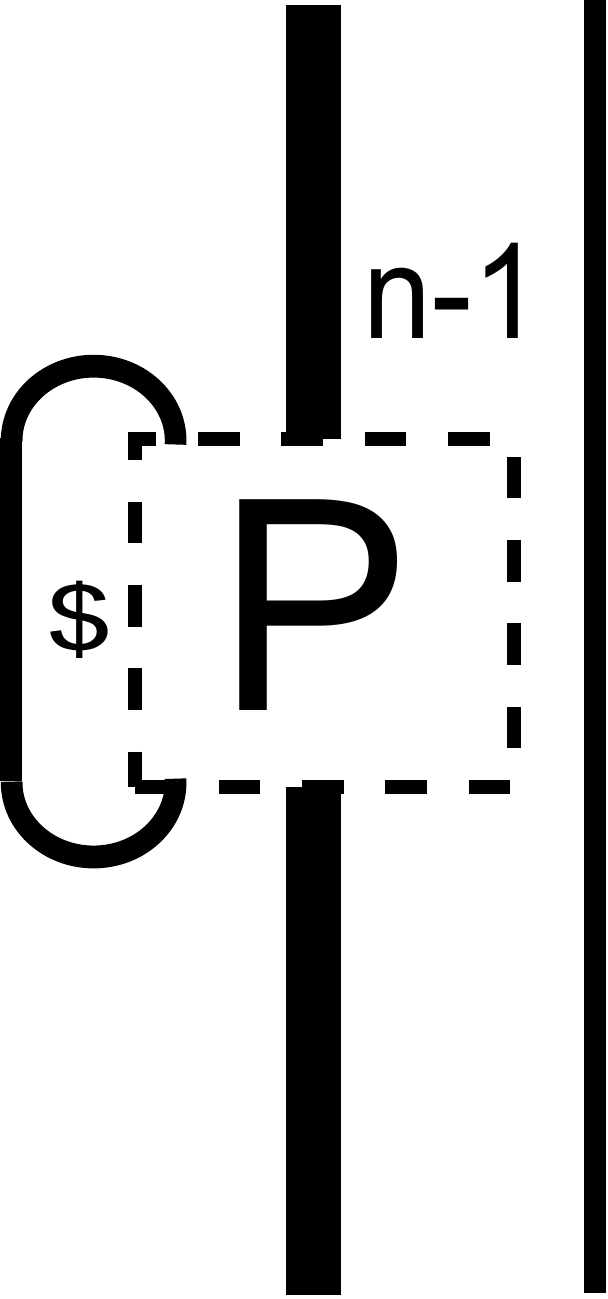}=\sum_{i=1}^m\sum_{j=1}^{n(i)} \frac{tr_{n-1}(Q_i)}{tr_n(P)}\grb{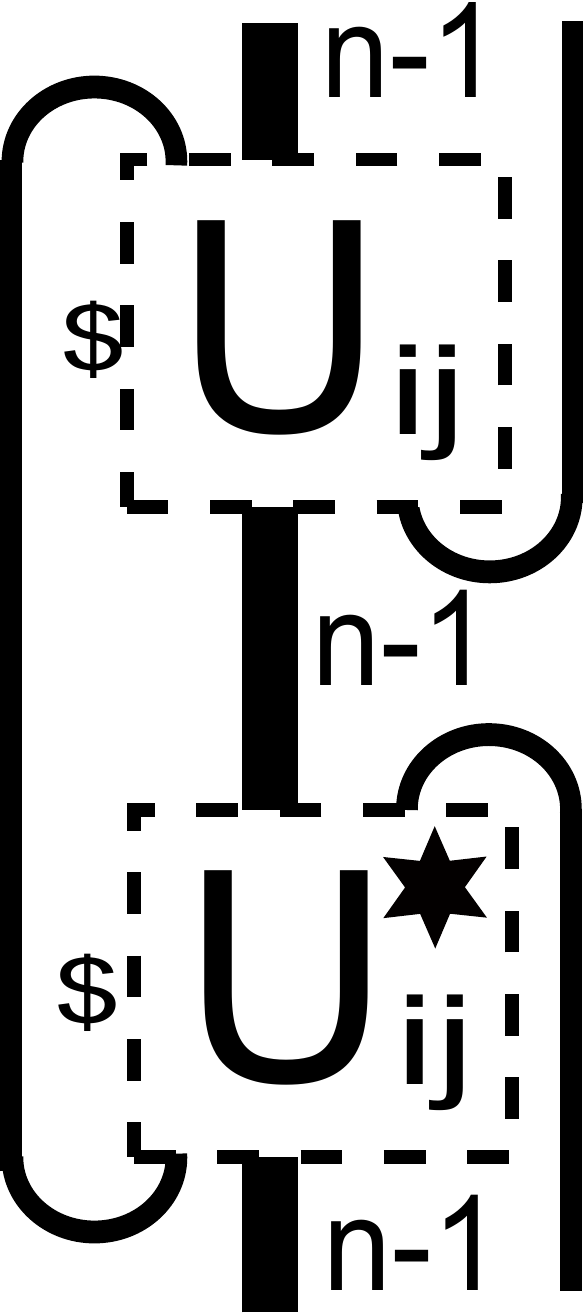}+\grb{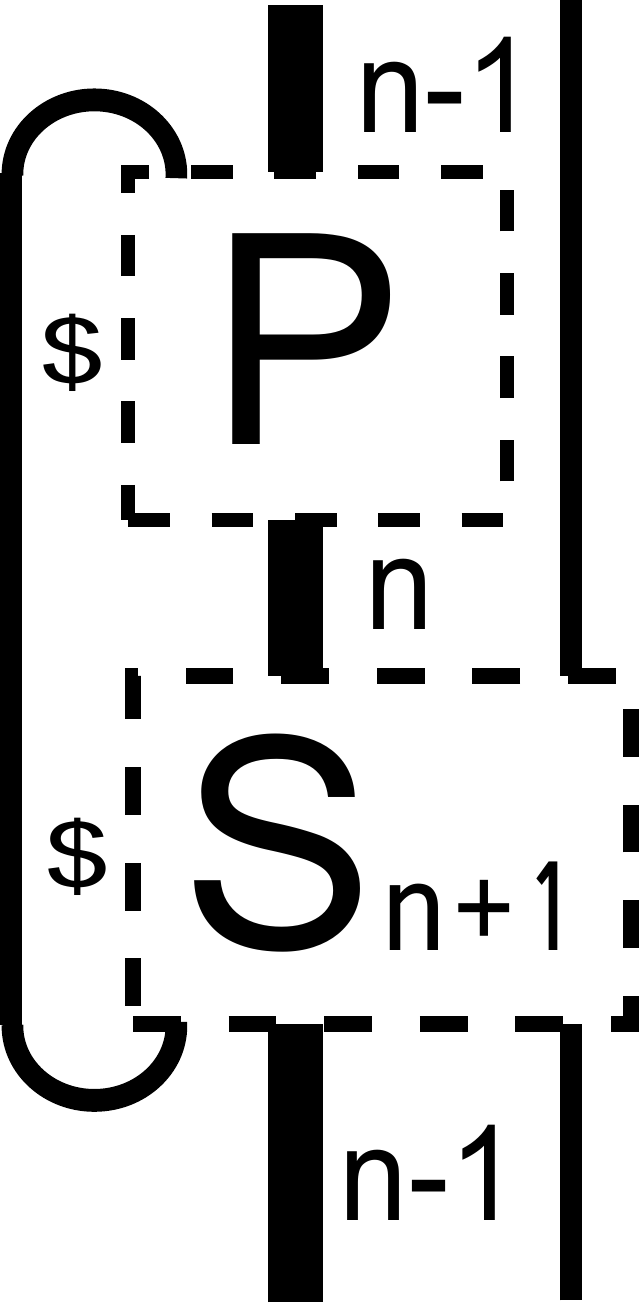}.
\end{equation}

\subsection{The rank-one decomposition}
We would like to introduce a decomposition of an element in a finite dimensional C$^*$ algebras. Suppose $\mathcal{A}$ is a finite dimensional C$^*$ algebra and $x$ is an element in $\mathcal{A}$. Let $x=v|x|$ be the polar decomposition of $x$, where $|x|=(x^*x)^{1/2}$, $v$ is the unique partial isometry from the range of $x^*$ onto the range of $x$. Let $|x|=\sum_j\lambda_j p_j$ be the decomposition of $|x|$ such that $\lambda_j\geq 0$ and $p_j$ are minimal projections in $\mathcal{A}$. Then $x$ can be uniquely written as $\sum_j\lambda_jvp_j=\sum_j\lambda_jv_j$, where $v_j=vp_j$ is a rank-one partial isometry. We will say this unique decomposition of $x$ as the rank-one decomposition of $x$.

\section{Inequalities}\label{Inequ}
In this section, $\mathscr{P}$ is an irreducible subfactor planar algebra.
We will prove the Hausdorff-Young inequality and Young's inequality for $\mathscr{P}$.

\begin{definition}
For $x$ in $\mathscr{P}_{2,\pm}$, we define its $p$-norm, $1\leq p< \infty$, as
$$\|x\|_p=tr_2((x^*x)^{p/2})^{1/p}=tr_2(|x|^p)^{1/p};$$
$$\|x\|_\infty=\|x\|,$$
the operator norm of $x$.
\end{definition}

Note that $\mathscr{P}_{2,\pm}$ are finite dimensional $C^*$ algebras, so all the topologies induced by $\|\cdot\|_p$, $1\leq p\leq \infty$, are identical to its norm topology.
\begin{remark}
The Fourier transform $\mathcal{F}$ and $\mathcal{F}^{-1}$ are continuous.
\end{remark}

\begin{proposition}
For any $x$ in $\mathscr{P}_{2,\pm}$ and $1\leq p\leq \infty$, we have
$$\|x\|_p=\||x|\|_p=\|x^*\|_p=\|\overline{x}\|_p.$$
\end{proposition}
\begin{proof}
We omit its proof here.
\end{proof}

\begin{proposition}[H\"{o}lder's Inequality]\label{holder}
For any $x$, $y$, $z$ in $\mathscr{P}_{2,\pm}$, we have
\begin{itemize}
\item[(1)] $|tr_2(xy)|\leq \|x\|_p\|y\|_q$, where $1\leq p\leq \infty$, $\frac{1}{p}+\frac{1}{q}=1$.
\item[(2)] $|tr_2(xyz)|\leq \|x\|_p\|y\|_q\|z\|_r$, where $1\leq p,q\leq \infty$, $\frac{1}{p}+\frac{1}{q}+\frac{1}{r}=1$.
\item[(3)] $\|xy\|_r\leq \|x\|_p\|y\|_q$, where $1\leq p,q,r \leq \infty$, $\frac{1}{r}=\frac{1}{p}+\frac{1}{q}$.
\end{itemize}
\end{proposition}
We will use the first inequality frequently.
\begin{proof}
The proof can be found in \cite{Xu07}.
\end{proof}

\begin{proposition}\label{norm}
For any $x$ in $\mathscr{P}_{2,\pm}$, we have
$$\|x\|_p=\sup\{|\tau(xy)|:y\in\mathscr{P}_{2,\pm},\|y\|_q\leq 1\},$$
where $1\leq p<\infty$, $\frac{1}{p}+\frac{1}{q}=1$.

\end{proposition}
\begin{proof}
The proof can be found in \cite{Xu07}.
\end{proof}

\begin{proposition}\label{holdereq}
For any $x$, $y$ in $\mathscr{P}_{2,\pm}$, $\frac{1}{p}+\frac{1}{q}=1$, $p,q\geq 1$, we have
$$|tr_2(x^*y)|=\|x\|_p\|y\|_q$$
if and only if
$$x=u|x|,y=\lambda u|y|,\quad \frac{|x|^p}{\|x\|_p^p}=\frac{|y|^q}{\|y\|_q^q}$$
for some unitary element $u$ and some complex number $\lambda$ with $|\lambda|=1$. When $p=\infty$, $\frac{|x|^p}{\|x\|_p^p}$ is defined to be the spectral projection of $|x|$ corresponding to its maximal spectrum (which is $\|x\|_\infty$).
\end{proposition}
\begin{proof}
This follows from the proof of Proposition 1.9 in \cite{Xu07} and the conditions for the equalities to hold in H\"{o}lder's inequality and the Cauchy-Schwartz inequality.
\end{proof}

\begin{proposition}[Interpolation Theorem]\label{Inter}
Let $\mathcal{M}$ be a finite von Neumann algebra with a faithful normal tracial state $\tau$. Suppose $T:\mathcal{M}\to\mathcal{M}$ is a linear map. If
$$\|Tx\|_{p_1}\leq K_1\|x\|_{q_1} \mbox{ and }\|Tx\|_{q_1}\leq K_2\|x\|_{q_2},$$
then $$\|Tx\|_{p_\theta}\leq K_1^{1-\theta}K_2^{\theta}\|x\|_{q_\theta},$$
where $\frac{1}{p_\theta}=\frac{1-\theta}{p_1}+\frac{\theta}{p_2}$, $\frac{1}{q_\theta}=\frac{1-\theta}{q_1}+\frac{\theta}{q_2}$, $0\leq \theta\leq 1$.

\end{proposition}
\begin{proof}
This is a special case of the interpolation theorem in \cite{Ko84}.
\end{proof}

\begin{proposition}\label{tr1}
For any $x$ in $\mathscr{P}_{2,\pm}$, we have
$$\|\mathcal{F}(x)\|_\infty\leq \frac{tr_2(|x|)}{\delta}=\frac{\|x\|_1}{\delta}.$$
\end{proposition}
\begin{proof}
Note that $\mathscr{P}_{2,\pm}=\mathscr{I}_{2,\pm}\oplus\mathscr{P}_{2,\pm}/\mathscr{I}_{2,\pm}=\mathbb{C}e_1\oplus \mathscr{P}_{2,\pm}/\mathscr{I}_{2,\pm}$.
If $x=e_1$, $\mathcal{F}(x)=1$ and $tr_2(e_1)=\delta$, then $\|\mathcal{F}(x)\|_\infty=1=\frac{tr_2(|x|)}{\delta}$

If $x$ is a rank-one partial isometry $v$ in $\mathscr{P}_{2,\pm}/\mathscr{I}_{2,\pm}$,
then by Wenzl's formula (\ref{wenzl}), we have
$$\frac{\delta}{\|v\|_1} \grb{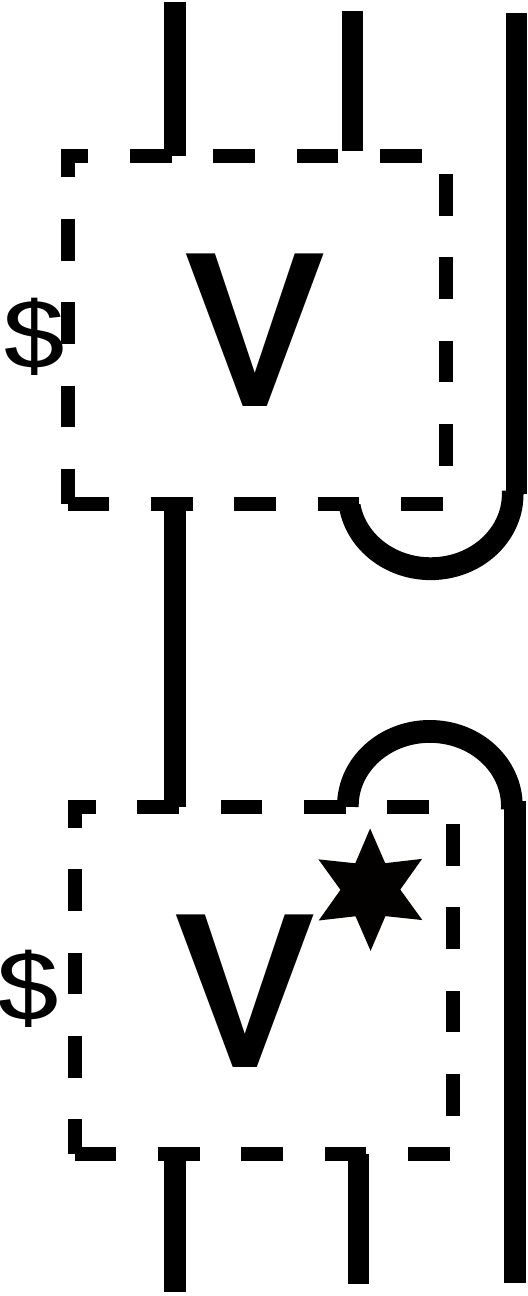} \leq \grb{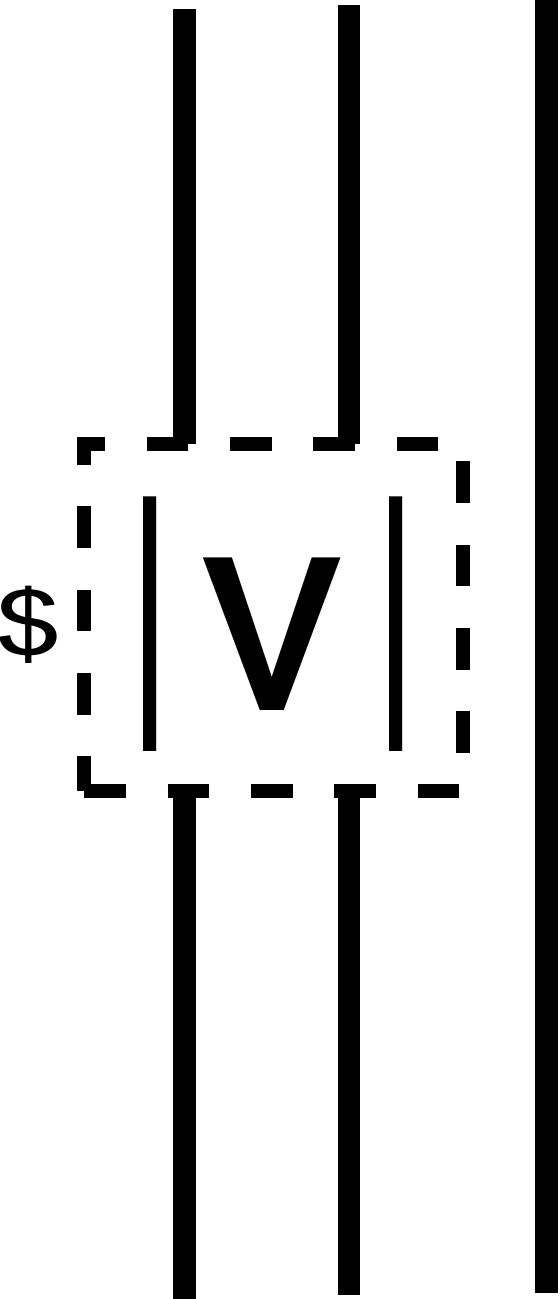},$$
as the left side is a subprojection of the right side.
By the local relation (\ref{liu}), we have
$$\frac{\delta}{\|v\|_1} \grb{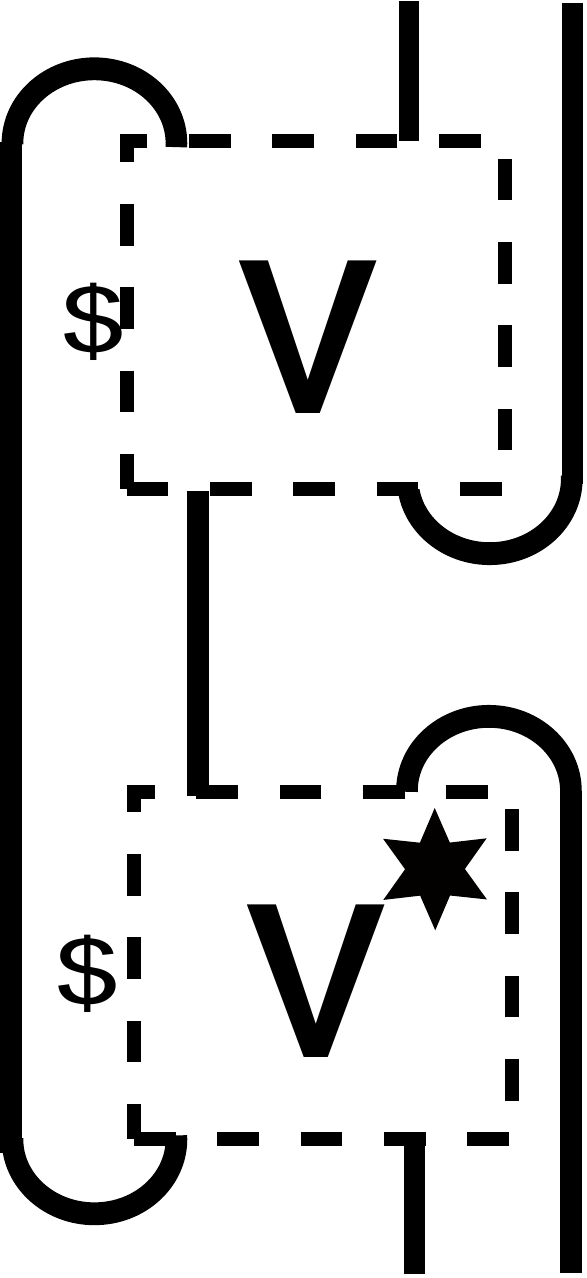} \leq \grb{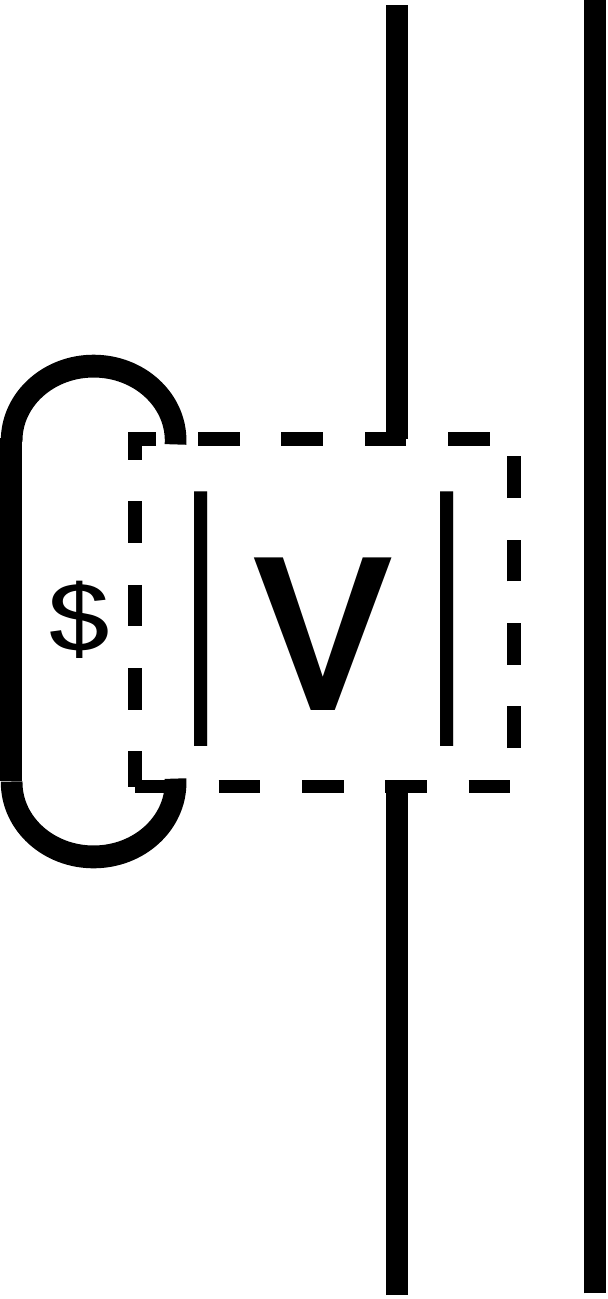}.$$
This is to say that
$$(\mathcal{F}^{-1}(v))^*(\mathcal{F}^{-1}(v))\leq (\frac{\|v\|_1}{\delta})^2 1.$$
Taking norm, we have
$$\|(\mathcal{F}^{-1}(v))^*(\mathcal{F}^{-1}(v))\|_\infty\leq (\frac{\|v\|_1}{\delta})^2.$$
Hence
$$\|\mathcal{F}^{-1}(v)\|_\infty\leq \frac{\|v\|_1}{\delta}.$$
Recall that $\mathcal{F}(v)=(\mathcal{F}^{-1}(v^*))^*$ (see proposition \ref{fb}) and $\|v\|_1=\|v^*\|_1$, we obtain that
$$\|\mathcal{F}(v)\|_\infty\leq \frac{\|v\|_1}{\delta}.$$

For an arbitrary $x$ in $\mathscr{P}_{2,\pm}$, let $x=\sum_k\lambda_k v_k$ be the rank-one decomposition. Then $\|x\|_1=\sum_k\lambda_k \|v_k\|_1$.
We have proved that $\|\mathcal{F}(v_k)\|_\infty\leq \frac{\|v_k\|_1}{\delta}$ for rank-one partial isometry $v_k$, so
$$\|\mathcal{F}(x)\|_\infty\leq \sum_k \lambda_k\|\mathcal{F}(v_k)\|_\infty \leq \sum_k \lambda_k\frac{\|v_k\|_1}{\delta}=\frac{\|x\|_1}{\delta}.$$

\end{proof}

\begin{remark}
When $x$ is positive, we see that $\|\mathcal{F}(x)\|_\infty=\frac{\|x\|_1}{\delta}$ since $e_1\mathcal{F}(x)=\frac{tr_2(x)}{\delta}e_1$.
\end{remark}

\begin{theorem}\label{Young1}[Hausdorff-Young Inequality]
For any $x$ in $\mathscr{P}_{2,\pm}$,
$$\|\mathcal{F}(x)\|_{p}\leq \left(\frac{1}{\delta}\right)^{1-\frac{2}{p}}\|x\|_{q},$$
where $2\leq p\leq \infty$ and $\frac{1}{p}+\frac{1}{q}=1$.
\end{theorem}
\begin{proof}
By Proposition \ref{tr1} and the sphericality of subfactor planar algebras, we have
$$\|\mathcal{F}(x)\|_\infty\leq \frac{\|x\|_1}{\delta},\text{ and }\|\mathcal{F}(x)\|_2=\|x\|_2.$$
By Proposition \ref{Inter}, we have
$$\|\mathcal{F}(x)\|_p\leq \left(\frac{1}{\delta}\right)^{1-\frac{2}{p}}\|x\|_q,$$
where $2\leq p\leq \infty$ and $\frac{1}{p}+\frac{1}{q}=1$.
\end{proof}

Now we are going to prove Young's inequality.

\begin{lemma}\label{con1}
For any $x$, $y$ in $\mathscr{P}_{2,\pm}$, we have
$$\|x*y\|_\infty\leq \frac{\|x\|_\infty\|y\|_1}{\delta},\quad \|y*x\|_\infty\leq \frac{\|x\|_\infty\|y\|_1}{\delta}.$$
\end{lemma}

\begin{proof}
If $y$ is a rank-one partial isometry $v$ in $\mathscr{P}_{2,\pm}/\mathscr{I}_{2,\pm}$,
we have
$$\frac{\delta}{\|v\|_1} \grb{inev1} \leq \grb{inev2}.$$
Then
$$\frac{\delta}{\|v\|_1} \grc{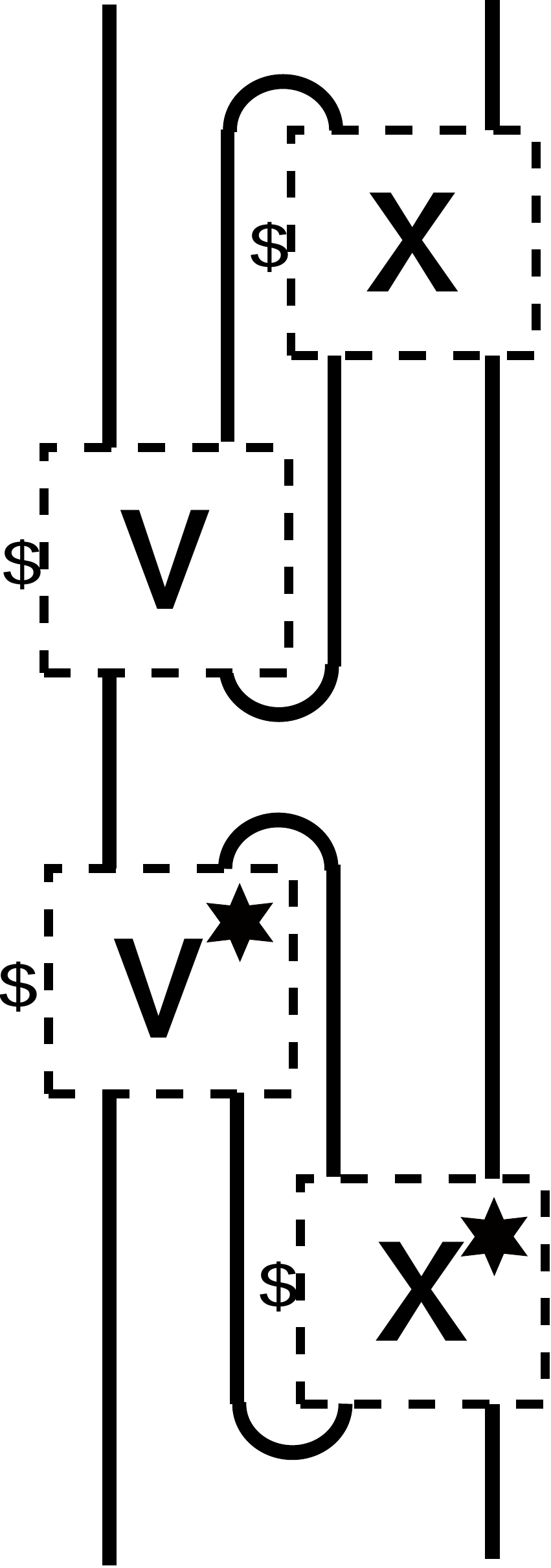} \leq \grc{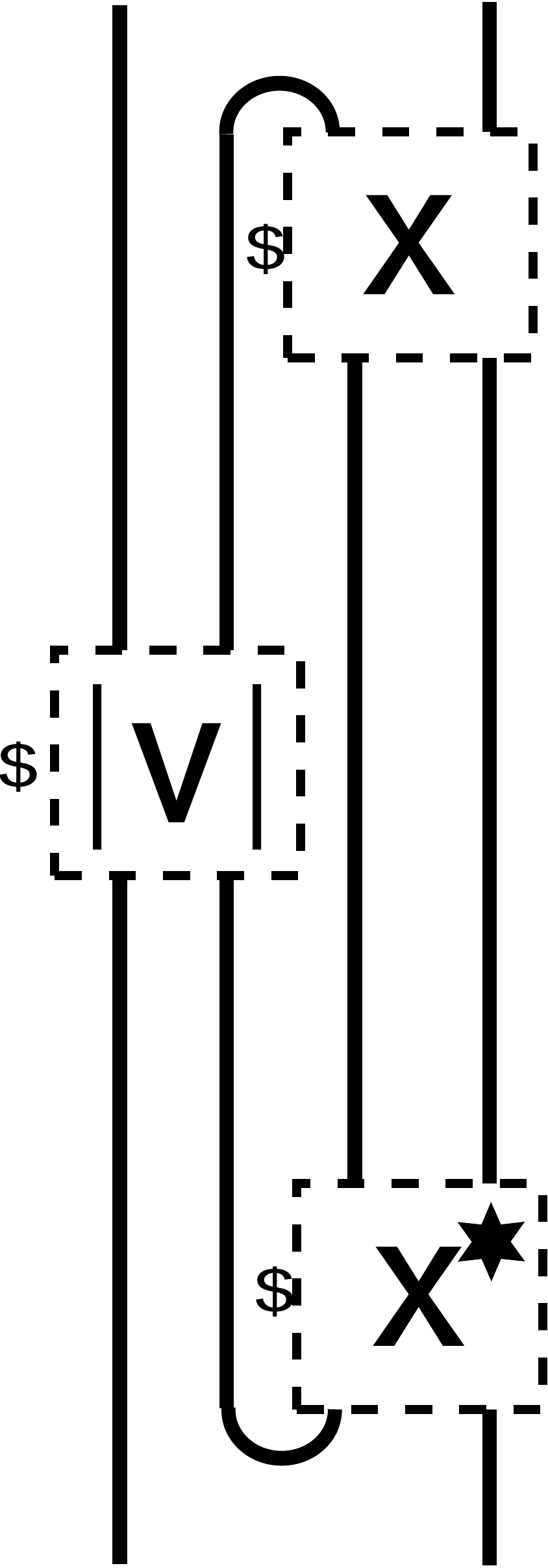} \leq \gra{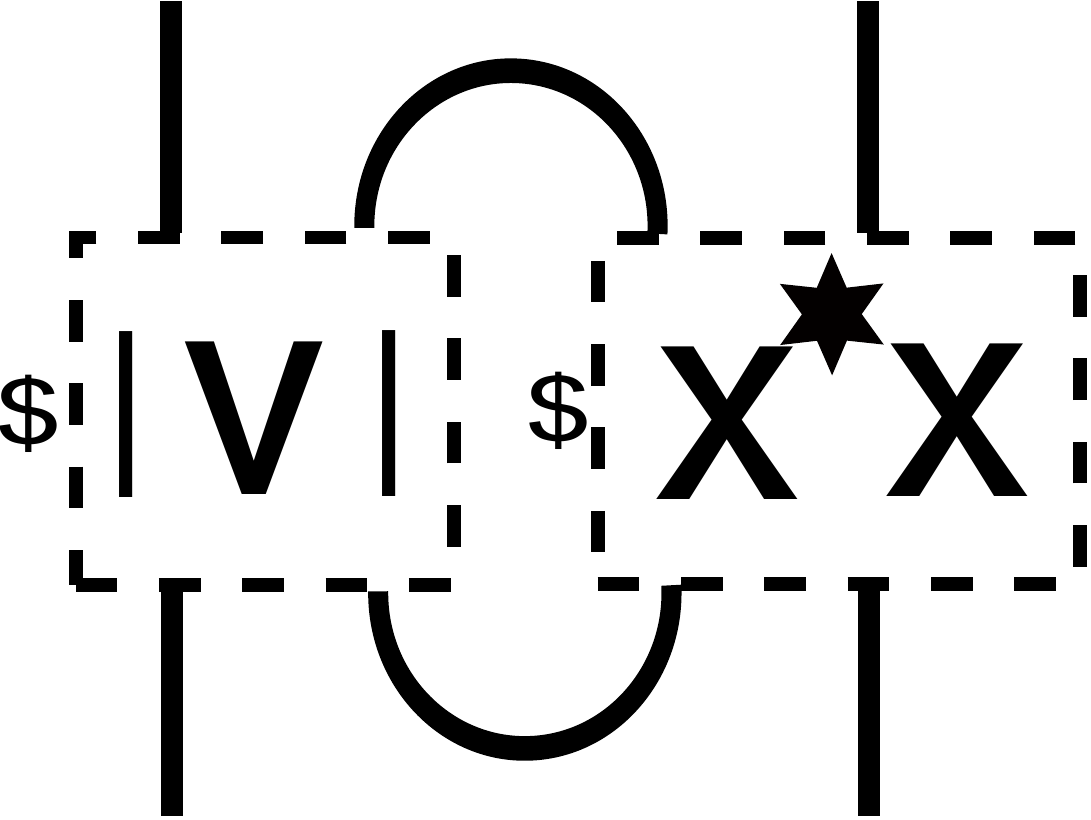} \leq \|x\|_\infty^2 \gra{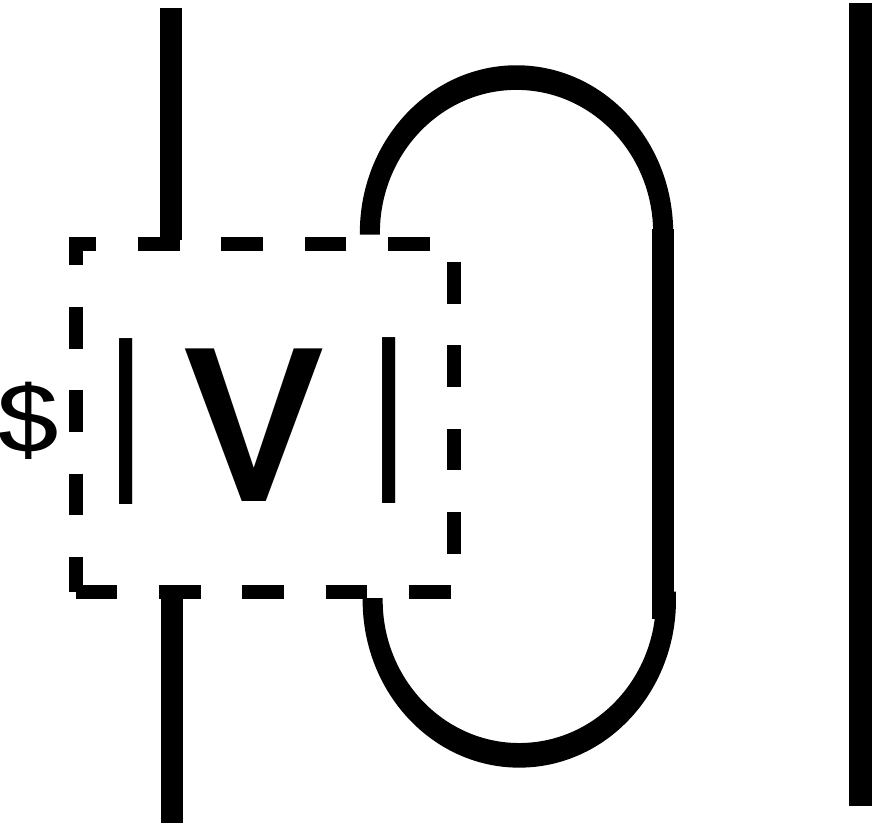} = \|x\|_\infty^2 \frac{\|v\|_1}{\delta} 1,$$
the last inequality follows from Schur Product Theorem (Proposition \ref{schur}).
Hence
$$\|(v*x)^*(v*x)\|_\infty\leq (\|x\|_\infty \frac{\|v\|_1}{\delta})^2,$$ and
$$\|v*x\|_\infty\leq \|x\|_\infty \frac{\|v\|_1}{\delta}.$$

For an arbitrary  $y$ in $\mathscr{P}_{2,\pm}$, let $y=\sum_k\lambda_k v_k$ be the rank-one decomposition. Then $\|y\|_1=\sum_k\lambda_k \|v_k\|_1$.
We have proved that $\|v_k*x\|_\infty\leq \|x\|_\infty \frac{\|v_k\|_1}{\delta}$, so
$$\|y*x\|_\infty\leq \sum_k \lambda_k\|v_k*x\|_\infty \leq \sum_k \lambda_k \|x\|_\infty \frac{\|v_k\|_1}{\delta}= \|x\|_\infty \frac{\|y\|_1}{\delta}.$$

Note that $\|x*y\|_\infty=\|\overline{y*x}\|_\infty$, $\|x\|_\infty=\|\overline{x}\|_\infty$, and $\|y\|_1=\|\overline{y}\|_1$, we obtain that
$$\|x*y\|_\infty\leq \frac{\|x\|_\infty\|y\|_1}{\delta}.$$

\end{proof}

\begin{lemma}\label{con2}
For any $x$, $y$ in $\mathscr{P}_{2,\pm}$, we have
$$\|x*y\|_1\leq \frac{\|x\|_1\|y\|_1}{\delta}.$$
\end{lemma}

\begin{proof}
For any $x$, $y$ in $\mathscr{P}_{2,\pm}$, by Proposition \ref{norm} and Lemma \ref{change}, we have

\begin{eqnarray*}
\|x*y\|_1&=& \sup_{\|z\|_\infty=1}|tr_2((x*y)z)| \\
&=&  \sup_{\|z\|_\infty=1}|tr_2(x(z*\overline{y}))| \\
&\leq&  \|x\|_1 \|z*\overline{y}\|_\infty  \quad \quad\quad\quad\text{H\"{o}lder's inequality}\\
&\leq&  \|x\|_1 \frac{\|y\|_1}{\delta} \quad \quad\quad\quad\quad\quad\quad \text{Lemma \ref{con1}}\\
\end{eqnarray*}

\end{proof}

\begin{lemma}\label{con3}
For any $x$, $y$ in $\mathscr{P}_{2,\pm}$, we have
$$\|x*y\|_p\leq \frac{\|x\|_p\|y\|_1}{\delta},\quad
\|y*x\|_p\leq \frac{\|x\|_p\|y\|_1}{\delta},$$
where $1\leq p\leq \infty$.
\end{lemma}
\begin{proof}
It follows from Lemma \ref{con1}, \ref{con2} and Proposition \ref{Inter}.
\end{proof}

\begin{lemma}\label{con4}
For any $x,y$ in $\mathscr{P}_{2,\pm}$, we have
$$\|x*y\|_\infty\leq \frac{\|x\|_p\|y\|_q}{\delta},$$
where $1\leq p \leq \infty$ and $\frac{1}{p}+\frac{1}{q}=1$.
\end{lemma}
\begin{proof}
Suppose that $x*y=\sum_{k}\lambda_kv_k$ is the rank-one decomposition. Then
$$\|x*y\|_\infty=\sup_k\lambda_k=\sup_k\frac{tr_2((x*y)v_k^*)}{tr_2(|v_k|)}.$$
By Lemma \ref{change} and Proposition \ref{holder}, we have that
$$|tr_2((x*y)v_k^*)|=|tr_2((\overline{v_k^*}*a)\overline{y})|\leq \|\overline{v_k^*}*x\|_p\|\overline{y}\|_q$$
By Lemma \ref{con3}, we have
$$\|\overline{v_k^*}*x\|_p\leq \|x\|_p\frac{tr_2(|\overline{v_k^*}|)}{\delta}.$$
Moreover, since $\|b\|_q=\|\overline{b}\|_q$, we obtain
$$\|x*y\|_\infty\leq \frac{\|x\|_ptr_2(|\overline{v_k^*}|)\|y\|_q}{\delta tr_2(|v_k^*|)}=\frac{\|x\|_p\|y\|_q}{\delta}.$$
This completes the proof of the lemma.
\end{proof}

\begin{theorem}[Young's Inequality]\label{Young2}
For any $x,y$ in $\mathscr{P}_{2,\pm}$, we have
$$\|x*y\|_r\leq \frac{\|x\|_p\|y\|_q}{\delta},$$
where $1\leq p,q,r\leq \infty$, $\frac{1}{p}+\frac{1}{q}=\frac{1}{r}+1$.
\end{theorem}
\begin{proof}
It follows from Lemma \ref{con3}, \ref{con4} and Proposition \ref{Inter}.
\end{proof}

\begin{remark}
The inequalities listed as above are sharp. This can be easily checked, if one lets $x=y=1$.
\end{remark}

\section{Noncommutative Uncertainty Principles}\label{SectionNUP}
In this section, we will prove the Donoho-Stark uncertainty principle and the Hischman-Beckner uncertainty principle for irreducible subfactor planar algebras.
It was shown \cite{DoSt,Sm90} that for a finite abelian group $G$ and a nonzero function $f$ on $G$,
$$|\text{supp}(f)||\text{supp}(\hat{f})|\geq |G|,$$
where $\hat{f}$ is the Fourier transform of $f$, $\text{supp}(f)=\{x\in G: f(x)\neq 0\}$. We will generalize this uncertainty principle in subfactor planar algebras.

\begin{notation}
Suppose $\mathscr{P}$ be a subfactor planar algebra. For any $x$ in $\mathscr{P}_{n,\pm}$, recall that $\mathcal{R}(x)$ is the range projection of $x$. Let us define $\mathcal{S}(x)=tr_n(\mathcal{R}(x))$.
\end{notation}

\begin{theorem}\label{UnP}
Suppose $\mathscr{P}$ is an irreducible subfactor planar algebra. Then for any nonzero $x$ in $\mathscr{P}_{2,\pm}$, we have
$$\mathcal{S}(x)\mathcal{S}(\mathcal{F}(x))\geq \delta^2.$$
\end{theorem}

This partially proves Main Theorem \ref{mainthm1}.

\begin{proof}
Let $\mathcal{F}(x)=\sum_j\lambda_jv_j$ be the rank-one decomposition of $\mathcal{F}(x)$. Then $\|\mathcal{F}(x)\|_\infty=\sup_j\lambda_j$. Now we have
\begin{eqnarray*}
\sup_j\lambda_j&=&\|\mathcal{F}(x)\|_\infty\\
&\leq &\frac{\|x\|_1}{\delta}\leq \frac{\|\mathcal{R}(x)\|_2\|x\|_2}{\delta}\\
&=&\frac{\mathcal{S}(x)^{1/2}}{\delta}\|\mathcal{F}(x)\|_2\\
&=&\frac{\mathcal{S}(x)^{1/2}}{\delta}(\sum_j\lambda_j^2\|v_j\|_1)^{1/2}\\
&\leq &\frac{\mathcal{S}(x)^{1/2}}{\delta}(\sup_j\lambda_j^2)^{1/2}\|\mathcal{R}(\mathcal{F}(x))\|_1^{1/2}\\
&=&\frac{\mathcal{S}(x)^{1/2}\mathcal{S}(\mathcal{F}(x))^{1/2}}{\delta}\sup_j\lambda_j.
\end{eqnarray*}
Thus
$$\mathcal{S}(x)\mathcal{S}(\mathcal{F}(x))\geq \delta^2$$
\end{proof}

If we apply the Theorem \ref{UnP} to group subfactor planar algebras, we have
\begin{corollary}
For any finite abelian group $G$ and a nonzero function $f$ on $G$, we have
$$|\text{supp}(f)||\text{supp}(\hat{f})|\geq |G|,$$
where $\hat{f}$ is the Fourier transform of $f$.
\end{corollary}

In \cite{Tao05}, T. Tao shows that for the group $\mathbb{Z}_p$, $p$ prime, and a nonzero function $f$ on it,
$$|\text{supp}(f)|+|\text{supp}(\hat{f})|\geq p+1.$$
Similarly we can ask if a parallel inequality
\begin{equation}\label{plus}
\mathcal{S}(x)+\mathcal{S}(\mathcal{F}(x))\geq \delta^2+1
\end{equation}
holds in the 2-box space of subfactor planar algebras.
The statement is false if a subfactor planar algebra has a non-trivial biprojection.
We need a ``prime" condition.
It is proved by Bisch in \cite{Bis94e} that a finite depth subfactor planar algebra with index $p$, $p$ prime, has no non-trivial biprojection. We may ask the following question.

\begin{question}
Suppose $\mathscr{P}$ is an irreducible finite depth subfactor planar algebra with index $p$, $p$ prime. Is inequality (\ref{plus}) true?
\end{question}

\begin{remark}
If $\mathscr{P}$ is a group subgroup subfactor planar algebra with index $p$, then the argument reduces to the group case, where it is true.
\end{remark}

\begin{remark}
If $\mathcal{S}(\mathcal{F}(x))$ is replaced by the trace of the central support of $\mathcal{F}(x)$, then the corresponding weaker inequality for finite groups was proved in Corollary 5.1 in \cite{GGI}.
\end{remark}

For Hirschman's uncertainty principle \cite{Hirsch}, the Shannon entropy is used to describe his uncertainty principle in abelian groups. In subfactors, we will use the von Neumann entropy in place of the Shannon entropy.

\begin{notation}
For any $x$ in $\mathscr{P}_{n,\pm}$, the von Neumann entropy of $|x|^2$ is
$$H(|x|^2)=-tr_n(|x|^2\log|x|^2)=-tr_n(x^*x\log x^*x).$$
\end{notation}

Von Neumann entropy is widely used as an important measurement in quantum information theory and it is also important for our next uncertainty principle.

\begin{theorem}\label{entropybase}
Suppose that $\mathscr{P}$ is an irreducible subfactor planar algebra and $x\in\mathscr{P}_{2,\pm}$. Then
$$H(|x|^2)+H(|\mathcal{F}(x)|^2)\geq \|x\|_2(2\log\delta-4\log\|x\|_2).$$
Specifically
$$H(|x|^2)+H(|\mathcal{F}(x)|^2)\geq 2\log\delta$$
whenever $x\in\mathscr{P}_{2,\pm}$ and $\|x\|_2=1$.
\end{theorem}

This together with Theorem \ref{UnP} proves Main Theorem \ref{mainthm1}.

\begin{proof}
Without loss of generality, we assume that $x\neq 0$.
By Theorem \ref{Young1},
$$\|\mathcal{F}(x)\|_{p}\leq \left(\frac{1}{\delta}\right)^{1-\frac{2}{p}}\|x\|_{q},$$
where $2\leq p\leq \infty$ and $\frac{1}{p}+\frac{1}{q}=1$.
Take
$$f(p)=\log\|\mathcal{F}(x)\|_p-\log\|x\|_q-\log(\frac{1}{\delta})^{1-\frac{2}{p}}.$$
Then $f(p)\leq 0$.
Note that $\|\mathcal{F}(x)\|_2=\|x\|_2$, so $f(2)=0$.
Then $f'(2)\leq 0$.

Note that
$$\frac{\text{d}}{\text{d}p}\sum_k x_k^p\Big|_{p=2}=\sum_k x_k^2\log x_k,$$ so
$$\frac{\text{d}}{\text{d}p}\|\mathcal{F}(x)\|_p^p\Big|_{p=2}=-\frac{1}{2}H(\mathcal{F}(x)).$$
Then
$$\frac{\text{d}}{\text{d}p} \log\|\mathcal{F}(x)\|_p\Big|_{p=2}=\frac{\text{d}}{\text{d}p} \frac{1}{p}\log\|\mathcal{F}(x)\|_p^p\Big|_{p=2}
=-\frac{1}{4}\log\|\mathcal{F}(x)\|_2^2-\frac{1}{4}\frac{H(\mathcal{F}(x))}{\|\mathcal{F}(x)\|_2^2}.$$
Similarly
$$\frac{\text{d}}{\text{d}p} \log\|x\|_q\Big|_{p=2}=
\left(\frac{\text{d}}{\text{d}q} \log\|x\|_q\Big|_{q=2}\right)
\left(\frac{\text{d}q}{\text{d}p}\Big|_{p=2}\right)
=\frac{1}{4}\log\|x\|_2^2+\frac{1}{4}\frac{H(x)}{\|x\|_2^2}.$$
Moreover,
$$\frac{\text{d}}{\text{d}p}\log(\frac{1}{\delta})^{1-\frac{2}{p}}\Big|_{p=2}
=-\frac{1}{2}\log \delta.$$
Therefore
$$f'(2)=
\left(-\frac{1}{4}\log\|\mathcal{F}(x)\|_2^2-\frac{1}{4}\frac{H(\mathcal{F}(x))}{\|\mathcal{F}(x)\|_2^2}\right)
-\left(\frac{1}{4}\log\|x\|_2^2+\frac{1}{4}\frac{H(x)}{\|x\|_2^2}\right)
+\frac{1}{2}\log \delta.$$

Recall that $f'(2)\leq 0$ and $\|\mathcal{F}(x)\|_2=\|x\|_2$,
so
$$H(|x|^2)+H(|\mathcal{F}(x)|^2)\geq \|x\|_2(2\log\delta-4\log\|x\|_2).$$
Specifically
$$H(|x|^2)+H(|\mathcal{F}(x)|^2)\geq 2\log\delta,$$
when $\|x\|_2=1$.

\end{proof}

\begin{remark}
This is a noncommutative version of the proof of Theorem 23 in \cite{DeCoTh}.
\end{remark}

\begin{corollary}\label{HtoD}
Suppose $\mathscr{P}$ is an irreducible subfactor planar algebra. Then for any nonzero $x$ in $\mathscr{P}_{2,\pm}$, we have
$$\mathcal{S}(x)\mathcal{S}(\mathcal{F}(x))\geq \delta^2.$$
\end{corollary}

\begin{proof}
We will prove the inequality $\log \mathcal{S}(x)\geq H(|x|^2)$ when $\|x\|_2=1$. With this inequality, it is easy to see that the corollary follows from Theorem \ref{entropybase}.


Let $x=\sum_j\lambda_jv_j$ be the rank-one decomposition. Since $\|x\|_2=1$, we have $\|x\|_2^2=\sum_j\lambda_j^2tr_2(|v_j|)=1$. Then
\begin{eqnarray*}
H(|x|^2)&=&-tr_2(|x|^2\log |x|^2)=-\sum_j\lambda_j^2\log\lambda_j^2tr_2(|v_j|)\\
&=&-tr_2(\mathcal{R}(|x|))\sum_j\frac{tr_2(|v_j|)}{tr_2(\mathcal{R}(|x|))}\lambda_j^2\log\lambda_j^2\\
&\leq &-tr_2(\mathcal{R}(|x|)) \left(\sum_j\frac{tr_2(|v_j|)}{tr_2(\mathcal{R}(|x|))}\lambda_j^2\right)\log \left(\sum_j\frac{tr_2(|v_j|)}{tr_2(\mathcal{R}(|x|))}\lambda_j^2\right)\quad \text{Jensen's inequality}\\
&=&\log tr_2(\mathcal{R}(|x|))=\log tr_2(\mathcal{R}(x))=\log\mathcal{S}(x).
\end{eqnarray*}
\end{proof}

\begin{remark}
The Hirschman-Beckner uncertainty principle is stronger than the Donoho-Stark uncertainty principle.
\end{remark}

\begin{remark}
If $\|x\|_2=\delta$, then $H(|x|^2)\leq 0$. Moreover, $H(|x|^2)=0$ if and only if $x$ is a unitary.
Therefore the maximaizer of $H(|x|^2)+H(|\mathcal{F}(x)|^2)$ is a biunitary if there is one.
\end{remark}

\section{Minimizers for Noncommutative Uncertainty Principles}\label{SectionMin}
Throughout this section, $\mathscr{P}$ is an irreducible subfactor planar algebra. We will discuss the minimizers of its two uncertainty principles shown in the last section. First, we would like to introduce some notions.

\begin{definition}\label{ebipartial}
An element $x$ in $\mathscr{P}_{2,\pm}$ is said to be extremal if $\|\mathcal{F}(x)\|_\infty=\displaystyle \frac{\|x\|_1}{\delta}$. We say a nonzero element $x$ is an (extremal) bi-partial isometry if $x$ and $\mathcal{F}(x)$ are multiplies of (extremal) partial isometries.
\end{definition}

For example, a positive operator is always extremal, since the norm of its fourier transform is achieved on the Jones projection; a biprojection is an extremal bi-partial isometry.

\begin{proposition}
If $x$ in $\mathscr{P}_{2,\pm}$ is extremal, then $x^*$ and $\overline{x}$ are extremal.
\end{proposition}
\begin{proof}
It follows from the facts that
$$\|\mathcal{F}(x^*)\|_\infty=\|\mathcal{F}^{-1}(x)^*\|_\infty=\|\mathcal{F}^{-1}(x)\|_\infty=\|\mathcal{F}(x)\|_\infty$$
and
$$\|x^*\|_1=\|x\|_1=\|\overline{x}\|_1.$$

\end{proof}

We will first show that the two uncertainty principles have the same minimizers which are extremal bi-partial isometries. The proof of the following theorem benefits a lot from \cite{OzPr}, and we also need Hopf's maximum principle. For readers' convenience, we state it here:
\begin{proposition}[Hopf's maximum principle,\cite{Hor}]\label{Hopf}
Let $D\subset\mathbb{C}$ be an open unit disc, and let $u:D\to \mathbb{R}$ be a harmonic function which extends to a continuous function on the closure $\overline{D}$ of $D$, $u:\overline{D}\to\mathbb{R}$. Suppose $z$ is a point on the boundary of $D$ such that $u(z)\geq u(z')$ for all $z'\in \overline{D}$, and the directional derivative of $u$ at $z$ along the radius which ends at $z$, is zero. Then $u(z)=u(z')$ for all $z'\in \overline{D}$.
\end{proposition}

\begin{theorem}\label{minmain1}
For a non-zero element $x$ in $\mathscr{P}_{2,\pm}$, the following statements are equivalent:
\begin{itemize}
\item[(1)] $H(|x|^2)+H(|\mathcal{F}(x)|^2)=\|x\|_2(2\log\delta-4\log\|x\|_2);$
\item[(2)] $\mathcal{S}(x)\mathcal{S}(\mathcal{F}(x))=\delta^2;$
\item[(3)] $x$ is an extremal bi-partial isometry.
\end{itemize}
\end{theorem}

This proves part of Main Theorem \ref{mainthm2}.

\begin{proof}
``(1)$\Rightarrow$ (3)". Suppose $x$ is in $\mathscr{P}_{2,\pm}$ such that $H(|x|^2)+H(|\mathcal{F}(x)|^2)=0$.
Then $\lambda x$ also satisfies the above equation.
Without loss of generality, we assume that $\|x\|_2=\sqrt{\delta}$.
We define a complex function
$$F(z)=tr_2\left(\mathcal{F}\left(\frac{x}{|x|}|x|^{2z}\right)\left|\mathcal{F}(x)\right|^{2z} \left(\frac{\mathcal{F}(x)}{|\mathcal{F}(x)|}\right)^*\right),$$
where $\frac{x}{|x|}$ means $x|x|^{-1}$ and $|x|^{-1}$ takes inverves on its support and zero elsewhere, i.e. if $x=w|x|$ by the polar decomposition, $w=x|x|^{-1}$ here.

Now we will show that $F(z)$ is analytic and bounded on the strip $\frac{1}{2}<\sigma <1$, $z=\sigma+it$. By Proposition \ref{holder} and \ref{Young1},
\begin{eqnarray*}
|F(\sigma+it)|&\leq &\|\mathcal{F}\left(\frac{x}{|x|}|x|^{2z}\right)\|_{\frac{1}{1-\sigma}}\||\mathcal{F}(x)|^{2z}\|_{\frac{1}{\sigma}}\\
&\leq &\left(\frac{1}{\delta}\right)^{1-2(1-\sigma)}\||x|^{2\sigma}\|_{1/\sigma}\||\mathcal{F}(x)|^{2\sigma}\|_{1/\sigma}\\
&=&\left(\frac{1}{\delta}\right)^{2\sigma-1}\delta^\sigma\delta^\sigma=\delta.
\end{eqnarray*}
This implies that $F(z)$ is bounded on $\frac{1}{2}<\sigma<1$. Differentiating the function with respect to $z$, we have
\begin{eqnarray*}
F'(z)&=&tr_2\left(\mathcal{F}\left(\frac{x}{|x|}|x|^{2z}(\log|x|^2)\right) \left|\mathcal{F}(x)\right|^{2z}\left(\frac{\mathcal{F}(x)}{|\mathcal{F}(x)|}\right)^*\right)+\\
&&+tr_2\left(\mathcal{F}\left(\frac{x}{|x|}|x|^{2z}\right)\left|\mathcal{F}(x)\right|^{2z} (\log|\mathcal{F}(x)|^2)\left(\frac{\mathcal{F}(x)}{|\mathcal{F}(x)|}\right)^*\right).
\end{eqnarray*}
Then evaluating the function at $z=\frac{1}{2}$, we obtain
\begin{eqnarray*}
F'(\frac{1}{2})&=&tr_2\left(\mathcal{F}\left(\frac{x}{|x|}|x|(\log|x|^2)\right) \left|\mathcal{F}(x)\right|\left(\frac{\mathcal{F}(x)}{|\mathcal{F}(x)|}\right)^*\right)+\\
&&+tr_2\left(\mathcal{F}\left(\frac{x}{|x|}|x|\right)\left|\mathcal{F}(x)\right| (\log|\mathcal{F}(x)|^2)\left(\frac{\mathcal{F}(x)}{|\mathcal{F}(x)|}\right)^*\right)\\
&=&tr_2(\mathcal{F}(x(\log |x|^2))\mathcal{F}(x)^*)+tr_2(\mathcal{F}(x)(\log\mathcal{F}(x)^2)\mathcal{F}(x)^*)\\
&=&tr_2(x(\log |x|^2)x^*)-H(|\mathcal{F}(x)|^2)\\
&=&-H(|x|^2)-H(|\mathcal{F}(x)|^2)=0.
\end{eqnarray*}

By Hopf's maximum principle, Proposition \ref{Hopf}, we have that $F(z)$ is constant on the disc center at $\frac{3}{4}$ with radius $\frac{1}{4}$. Hence
$$F(1)=F(\frac{1}{2})=tr_2(\mathcal{F}(x)\mathcal{F}(x)^*)=\delta.$$

On the other hand, we have
\begin{eqnarray*}
F(1)&=&tr_2\left(\mathcal{F}\left(\frac{x}{|x|}|x|^2\right)|\mathcal{F}(x)|^2\left(\frac{\mathcal{F}(x)}{|\mathcal{F}(x)|}\right)^*\right)\\
&=&tr_2(\mathcal{F}(x|x|)|\mathcal{F}(x)|\mathcal{F}(x)^*).
\end{eqnarray*}
Let $x=\sum_j\mu_jw_j$ and $\mathcal{F}(x)=\sum_k\lambda_kv_k$ be the rank-one decompositions.
Then
$$\delta=\|x\|_2^2=\sum_j\mu_j^2\|w_j\|_1,\quad \delta=\|\mathcal{F}(x)\|_2^2=\sum_k\lambda^2_k\|v_k\|_1,$$
and
$$\mathcal{F}(x|x|)=\sum_j\mu_j^2\mathcal{F}(w_j),\quad |\mathcal{F}(x)|\mathcal{F}(x)^*=\sum_k\lambda_k^2v_k^*.$$
Therefore
\begin{eqnarray*}
\delta=F(1)&=&tr_2(\sum_j\mu_j^2\mathcal{F}(w_j)\sum_k\lambda_k^2v_k^*)\\
&=&\sum_{j,k}\lambda_k^2\mu_j^2 tr_2(\mathcal{F}(w_j)v_k^*)
\end{eqnarray*}
On the other hand,
$$\delta^2=\|x\|_2^2\|\mathcal{F}(x)\|_2^2 =\sum_{j,k}\lambda_k^2\mu_j^2\|w_j\|_1\|v_k\|_1.$$
Combining the two equations above, we see that
\begin{equation}\label{eq0}
\sum_{j,k}\lambda_k^2\mu_j^2 \left(\frac{1}{\delta}\|w_j\|_1\|v_k\|_1-tr_2(\mathcal{F}(w_j)v_k^*)\right)=0.
\end{equation}
Note that $\lambda_k>0, \mu_j>0$, and by Proposition \ref{holder}, \ref{tr1},
$$|tr_2(\mathcal{F}(w_j)v_k^*)|\leq \|\mathcal{F}(w_j)\|_\infty \|v_k^*\|_1\leq \frac{\|w_j\|_1}{\delta}\|v_k\|_1.$$
By Equation (\ref{eq0}), we see that $\displaystyle tr_2(\mathcal{F}(w_j)v_k^*)=\frac{1}{\delta}\|w_j\|_1\|v_k\|_1$.
Therefore
$$\lambda_k\|v_k\|_1=tr_2(\mathcal{F}(x)v_k^*)=\sum_j\mu_jtr_2(\mathcal{F}(w_j)v_k^*)=\frac{1}{\delta}\sum_j\mu_j\|w_j\|_1\|v_k\|_1=\frac{1}{\delta}\|x\|_1\|v_k\|_1.$$
Then $\displaystyle \lambda_k=\frac{\|x\|_1}{\delta}$ for any $k$.
So $x$ is extremal and $\mathcal{F}(x)$ is a multiple of a partial isometry.
Similarly, we have $\displaystyle \mu_j=\frac{\|\mathcal{F}(x)\|_1}{\delta}$ for any $j$.
So $x$ is an extermal bi-partial isometry.

``(3)$\Rightarrow$(2)"
If $x$ is an extremal bi-partial isometry, then it is easy to check that all the equalities of the inequalities hold in the proof of Theorem \ref{UnP}.

``(2)$\Rightarrow$(1)"
From the proof of Corollary \ref{HtoD}, we see that the Hirshman-Beckner uncertainty principle is stronger than the Donoho-Stark uncertainty priniple.
So the minimizer of the latter one has to be that of the former one.
\end{proof}

Theorem \ref{minmain1} is a noncommutative version of Theorem 1.5 in \cite{OzPr} when $A$ is a finite abelian group. As showed in \cite{OzPr}, the minimizer of the classical uncertainty principle is a nonzero scalar multiple of a translation and a modulation of the indicator function of a subgroup of $A$. Their techniques to describe the extremal bi-partial isometries do not work in subfactor planar algebras, since we do not have the translation or the modulation to shift an extremal bi-partial isometry to a biprojection. We will define a notion of shift, a generalization of the translation and the modulation, and show that  extremal bi-partial isometries are \emph{bi-shift} of biprojections.
Recall that biprojections are a generalization of indicator functions of subgroups.

First, we need some new notions as follows.

\begin{definition}\label{shifts}
A projection $x$ in $\mathscr{P}_{2,\pm}$ is said to be a left shift of a biprojection $B$ if $tr_2(x)=tr_2(B)$ and $\displaystyle x*B=\frac{tr_2(B)}{\delta}x$. A projection $x$ in $\mathscr{P}_{2,\pm}$ is said to be a right shift of a biprojection $B$ if $tr_2(x)=tr_2(B)$ and $\displaystyle B*x=\frac{tr_2(B)}{\delta}x$.
\end{definition}

\begin{remark}
Note that this is similar to the translation. Since it is noncommutative, we have ``left'' and ``right'' here.
Later we will see a left shift of a biprojection is always a right shift of a biprojection.
We are able to find out all shifts of biprojections for concrete examples.
\end{remark}

\begin{remark}
For a right shift $x$ of a biprojection $B$, we have that $\mathcal{R}(\mathcal{F}^{-1}(x))=\mathcal{R}(\mathcal{F}^{-1}(B))$, since $tr_2(x)=tr_2(B)$, $\mathcal{R}(\mathcal{F}^{-1}(x))\leq \mathcal{R}(\mathcal{F}^{-1}(B))$, and the uncertainty principle $tr_2(x)tr_2(\mathcal{R}(\mathcal{F}^{-1}(x)))\geq \delta^2$. In particular, a right (or left) shift of a biprojection is a minimizer of the uncertainty principles.
\end{remark}

For an extremal bi-partial isometry $v$, we will see the range projections of $v^*$ and $\mathcal{F}^{-1}(v)$ are shifts of a pair of biprojections.
Moreover, $v$ is uniquely determined by the two range projections up to a scalar.
To construct such an extermal bi-partial isometry from two certain range projections, we introduce bishifts of biprojections.
There are 8 different ways to construct bi-shifts of biprojections.
We refer the reader to the Appendix for details.
Here we use one of them as its definition.

\begin{definition}\label{bishifts}
A nonzero element $x$ in $\mathscr{P}_{2,\pm}$ is said to be a bi-shift of a biprojection $B$ if there exist a right shift $Bg$ of the biprojection $B$ and a right shift $\widetilde{B}h$ of the biproejction $\widetilde{B}$ and an element $y$ in $\mathscr{P}_{2,\pm}$ such that $x=\mathcal{F}(\widetilde{B}h)*(yBg)$, i.e.,
$$x=\grb{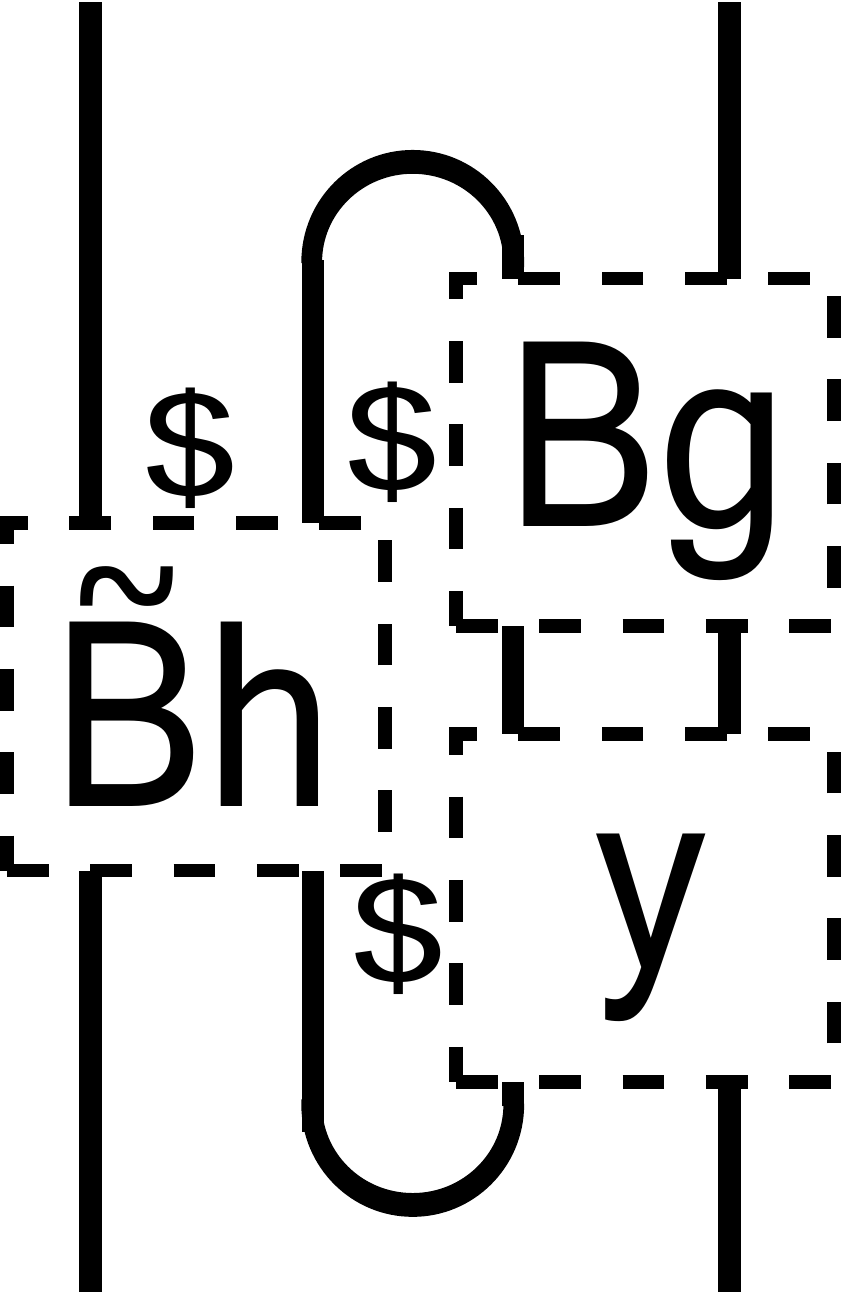},$$
where $\widetilde{B}$ is the range projection of $\mathcal{F}(B)$.
\end{definition}

\begin{lemma}\label{rangebipro}
Let $x$ be the above bi-shift of the biprojection $B$. Then $\mathcal{R}(x^*)=Bg$ and $\mathcal{R}(\mathcal{F}^{-1}(x))=\tilde{B}h$. Moreover, $x$ is a minimizer of the uncertainty principles.
\end{lemma}

\begin{proof}
Note that $x=\mathcal{F}(\tilde{B}h)*(yBg)$, we then have $\mathcal{F}^{-1}(x)=\tilde{B}h\mathcal{F}^{-1}(yBg)$. This implies that $\mathcal{R}(\mathcal{F}^{-1}(x))\leq \tilde{B}h$. On the other hand, we have
\begin{eqnarray*}
x^*&=&\mathcal{F}(\tilde{B}h)^**(Bgy^*)\\
&=&\mathcal{F}^{-1}(\tilde{B}h)*(Bg y^*),\\
\mathcal{R}(\mathcal{F}^{-1}(\tilde{B}h))&=&\mathcal{R}(\mathcal{F}^{-1}(\tilde{B}))\\
&=&\mathcal{R}(B)=B,\\
\mathcal{R}(x^*)&\leq & \mathcal{R}(\mathcal{R}(\mathcal{F}^{-1}(\tilde{B}h))*\mathcal{R}(Bg y^*))\quad \text{Proposition \ref{PQR=0}}\\
&\leq &\mathcal{R}(B*Bg)=Bg.
\end{eqnarray*}
By Theorem \ref{UnP}, we have
\begin{eqnarray*}
\delta^2&\leq &tr_2(\mathcal{R}(x))tr_2(\mathcal{R}(\mathcal{F}^{-1}(x)))\\
&=&tr_2(\mathcal{R}(x^*))tr_2(\mathcal{F}^{-1}(x)))\\
&\leq &tr_2(Bg)tr_2(\tilde{B}h)=tr_2(B)tr_2(\tilde{B})=\delta^2.
\end{eqnarray*}
This means that $\mathcal{R}(x^*)=Bg$ and $\mathcal{R}(\mathcal{F}^{-1}(x))=\tilde{B}h$. Moreover, $x$ is a minimizer of the uncertainty principles.
\end{proof}

\begin{lemma}\label{extreme}
Suppose $x\in\mathscr{P}_{2,\pm}$ is extremal. Let $x=\sum_k \lambda_k v_k$ and $\mathcal{F}(x)=\sum_l \mu_l w_l$ be rank-one decompositions.
Then
$$w_l^*\mathcal{F}(v_k)=\frac{\|v_k\|_1}{\delta}|w_l|,\text{ for all } k,$$
whenever $\mu_l=\|\mathcal{F}(x)\|_\infty$.
\end{lemma}

\begin{proof}
If $\mu_l=\|\mathcal{F}(x)\|_\infty=\frac{\|x\|_1}{\delta}$, we have
$$tr_2(w_l^*\mathcal{F}(x))=\mu_l\|w_l\|_1 =\frac{\|x\|_1}{\delta}\|w_l\|_1.$$
Recall that $x=\sum_k \lambda_k v_k$, we obtain
$$\sum_k \lambda_ktr_2(w_l^*\mathcal{F}(v_k))=\sum_k \lambda_k\frac{\|v_k\|_1}{\delta}\|w_l\|_1.$$
By Proposition \ref{holder} and \ref{tr1},
$$|tr_2(w_l^*\mathcal{F}(v_k))|\leq tr_2(|w_l^*\mathcal{F}(v_k)|)\leq \|w_l\|_1\|\mathcal{F}(v_k)\|_\infty \leq \|w_l\|_1 \frac{\|v_k\|_1}{\delta}.$$
Then
$$tr_2(w_l^*\mathcal{F}(v_k))= tr_2(|w_l^*\mathcal{F}(v_k)|)=\|w_l\|_1\frac{\|v_k\|_1}{\delta}=tr_2(|w_1|\frac{\|v_k\|_1}{\delta}), \forall k.$$
Note that $\mathcal{R}(w_l^*\mathcal{F}(v_k))=|w_l|$, and
$$\|w_l^*\mathcal{F}(v_k)\|_\infty=\|\mathcal{F}(v_k)\|_\infty\leq \frac{\|v_k\|_1}{\delta},$$
so
$w_l^*\mathcal{F}(v_k)=\frac{\|v_k\|_1}{\delta}|w_l|$ for all $k$.
\end{proof}

\begin{definition}
Suppose $v\in\mathscr{P}_{2,+},w\in\mathscr{P}_{2,-}$, we say $v$ is extremal with respect to $w$ if
$$tr_2(w^*\mathcal{F}(v))=\overline{tr_2(v^*\mathcal{F}^{-1}(w))}=\frac{\|v\|_1\|w\|_1}{\delta}.$$
\end{definition}

\begin{remark}
In general, we have $$|tr_2(w^*\mathcal{F}(v))|=|\overline{tr_2(v^*\mathcal{F}^{-1}(w))}|\leq\frac{\|v\|_1\|w\|_1}{\delta}.$$
\end{remark}

\begin{proposition}\label{bipartial}
If $w$ is a partial isometry and $\mathcal{F}^{-1}(w)$ is extremal, then $w$ is an extremal bi-partial isometry.
\end{proposition}

This proves part of Main Theorem \ref{mainthm2}.

\begin{proof}
By Proposition \ref{holdereq}, $w$ is a multiple of a partial isometry if and only if $\|w\|_2^2=\|w\|_\infty\|w\|_1$. To see $\mathcal{F}(w)$ is a partial isometry, we are going to check $\|\mathcal{F}(w)\|_2^2=\|\mathcal{F}(w)\|_\infty\|\mathcal{F}(w)\|_1$. If $\mathcal{F}^{-1}(w)$ is extremal, then
$$\|w\|_\infty=\|\mathcal{F}(\mathcal{F}^{-1}(w))\|_\infty=\frac{\|\mathcal{F}^{-1}(w)\|_1}{\delta}.$$
If $w$ is a partial isometry, then
\begin{eqnarray*}
\|\mathcal{F}(w)\|_\infty\|\mathcal{F}(w)\|_1&\geq& \|\mathcal{F}(w)\|_2^2=\|w\|_2^2\\
&=&\|w\|_\infty\|w\|_1\\
&\geq &\frac{\|\mathcal{F}^{-1}(w)\|_1}{\delta}\delta\|\mathcal{F}(w)\|_\infty\\
&=&\|\mathcal{F}(w)\|_\infty\|\mathcal{F}(w)\|_1
\end{eqnarray*}
Hence $\|\mathcal{F}(w)\|_2^2=\|\mathcal{F}(w)\|_\infty\|\mathcal{F}(w)\|_1$ and $\|\mathcal{F}(w)\|_\infty=\frac{1}{\delta}\|w\|_1$. Then $\mathcal{F}(w)$ is a multiple of a partial isometry and $w$ is extremal.
\end{proof}

\begin{theorem}\label{squarerelation}
Suppose $\mathscr{P}$ is an irreducible subfactor planar algebra, and $w\in\mathscr{P}_{2,\pm}$.
If $w$ is a partial isometry and $\mathcal{F}^{-1}(w)$ is extremal, then $\frac{\delta}{\|w\|_2^2}w*\overline{w}^*$ is a partial isometry, and

$$(w^**\overline{w})(w*\overline{w}^*)=\frac{\|w\|_2^2}{\delta}(w^*w)*(\overline{w}~\overline{w}^*), \quad \text{i.e.}$$

$$\grc{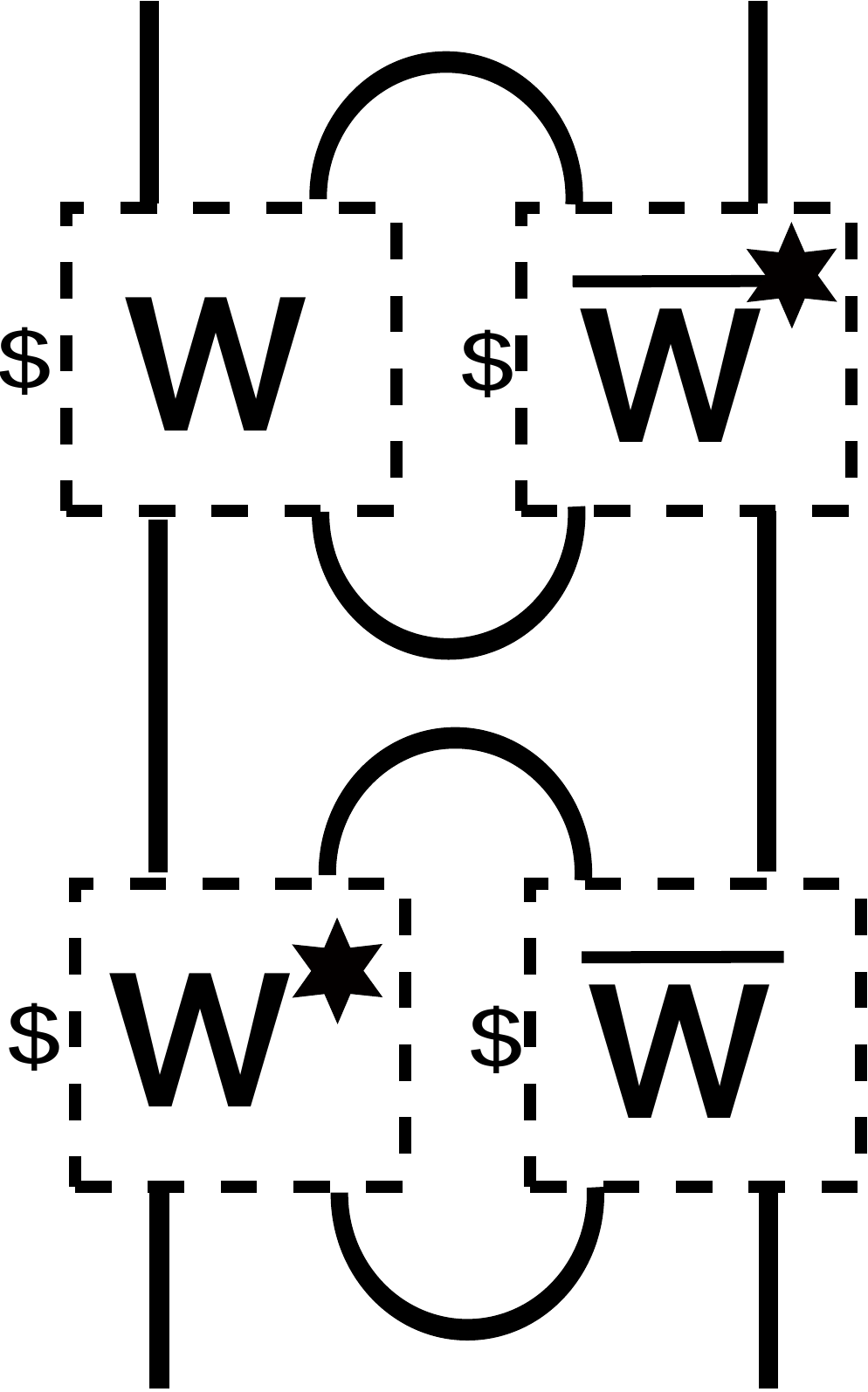}=\frac{\|w\|_2^2}{\delta}\grc{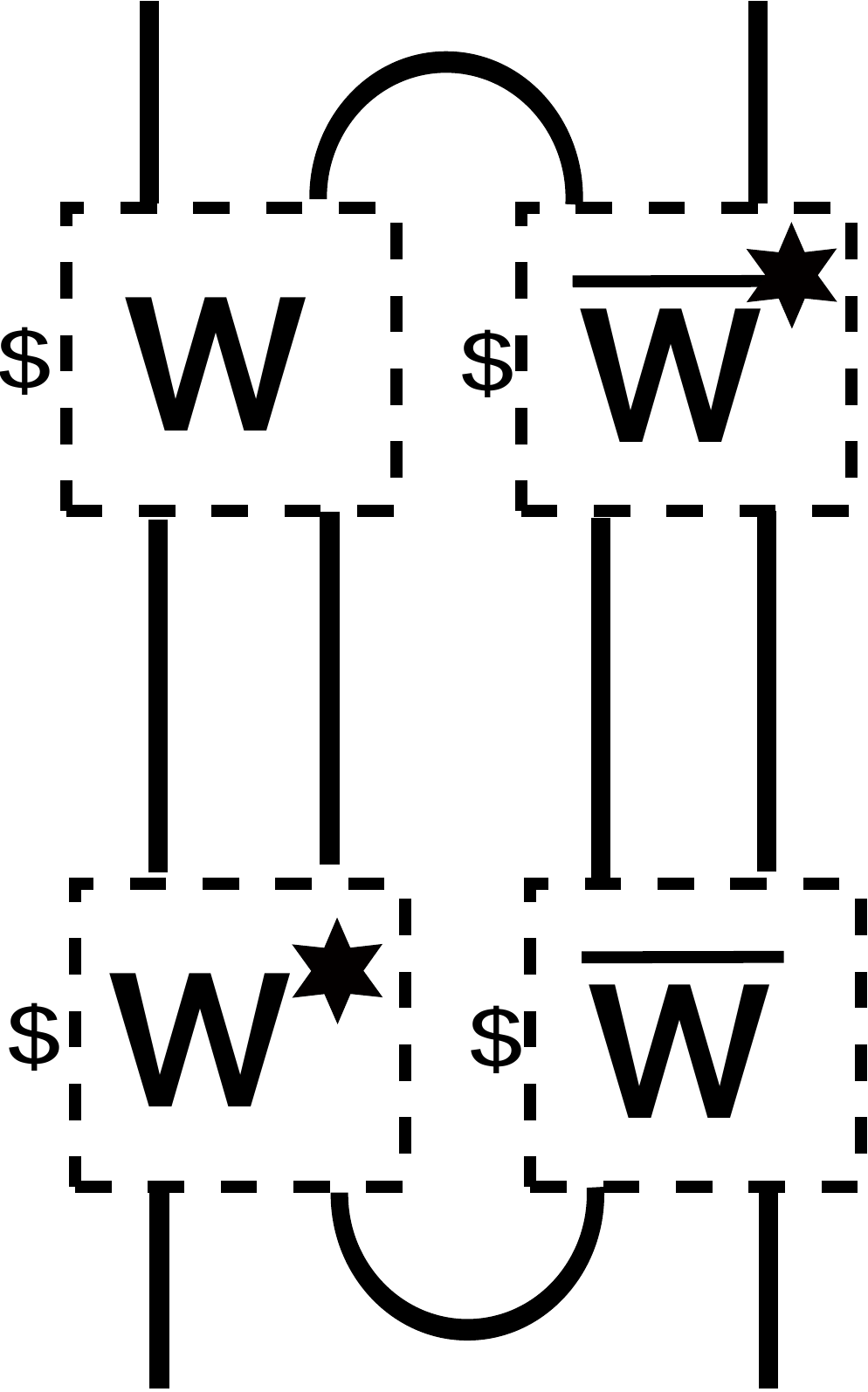}.$$
Consequently $\displaystyle \|w\|_1=\|\frac{\delta}{\|w\|_2^2}w*\overline{w}^*\|_1$.
\end{theorem}

\begin{proof}
Suppose $w$ is a partial isometry and $\mathcal{F}^{-1}(w)$ is extremal.
Let $w=\sum_l w_l$ and $x=\mathcal{F}^{-1}(w)^*=\sum_k \lambda_k v_k$ be the rank-one decompositions.

Note that $x$ is extremal, so $\displaystyle \frac{\|x\|_1}{\delta}=\|\mathcal{F}(x)\|_\infty=\|w\|_\infty=1$.
Let
\begin{equation}\label{equ2}
\grb{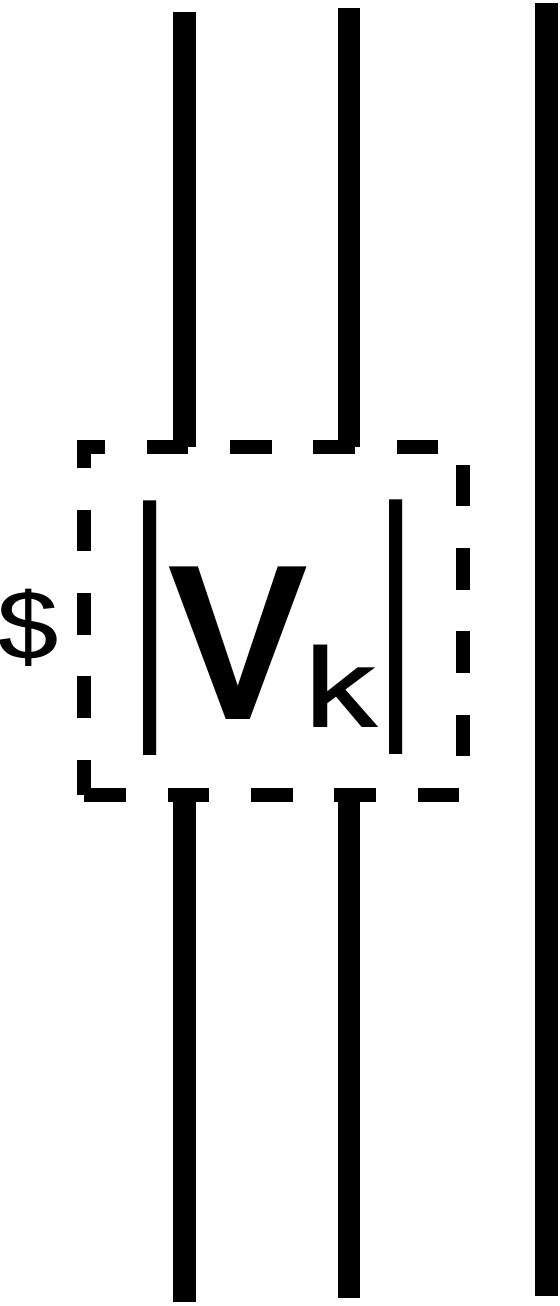}=\sum_{j=1}^{n} \frac{\delta}{\|v_k\|_1}\grb{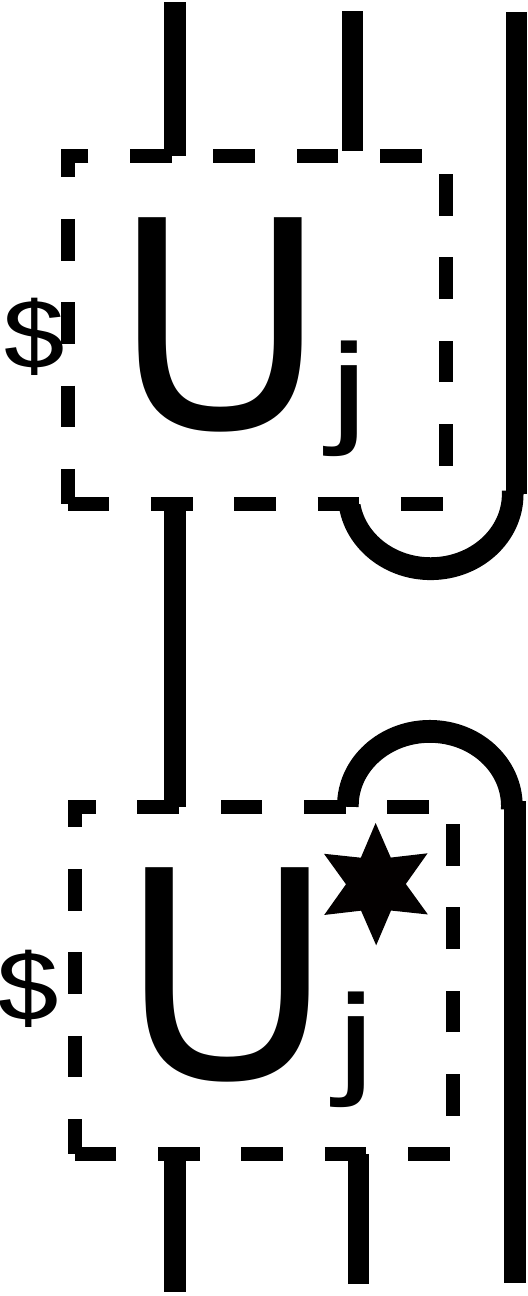}+\grb{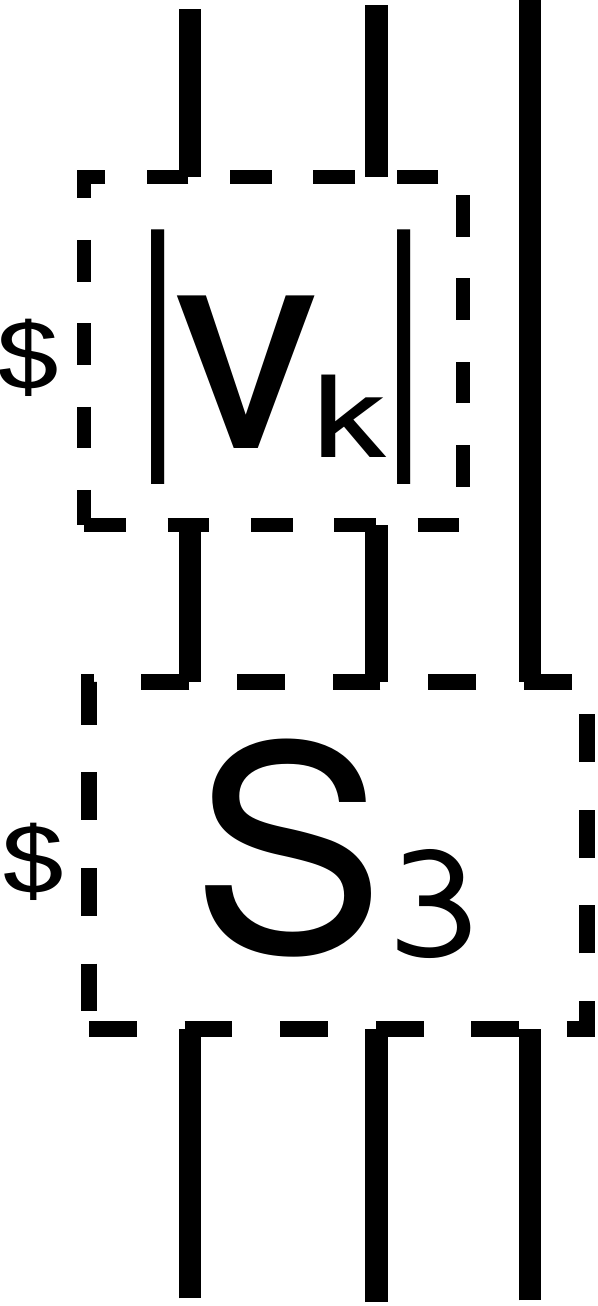}
\end{equation}
be Wenzl's formula (\ref{wenzl}), such that $U_1=v_k$.
Adding a cap to the left, we have the local relation (\ref{liu})
$$\grb{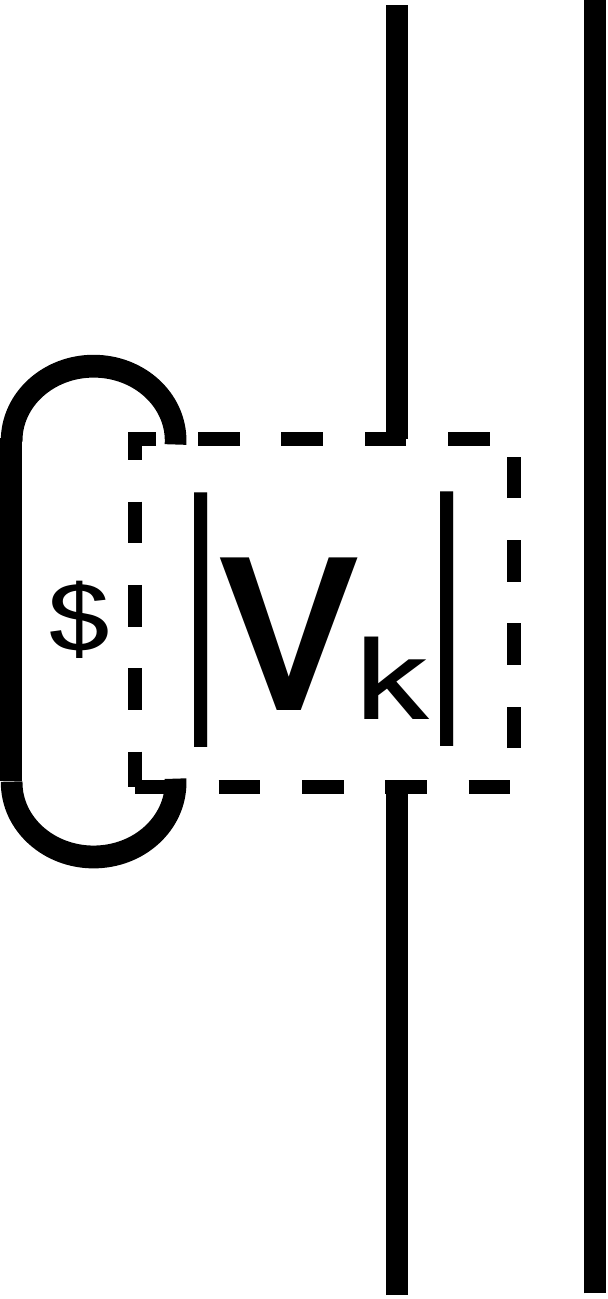}=\sum_{j=1}^{n} \frac{\delta}{\|v_k\|_1}\grb{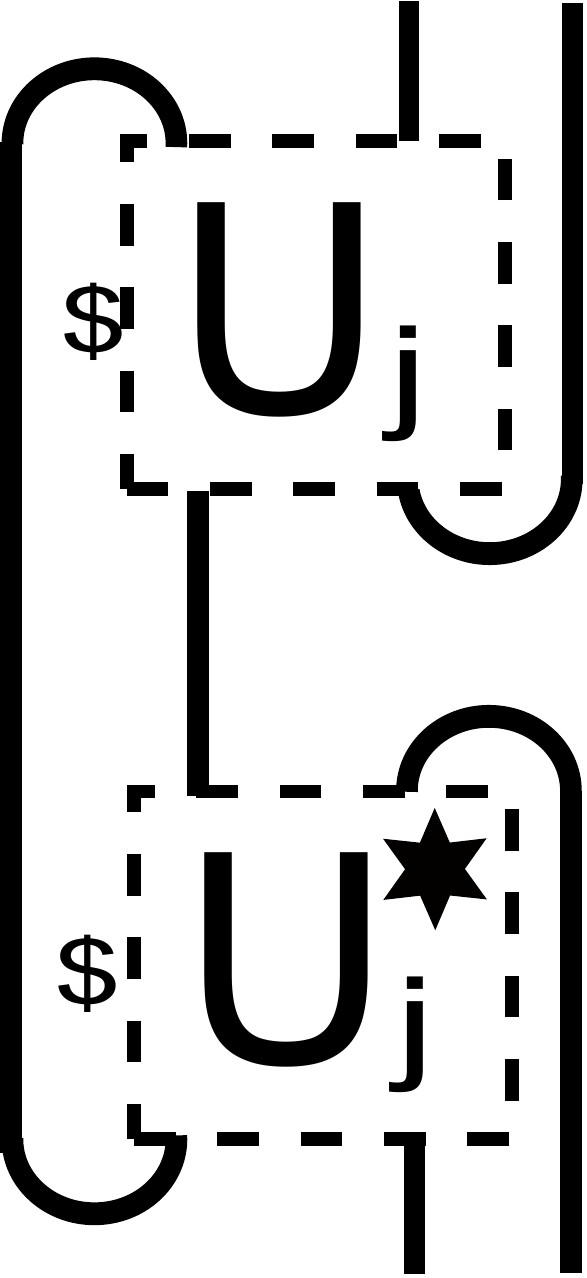}+\grb{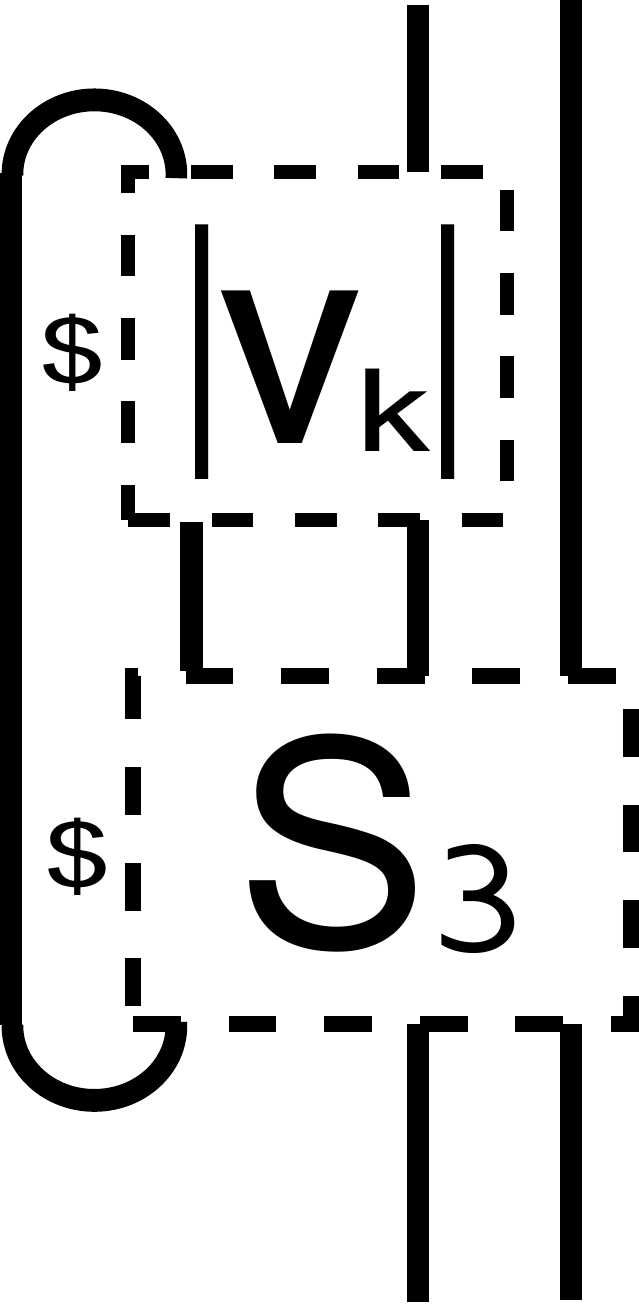}.$$
Then
\begin{equation}\label{equ1}
\grc{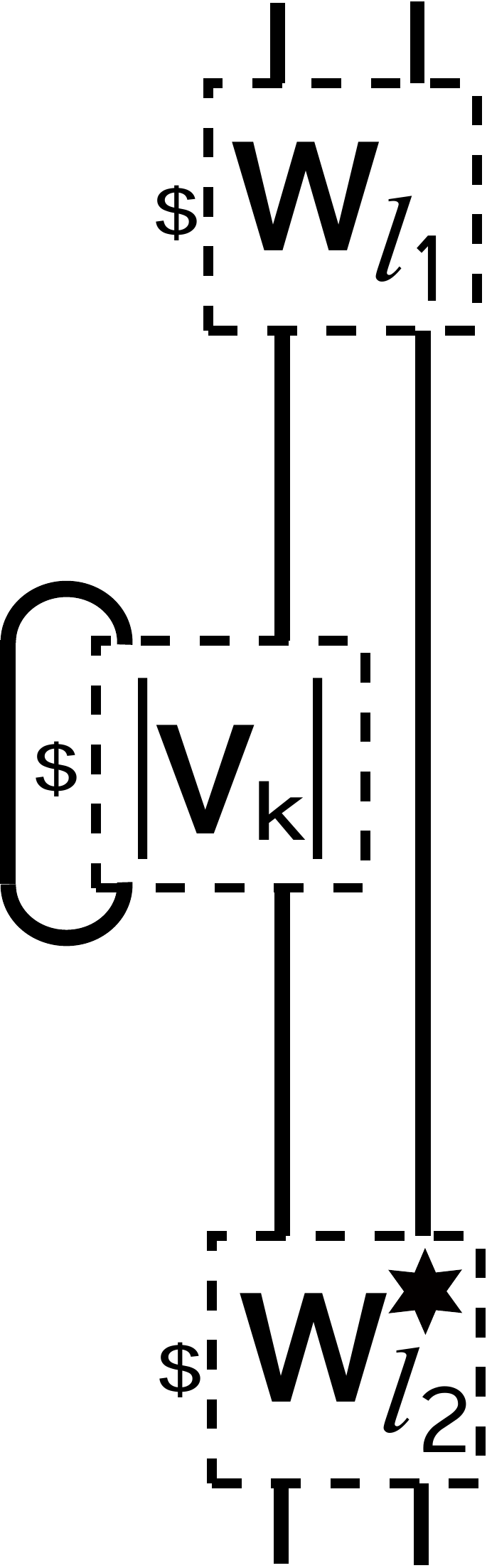}=\sum_{j=1}^{n} \frac{\delta}{\|v_k\|_1}\grc{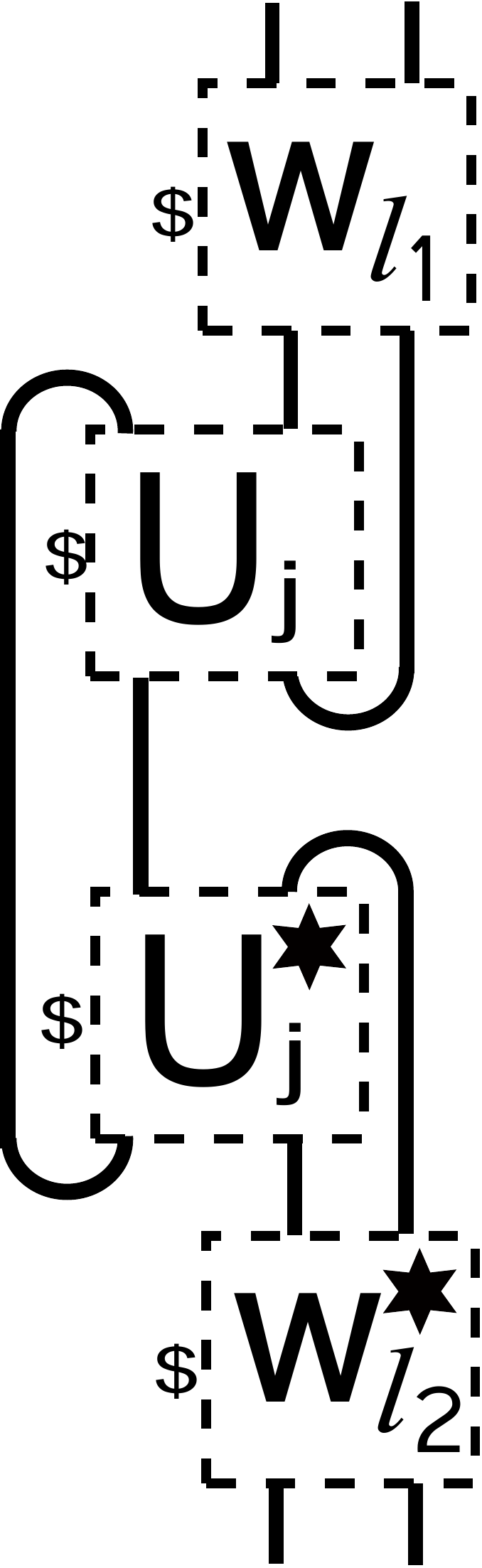}+\grc{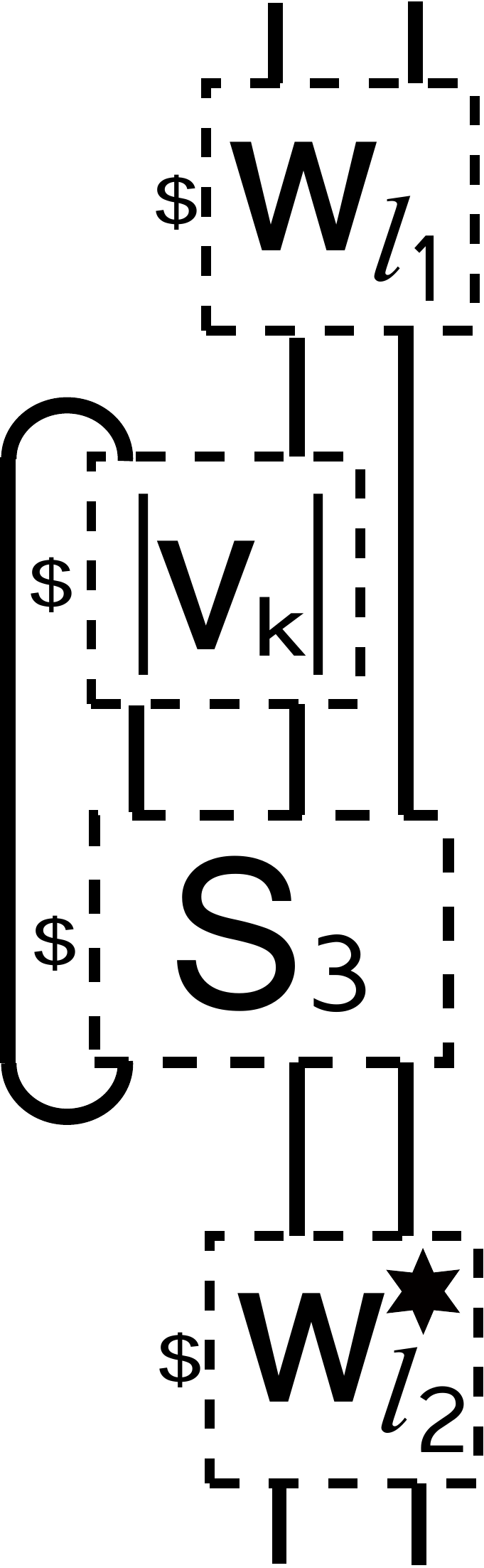}.
\end{equation}
When $l_1=l_2$, each summand of the right hand side is positive.

We are going to prove
\begin{equation}\label{ppp}
\grc{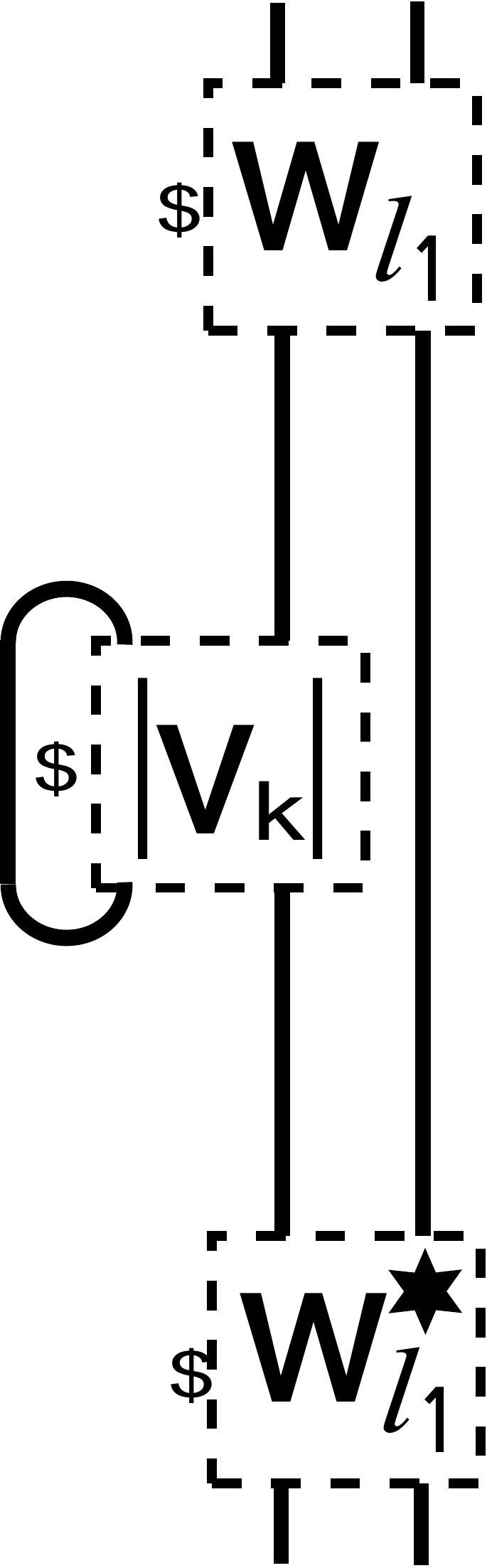}=\frac{\delta}{\|v_k\|_1}\grc{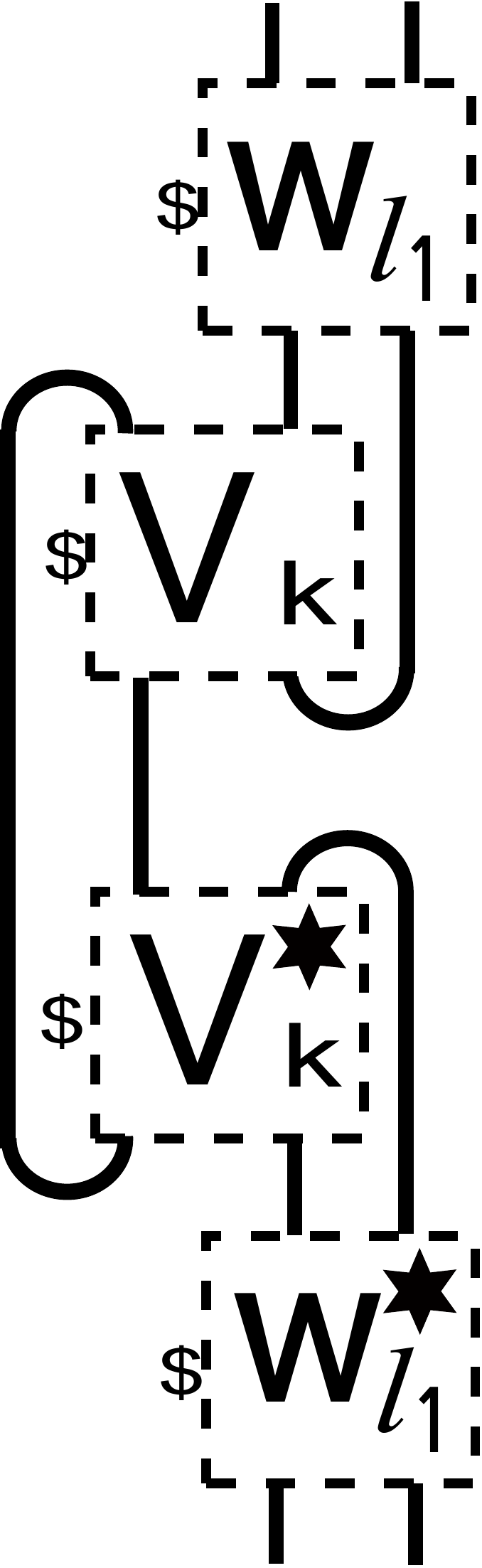},
\end{equation}
which implies the rest summands of the right side of the Equation (\ref{equ1}) are zeros.
The left hand side of Equation (\ref{ppp}) is
$$\frac{\delta}{\|v_k\|_1}w_{l_1}^*w_{l_1} =\frac{\delta}{\|v_k\|_1}|w_{l_1}|.$$
The right hand side of Equation (\ref{ppp}) is
$$\frac{\delta}{\|v_k\|_1}w_{l_1}^* \mathcal{F}(v_k^*) \mathcal{F}^{-1}(v_k)  w_{l_1}=\frac{\delta}{\|v_k\|_1}w_{l_1}^* \mathcal{F}(v_k^*) (\mathcal{F}(v_k^*))^*  w_{l_1}.$$
From Lemma \ref{extreme}, we have
$$\frac{\|v_k\|_1}{\delta}|w_{l_1}|=w_{l_1}^* \mathcal{F}(v_k^*).$$
Thus
$$\frac{\delta}{\|v_k\|_1}w_{l_1}^* \mathcal{F}(v_k^*) (\mathcal{F}(v_k^*))^*  w_{l_1}
=\frac{\delta}{\|v_k\|_1}\left(\frac{\|v_k\|_1}{\delta}|w_{l_1}|\right)^2
=\frac{\|v_k\|_1}{\delta}|w_{l_1}|.$$
So Equation (\ref{ppp}) holds.

Now we have $\grc{vkwliu2}=0$, for $j\neq 1$, so $\grb{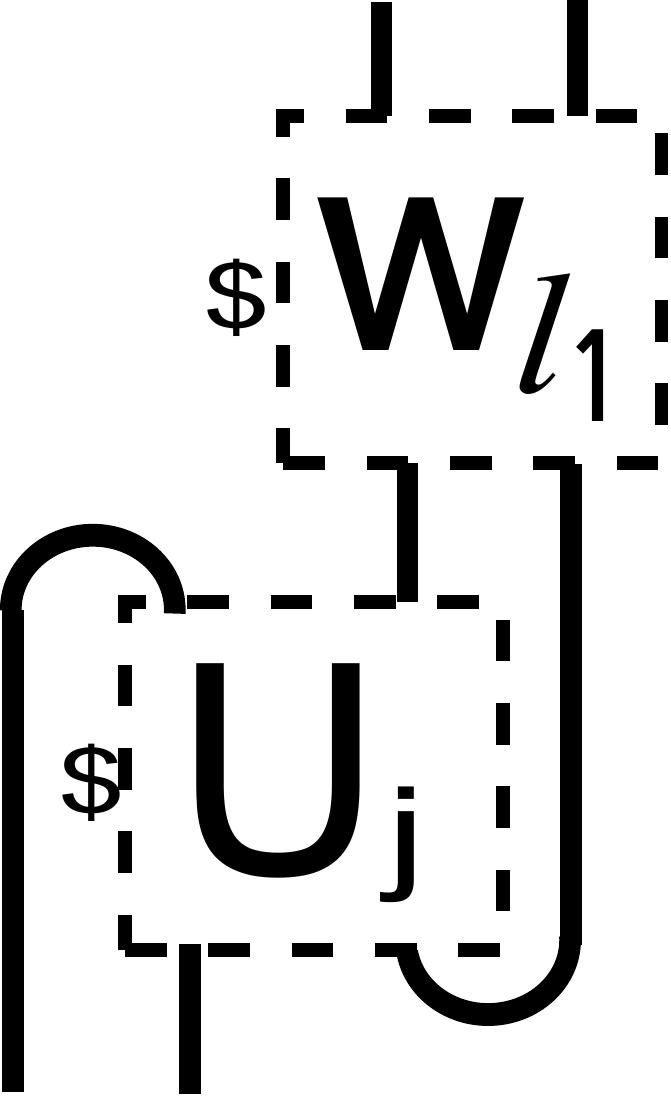}=0$.
Note that $\grb{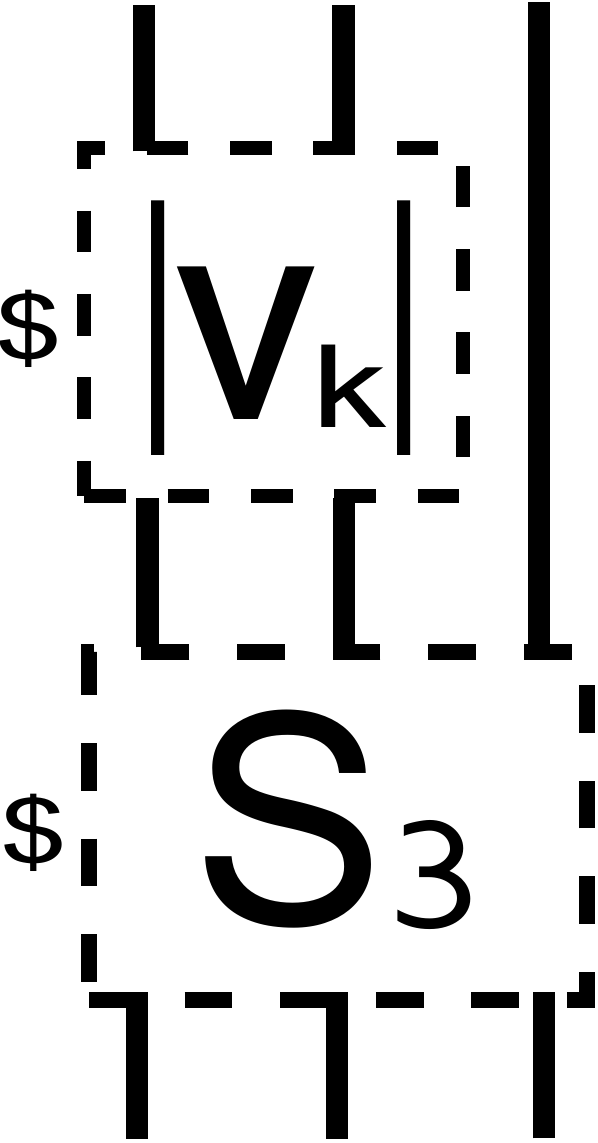}$ is a projection, so $\grb{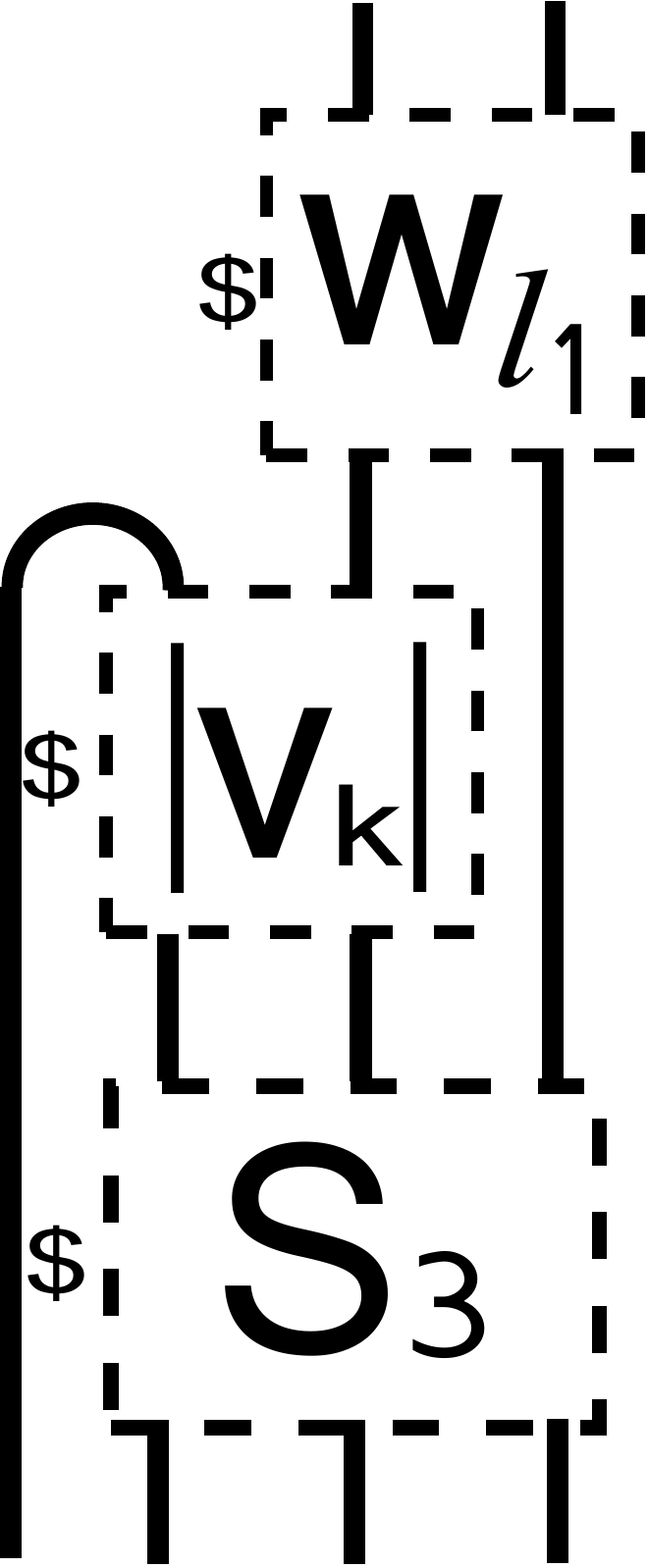}=0$.
By Equation (\ref{equ2}), we have
$$\grc{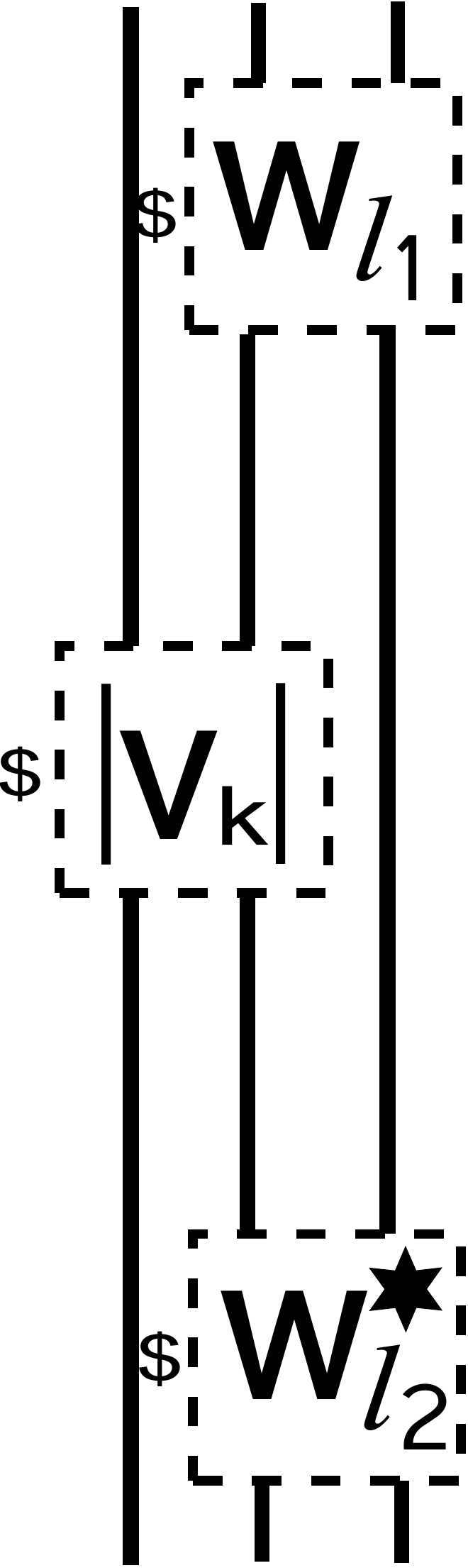}=\sum_{j=1}^{n} \frac{\delta}{\|v_k\|_1}\grc{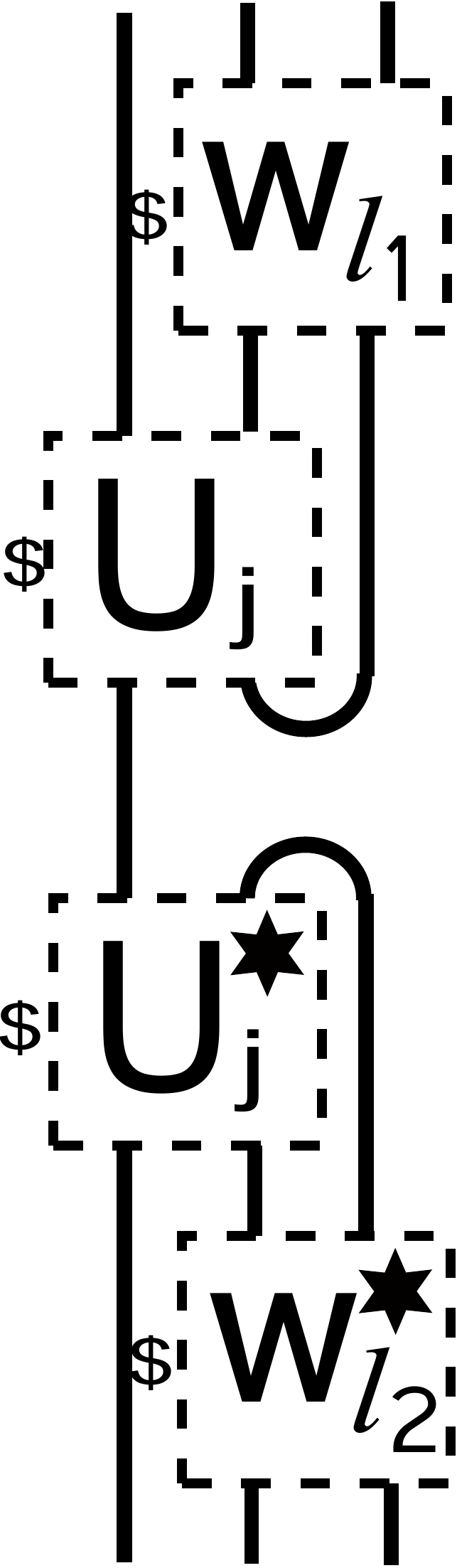}+\grc{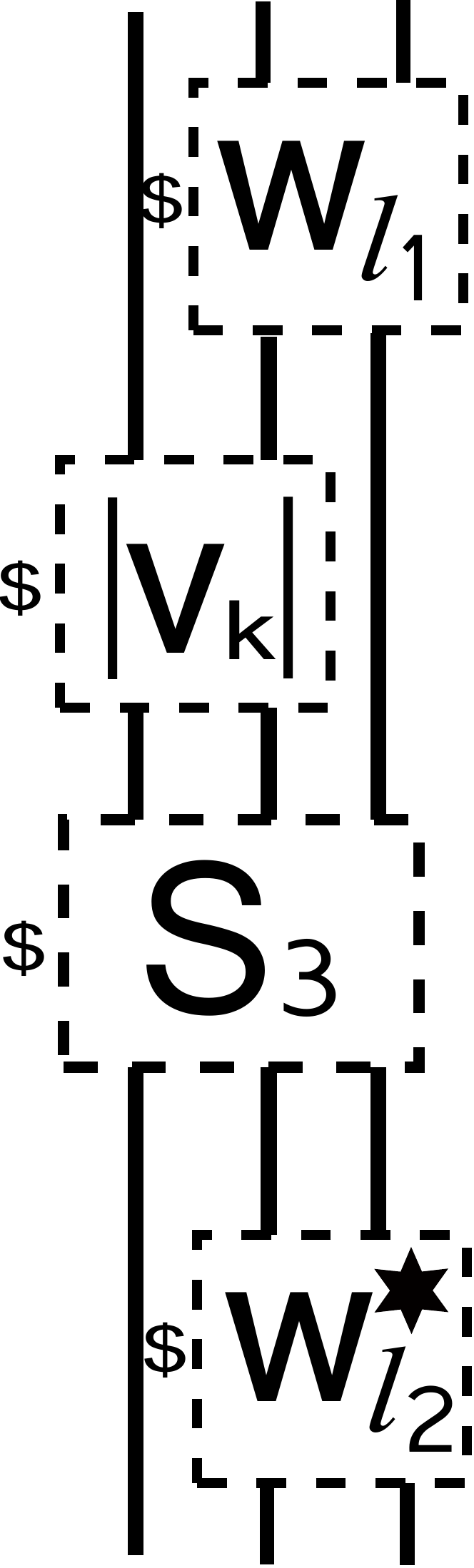}=\frac{\delta}{\|v_k\|_1}\grc{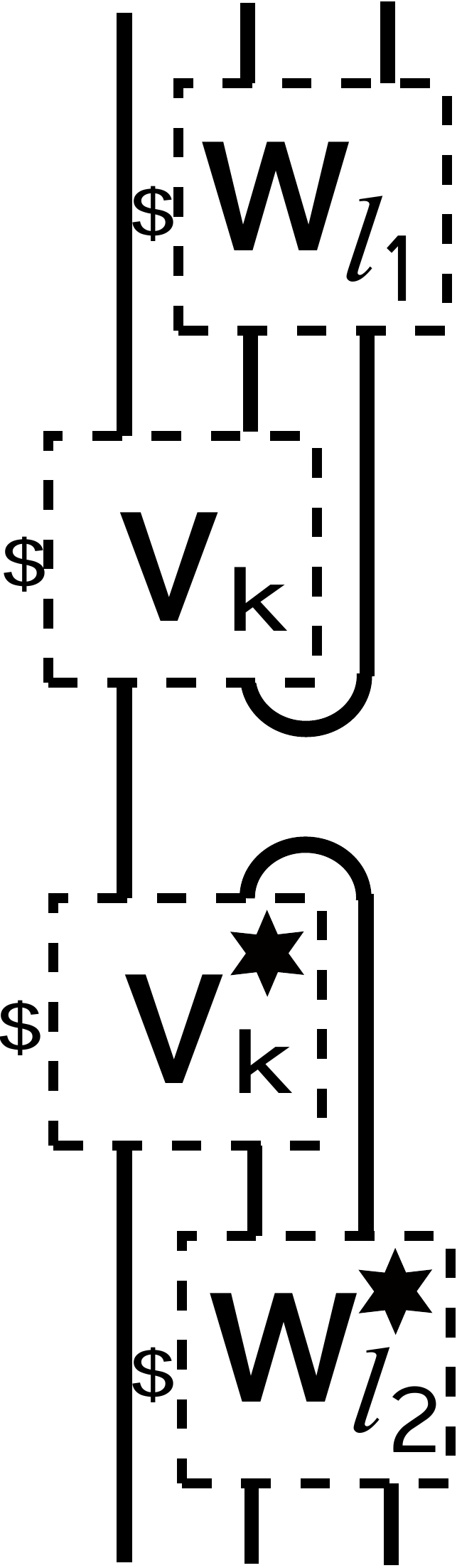}=\frac{\delta}{\|v_k\|_1}\grc{vkwwen4}=\frac{\delta}{\|v_k\|_1}(\frac{\|v_k\|_1}{\delta})^2\grb{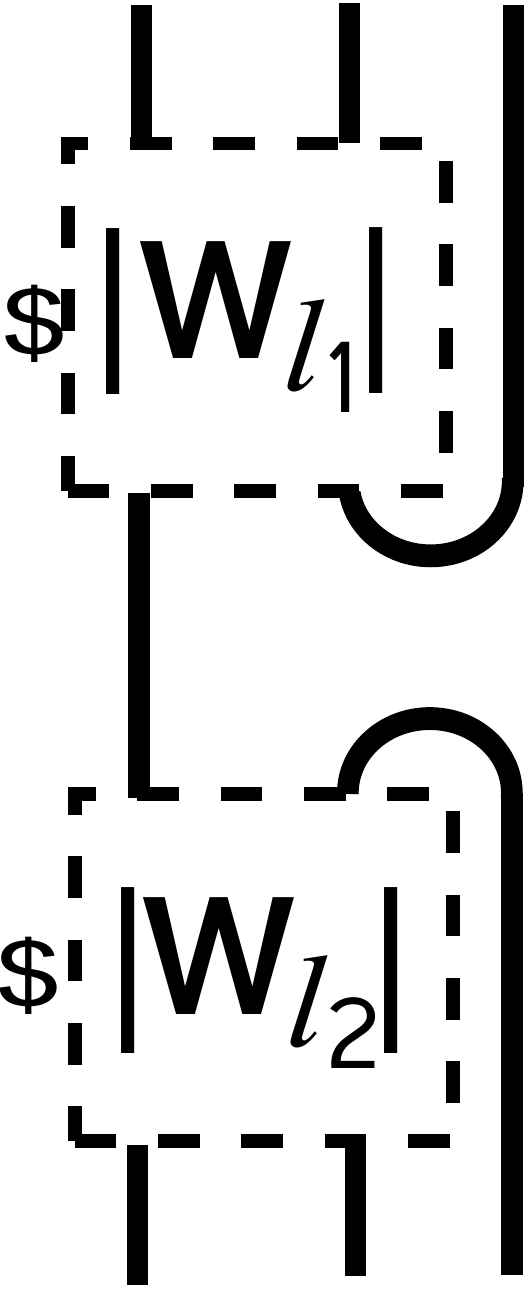},$$
the last equality follows from Lemma \ref{extreme}.
Adding a cap to the right and then taking the Fourier transform (the 1-click rotation), we have
$$\grb{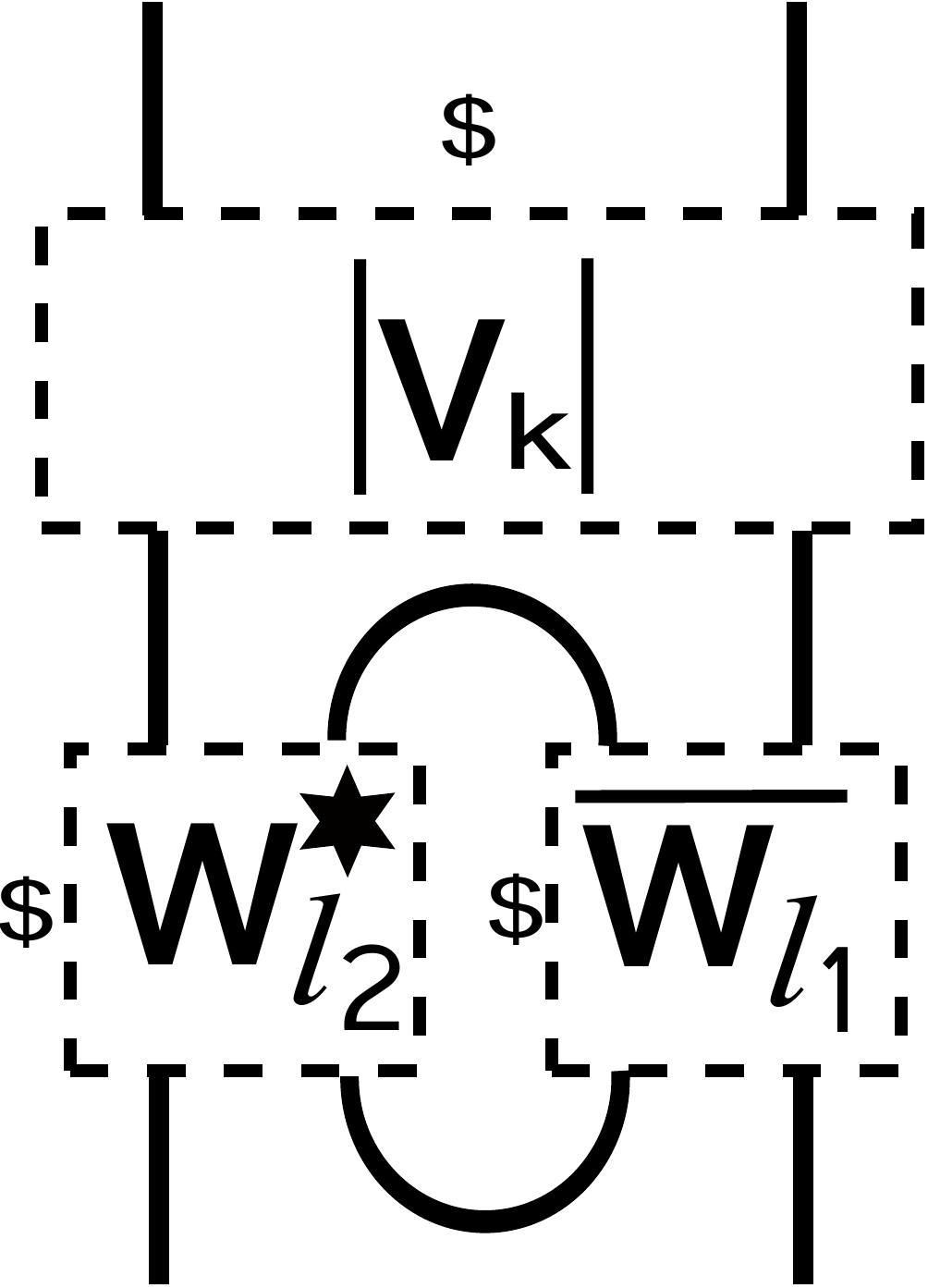}=\frac{\|v_k\|_1}{\delta}\graa{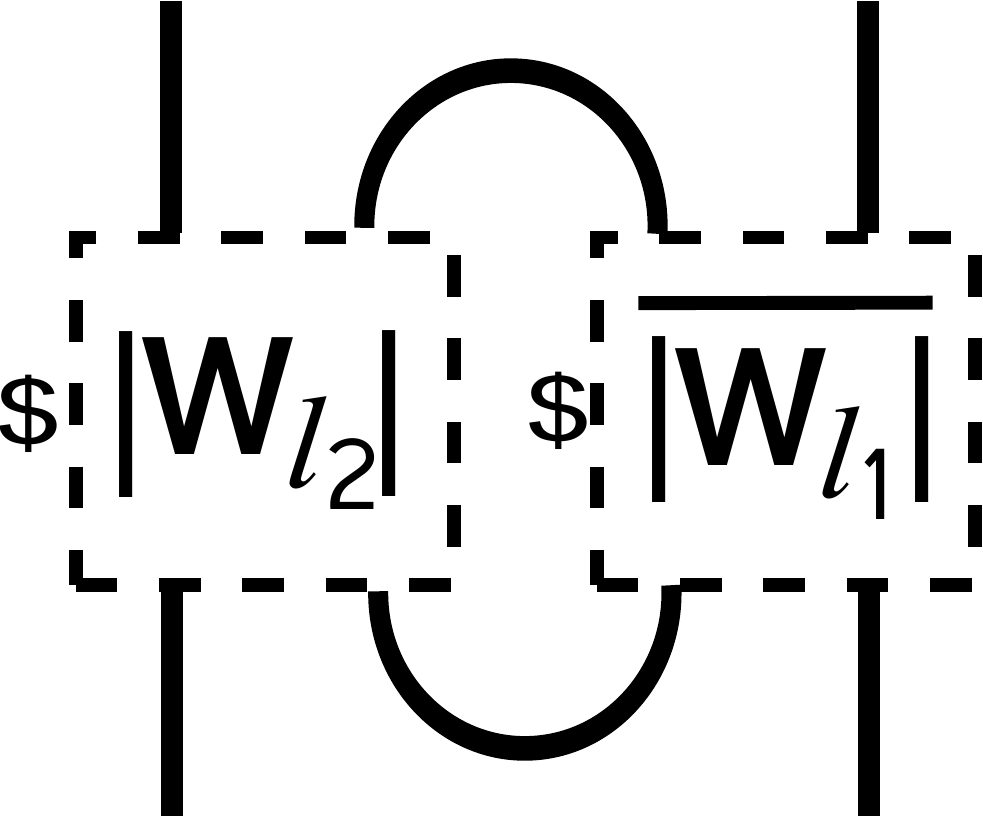}.$$
Then
$$\sum_{k,l_1,l_2}\lambda_k^2\grb{vkwco1}=\sum_{k,l_1,l_2}\lambda_k^2\frac{\|v_k\|_1}{\delta}\graa{vkwco2}.$$
Note that $\sum_l w_l=w$, $\sum_k\lambda_k^2\|v_k\|_1=\|x\|_2^2$, and
$$\mathcal{F}(\sum_k\lambda_k^2 |v_k|)=\mathcal{F}((\sum_k\lambda_kv_k)^*(\sum_k\lambda_kv_k))=\mathcal{F}(x^*x)=w*\overline{w}^*,$$
so
$$\grc{square1}=\sum_{l_1,l_2}\frac{\|x\|_2^2}{\delta}\graa{vkwco2}.$$
Computing the trace on both sides, we have
\begin{eqnarray*}
tr_2((w^**\overline{w})(w*\overline{w}^*))&=&\frac{\|x\|_2^2}{\delta}\sum_{l_1,l_2}tr_2(|w_{l_2}|*\overline{|w_{l_1}|})\\
&=&\frac{\|x\|_2^2}{\delta}\sum_{l_1,l_2} \frac{\|w_{l_2}\|_1\|w_{l_1}\|_1}{\delta}\\
&=&\frac{\|x\|_2^2}{\delta} \frac{\|w\|_1^2}{\delta}.
\end{eqnarray*}
By H\"{o}lder's inequality (Proposition \ref{holder}), we have
\begin{equation}\label{eq4}
tr_2((w^**\overline{w})(w*\overline{w}^*))\leq \|w^**\overline{w}\|_\infty\|w*\overline{w}^*\|_1.
\end{equation}
So
$$\frac{\|x\|_2^2}{\delta} \frac{\|w\|_1^2}{\delta}\leq \|w^**\overline{w}\|_\infty\|w*\overline{w}^*\|_1.$$
On the other hand, by Proposition \ref{tr1} and Lemma \ref{con1}, we have
\begin{equation}\label{equ3}
||w*\overline{w}^*||_\infty\leq \frac{\|x^*x\|_1}{\delta}=\frac{\|x\|_2^2}{\delta}
\end{equation}
 and
$$\|w*\overline{w}^*\|_1\leq \frac{\|w\|_1\|\overline{w}^*\|_1}{\delta}=\frac{\|w\|_1^2}{\delta}.$$
Hence all equalities of the inequalities hold.
Note that $(w^**\overline{w})=(w*\overline{w}^*)^*$. Then by the equality of (\ref{eq4}), we have that
$$\|w*\overline{w}^*\|_2^2=\|w*\overline{w}^*\|_\infty\|w*\overline{w}^*\|_1.$$
By Proposition \ref{norm}, $w*\overline{w}^*$ is a multiple of a partial isometry.
By the equality of Equation (\ref{equ3}), we have that $\displaystyle \frac{\delta}{\|w\|_2^2}w*\overline{w}^*$ is a partial isometry.

Furthermore, since $\sum_l|w_l|=w^*w$ and $\|x\|_2^2=\|w\|_2^2$, we see that
$$\grc{square1}=\frac{\|w\|_2^2}{\delta}\grc{square2}.$$

Recall that $w$ and $\displaystyle \frac{\delta}{\|w\|_2^2}w*\overline{w}^*$ are partial isometries, and $\|w\|_2^2=\|w\|_1$. Computing the trace on both sides of the above equation, we have
$$\|w\|_1=\|\frac{\delta}{\|w\|_2^2}w*\overline{w}^*\|_1.$$

\end{proof}

\begin{corollary}
Suppose $\mathscr{P}$ is a subfactor planar algebra, and $w\in\mathscr{P}_{2,\pm}$. If $\mathcal{F}^{-1}(w)$ is extremal,  then $wQ$ is an extremal bi-partial isometry, where $Q$ is the spectral projection of $|w|$ with spectrum $\|w\|_\infty$.
\end{corollary}

\begin{proof}
Without loss of generality, we assume that $\|x\|_\infty=1$.
Note that
$$\lim_{k\rightarrow\infty}w(w^*w)^k=wQ,$$
which is a partial isometry.

If $\mathcal{F}^{-1}\left(w\right)$ is extremal, then $\|\mathcal{F}^{-1}\left(w\right)\|_1=\delta$. By Lemma \ref{con2}, we have
$$\|\mathcal{F}^{-1}(w(w^*w)^k)\|_1=
\|\mathcal{F}^{-1}(w)*\mathcal{F}^{-1}(w^*)*\mathcal{F}^{-1}(w)*\cdots*\mathcal{F}^{-1}(w^*)*\mathcal{F}^{-1}(w))\|_1
\leq \delta.$$
Recall that $\mathcal{F}^{-1}$ is continuous, so
$$\|\mathcal{F}^{-1}(wQ)\|_1\leq \delta.$$
On the other hand, by Proposition \ref{tr1}, we have
$$1=\|wQ\|_\infty\leq\frac{1}{\delta}\|\mathcal{F}^{-1}(wQ)\|_1.$$
So the equality of the above inequality holds
and $\mathcal{F}^{-1}(wQ)$ is extremal. By Proposition \ref{bipartial}, $wQ$ is an extremal bi-partial isometry.

\end{proof}

\begin{theorem}\label{minmain2}
Suppose $\mathscr{P}$ is an irreducible subfactor planar algebra, and $w\in\mathscr{P}_{2,\pm}$. Then
$w$ is an extremal bi-partial isometry if and only if $w$ is a bi-shift of a biprojection.
Furthermore, if $w$ is a projection, then it is a left (or right) shift of a biprojection.
\end{theorem}

This together with Theorem \ref{minmain1} and Proposition \ref{bipartial} proves Main Theorem \ref{mainthm2}.

\begin{proof}
Suppose $w$ is an extremal bi-partial isometry and $w$ is a partial isometry. We define an element
$$B=\left(\frac{\delta}{\|w\|_2^2}\right)^2\grb{square1}.$$
By Theorem \ref{squarerelation},
we have $\displaystyle \frac{\delta}{\|w\|_2^2}w*\overline{w}^*$ is a partial isometry and $B$ is a projection.
Note that $\mathcal{F}(w^*)$ is an extremal bi-partial isometry, by Theorem \ref{squarerelation}, $\mathcal{F}(B)$ is a multiple of a projection, and $B$ is a biprojection.

We define an element $Bg=w^*w$, then $Bg$ is a projection. We are going to show $Bg$ is a right shift of the biprojection $B$. By Theorem \ref{squarerelation}, we have $\displaystyle \frac{\delta}{\|w\|_2^2}Bg*\overline{Bg}=B$. Computing the trace on both sides, we have $\displaystyle \frac{\delta}{\|w\|_2^2}\frac{tr_2(Bg)^2}{\delta}=tr_2(B)$.
Note that $\|w\|_2^2=tr_2(Bg)$, so $tr_2(Bg)=tr_2(B)$.

Recall that $\mathcal{F}(w^*)$ is an extremal bi-partial isometry, so $\displaystyle \frac{\delta}{\|w\|_1}\mathcal{F}(w^*)$ is a partial isometry. By Theorem \ref{squarerelation}, we see that
$$\frac{\delta}{\|w\|_2^2}\mathcal{F}\left(Bg\right)
=\frac{\delta}{\|\frac{\delta}{\|w\|_1}\mathcal{F}\left(w^*\right)\|_2^2}\left(\displaystyle \frac{\delta}{\|w\|_1}\mathcal{F}\left(w^*\right)\right) *\left(\frac{\delta}{\|w\|_1}\overline{\mathcal{F}\left(w^*\right)}^*\right)$$
is a partial isometry.
That is, $$\left(\frac{\delta}{\|w\|_2^2}\right)^2\mathcal{F}\left(Bg\right)\mathcal{F}\left(Bg\right)^*\mathcal{F}\left(Bg\right)=\mathcal{F}\left(Bg\right).$$
Then
$$B*Bg=\frac{\delta}{\|w\|_2^2}Bg*\overline{Bg}*Bg=\frac{\|w\|_2^2}{\delta}Bg=\frac{\|Bg\|_1}{\delta}Bg.$$
Therefore $Bg$ is a right shift of the biprojection $B$.

Note that $\mathcal{F}^{-1}\left(\overline{w}^*\right)$ is an extremal bi-partial isometry, so $\displaystyle \frac{\delta}{\|w\|_1}\mathcal{F}^{-1}\left(\overline{w}^*\right)$ is a partial isometry.
Take $\widetilde{B}=\mathcal{R}(\mathcal{F}\left(B\right))$ and $\widetilde{B}h=\left(\displaystyle \frac{\delta}{\|w\|_1}\mathcal{F}^{-1}\left(\overline{w}^*\right)\right)^*\left(\displaystyle \frac{\delta}{\|w\|_1}\mathcal{F}^{-1}\left(\overline{w}^*\right)\right)$.
Similarly $\widetilde{B}h$ is a right shift of $\widetilde{B}$. Recall that $\displaystyle \frac{\delta}{\|w\|_1}\mathcal{F}^{-1}\left(\overline{w}^*\right)$ is a partial isometry, so $w=\left(\displaystyle \frac{\delta}{\|w\|_1}\right)^2w*\overline{w}^**w$.
Then
$$w=\left(\displaystyle \frac{\delta}{\|w\|_1}\right)^2\grb{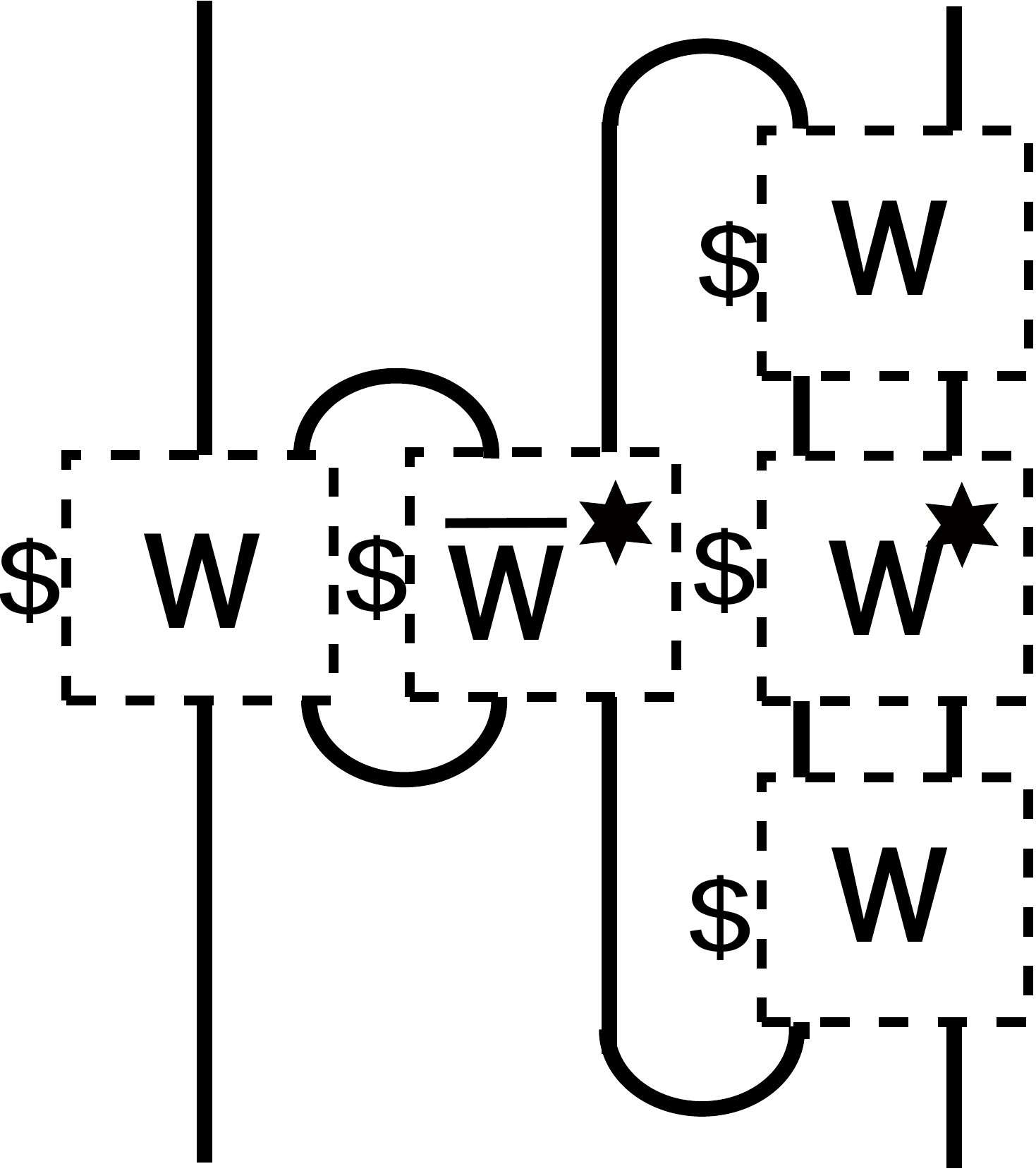}=\grb{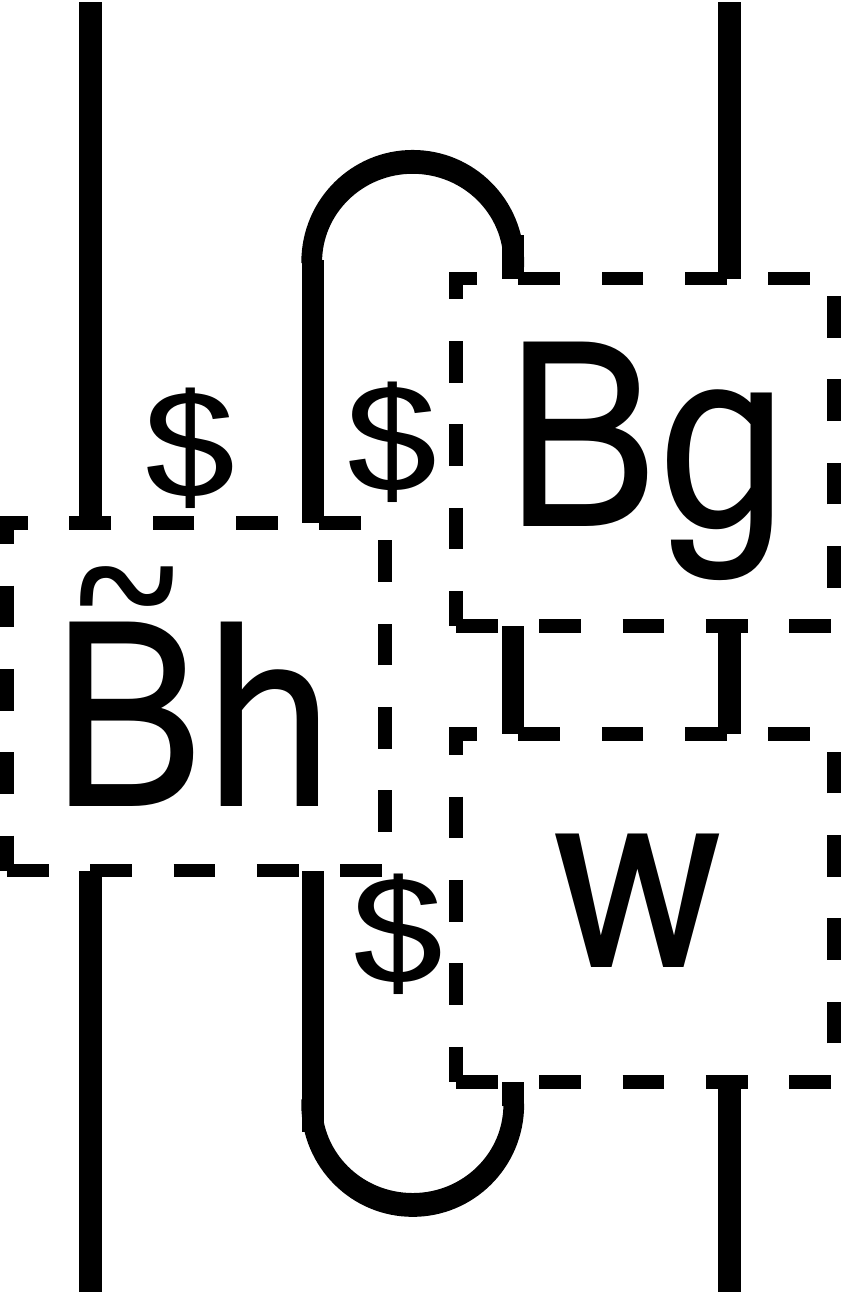}$$
is a bi-shift of the biprojection $B$.

By Lemma \ref{rangebipro}, a bi-shift of a biprojection obtains the minimal value of the Donoho-Stark uncertainty principle, so it is an extremal bi-partial isometry by Theorem \ref{minmain1}.

Furthermore, if $w$ is positive, then $w$ is identical to $Bg$ in the above argument.
\end{proof}

As a corollary, we obtain a new characterization of biprojections.
\begin{corollary}\label{positivebiprojection}
For a nonzero 2-box $x$ in an irreducible subfactor planar algebra. If $x$ and $\mathcal{F}(x)$ are positive operators and
$\mathcal{S}(x)\mathcal{S}(\mathcal{F}(x))= \delta^2$, then $x$ is a biprojection.
\end{corollary}

\begin{proof}
It follows from Theorem \ref{minmain1}, \ref{minmain2}.
\end{proof}

Note that the proof of Theorem \ref{minmain2} is independent of the choice of the form of a bi-shift of a biprojection. We refer the reader to the Appendix for the eight forms of bishifts of biprojections.

We have shown that an element is the minimizer of the uncertainty principles if and only if it is an extremal bi-partial isometry if and only if it is a bi-shift of a biprojection. Now let us prove the uniqueness of a bi-shift of a biprojection in the following sense.
Given a 2-box, we will obtain four range projections, the left, top, right, bottom ones. An extremal bi-partial isometry, or equivalently a bi-shift of a biprojection is determined by two adjacent range projections up to a scalar.

\begin{lemma}\label{uniquenessbiprojection}
Suppose $B$ is a biprojection in $\mathscr{P}_{2,+}$, and $\tilde{B}$ is the range projection of $\mathcal{F}(B)$.
For $x\in\mathscr{P}_{2,+}$, if $\mathcal{R}(x^*)=B$ and $\mathcal{R}(\mathcal{F}^{-1}(x))=\tilde{B}$, then $x$ is a multiple of $B$.
\end{lemma}

\begin{proof}
It is shown in \cite{BJ97a} that the biprojection $B$ can be expressed as a multiple of the (a,b-colored) Fuss-Catalan diagram
\grb{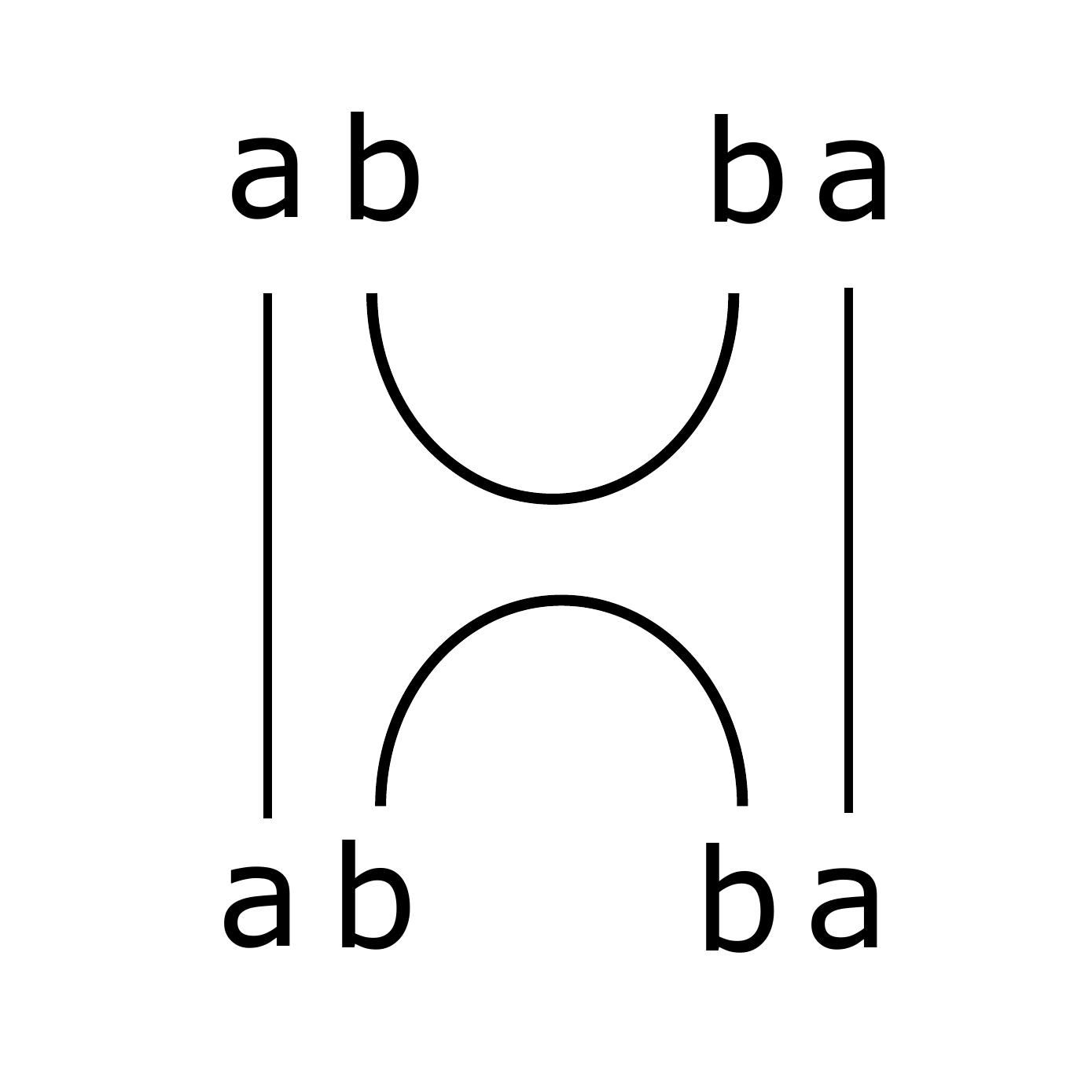}.
If $\mathcal{R}(x^*)=B$ and $\mathcal{R}(\mathcal{F}^{-1}(x))=\tilde{B}$,
then $x$ is a multiple of the Fuss-Catalan diagram
\grb{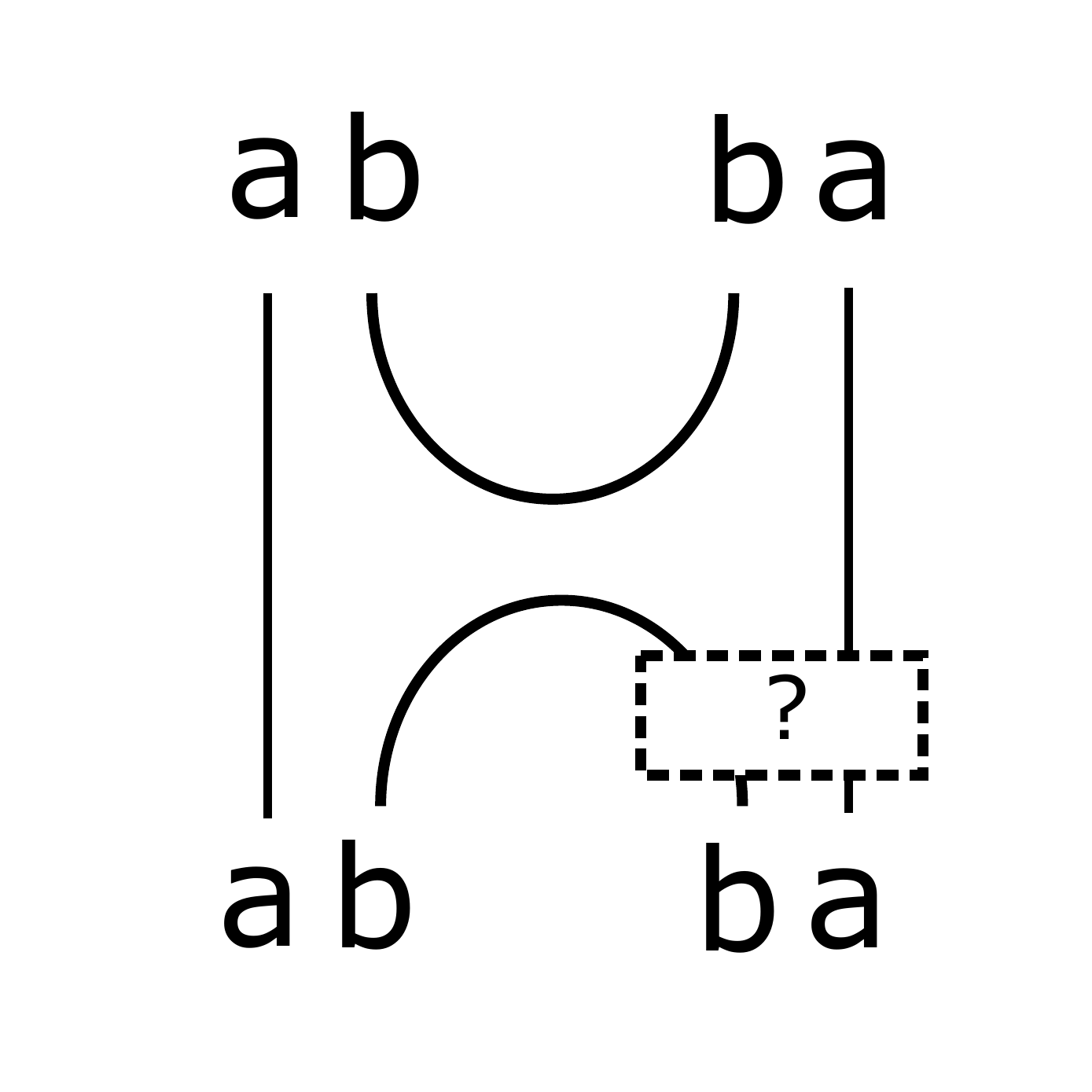}.
Recall that $\mathscr{P}$ is assumed to be irreducible, so
\grb{uniquebipro} is a multiple of \grb{bipro}.
Hence $x$ is a multiple of $B$.
\end{proof}

\begin{theorem}\label{uniqueness}
Suppose $B\in\mathscr{P}_{2,+}$ is a biprojection and $\widetilde{B}$ is the range projection of $\mathcal{F}(B)$. Take a right shift $Bg$ of $B$, and a right shift $\widetilde{B}h$ of $\widetilde{B}$. Then there is at most one element $x\in\mathscr{P}_{2,+}$ up to a scalar such that the range projection of $|x|$ is contained in $Bg$ and the range projection of $\mathcal{F}^{-1}(x)$ is contained in $\widetilde{B}h$.
\end{theorem}

\begin{proof}

If $x$ and $z$ are two nonzero elements of $\mathscr{P}_{2,+}$, such that the range projections $\mathcal{R}(|x|)$ and $\mathcal{R}(|z|)$ are contained in $Bg$ and the range projections $\mathcal{R}(\mathcal{F}^{-1}(x))$ and $\mathcal{R}(\mathcal{F}^{-1}(z))$ are contained in $\widetilde{B}h$, then by Theorem \ref{minmain1}, \ref{minmain2}, we have $\mathcal{R}(|x|)=\mathcal{R}(|z|)=Bg$ and $\mathcal{R}(\mathcal{F}^{-1}(x))=\mathcal{R}(\mathcal{F}^{-1}(z))=\widetilde{B}h$. Thus $zx^*$ and $xx^*$ are nonzero and $\mathcal{R}(|zx^*|)=\mathcal{R}(|xx^*|)=\mathcal{R}(x)$. By Lemma \ref{PQR=0}, we have
$$\mathcal{R}(\mathcal{F}^{-1}(zx^*))= \mathcal{R}(\overline{\mathcal{F}^{-1}(x)}^* * \mathcal{F}^{-1}(z))\leq  \mathcal{R}(\overline{\widetilde{B}h}*\widetilde{B}h).$$
By Theorem \ref{squarerelation},
$$\mathcal{S}(\overline{\widetilde{B}h}*\widetilde{B}h)=\mathcal{S}(\widetilde{Bh})=\mathcal{S}(\mathcal{F}^{-1}(x)).$$
Then
$$\mathcal{S}(zx^*)\mathcal{S}(\mathcal{F}^{-1}(zx^*))\leq\mathcal{S}(x)\mathcal{S}(\mathcal{F}^{-1}(x))=\delta^2.$$
By Theorem \ref{UnP}, we have
$$\mathcal{S}(zx^*)=\mathcal{S}(x); \quad \mathcal{S}(\mathcal{F}^{-1}(zx^*))=\mathcal{S}(\mathcal{F}^{-1}(x))=\mathcal{S}(\overline{\widetilde{B}h}*\widetilde{B}h).$$
Therefore
$$\mathcal{R}(zx^*)=\mathcal{R}(x);\quad \mathcal{R}(\mathcal{F}^{-1}(zx^*))=\mathcal{R}(\overline{\widetilde{B}h}*\widetilde{B}h).$$
By Theorem \ref{UnP}, \ref{minmain1}, we have that $zx^*$ is a bi-shift of a biprojection. Similarly $xx^*$ is a bi-shift of a biprojection. Moreover,
$$\mathcal{R}(|zx^*|)=\mathcal{R}(|xx^*|);~\mathcal{R}(\mathcal{F}^{-1}(zx^*))=\mathcal{R}(\mathcal{F}^{-1}(xx^*)).$$

By a similar argument, we have
$\overline{(xx^*)}^* * (zx^*)$ and $\overline{(xx^*)}^* * (xx^*)$ are bi-shifts of biprojections, and
$$\mathcal{R}(|\overline{(xx^*)}^* * (zx^*)|)=\mathcal{R}(|\overline{(xx^*)}^* * (xx^*)|);$$
$$\mathcal{R}(\mathcal{F}^{-1}(\overline{(xx^*)}^* * (zx^*)))=\mathcal{R}(\mathcal{F}^{-1}(\overline{(xx^*)}^* * (xx^*))).$$

By Theorem \ref{squarerelation}, $\overline{(xx^*)}^* * (xx^*)$ is a multiple of a biprojection, denoted by $Q$.
By Lemma \ref{uniquenessbiprojection}, $\overline{(xx^*)}^* * (zx^*)$ is also a multiple of $Q$.
Observe that both $z$ and $x$ are multiples of $((xx^*) * Q)x$, so $z$ is a multiple of $x$.
\end{proof}


\section{Uncertainty principles for $n$-boxes}\label{nbox}

In this section, we will prove the uncertainty principles for general cases, in particular for reducible subfactor planar algebras and $n$-boxes.
Suppose $\mathscr{P}$ is a subfactor planar algebra. For a projection $P\in\mathscr{P}_{n,+}$, we define
$P\mathscr{P}=\{Px|x\in \mathscr{P}_{n,+}\}$ to be the Hilbert subspace of $\mathscr{P}_{n,+}$.
We make a similar definition for a projection in $\mathscr{P}_{n,-}$.
For projections $P_i\in\mathscr{P}_{n_i,\pm}$, $i=1,2,3,4$, we define
$$P_i \otimes P_j=\gra{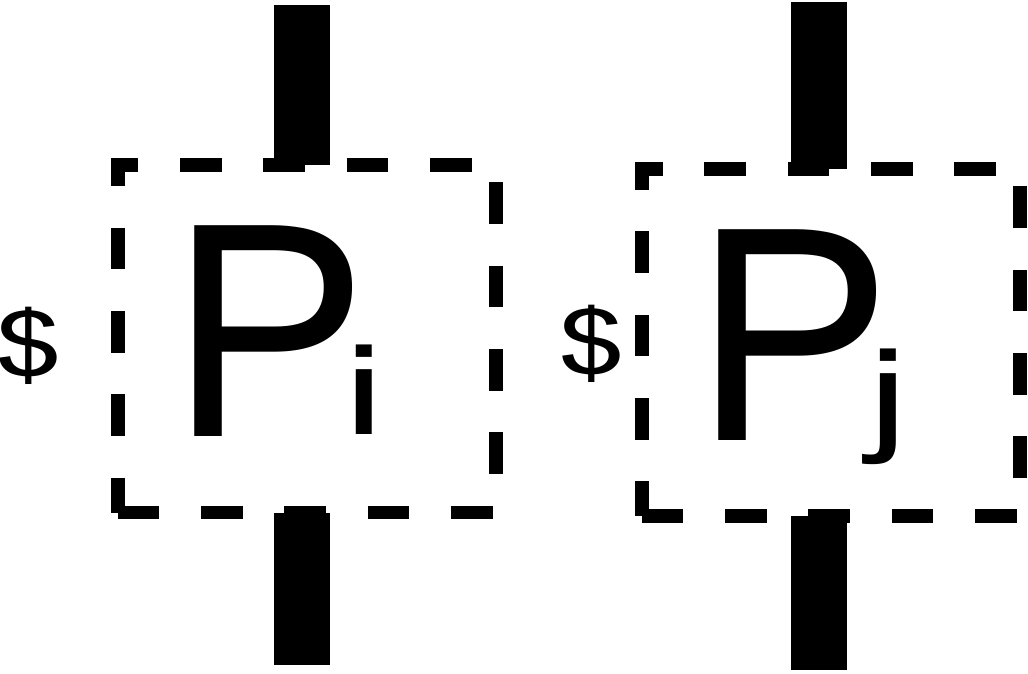} ~,$$
when the shadings match.
We define $\mathscr{P}_{P_1\otimes P_2}^{P_4 \otimes P_3}$
to be the space of all bounded operators from $(P_4\otimes P_3) \mathscr{P}$ to $(P_1\otimes P_2) \mathscr{P}$. An element $x$ in $\mathscr{P}_{P_1\otimes P_2}^{P_3 \otimes P_4}$ has the following form
$$\grb{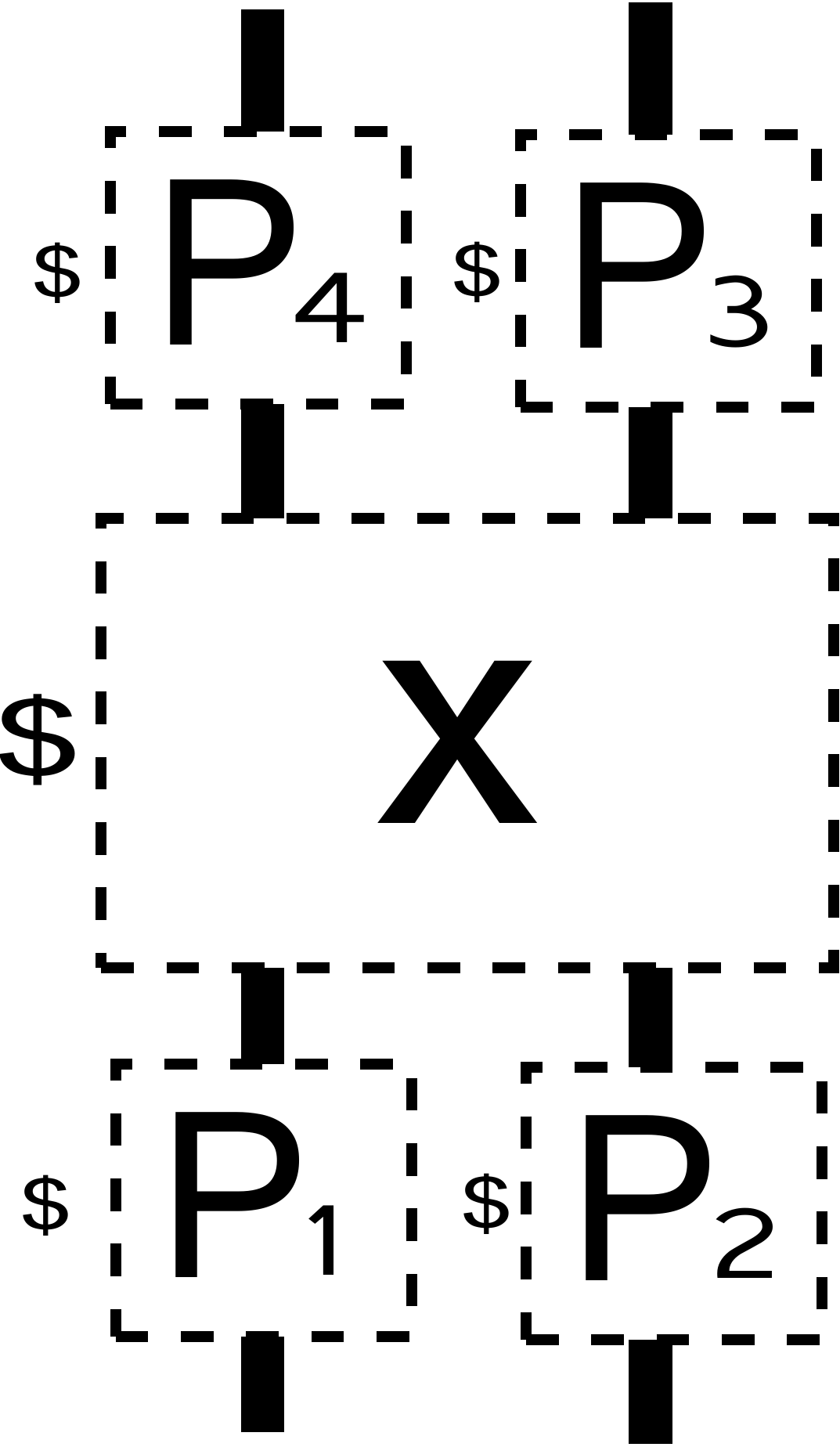}.$$

For an element in $\mathscr{P}_{P_1\otimes P_2}^{P_3 \otimes P_4}$, we also have the rank-one decomposition similar to the rank one decomposition in Section 2.

Note that $|x|=(x^*x)^{1/2}\in\mathscr{P}_{n_3+n_4,\pm}$.
Let us define $\|x\|_p=\||x|\|_p$, $1\leq p\leq \infty$; $\mathcal{S}(x)=\mathcal{S}(|x|)$.

\begin{proposition}
For any $x$ in $\mathscr{P}_{n,\pm}$, we have
\begin{itemize}
\item[(1)] $\|x\|_p=\|x^*\|_p=\|\overline{x}\|_p, 1\leq p\leq \infty;$

\item[(2)] $\mathcal{S}(x)=\mathcal{S}(x^*)=\mathcal{S}(\overline{x});$

\item[(3)] $H(|x|^2)=H(|x^*|^2)=H(|\overline{x}|^2).$
\end{itemize}
\end{proposition}

Let us define the Fourier transform $\mathcal{F}:\mathscr{P}_{P_1\otimes P_2}^{P_4 \otimes P_3}\rightarrow\mathscr{P}_{P_2\otimes \overline{P_3}}^{\overline{P_1} \otimes P_4}$ as follows
$$\mathcal{F}(\grb{1234})=\grb{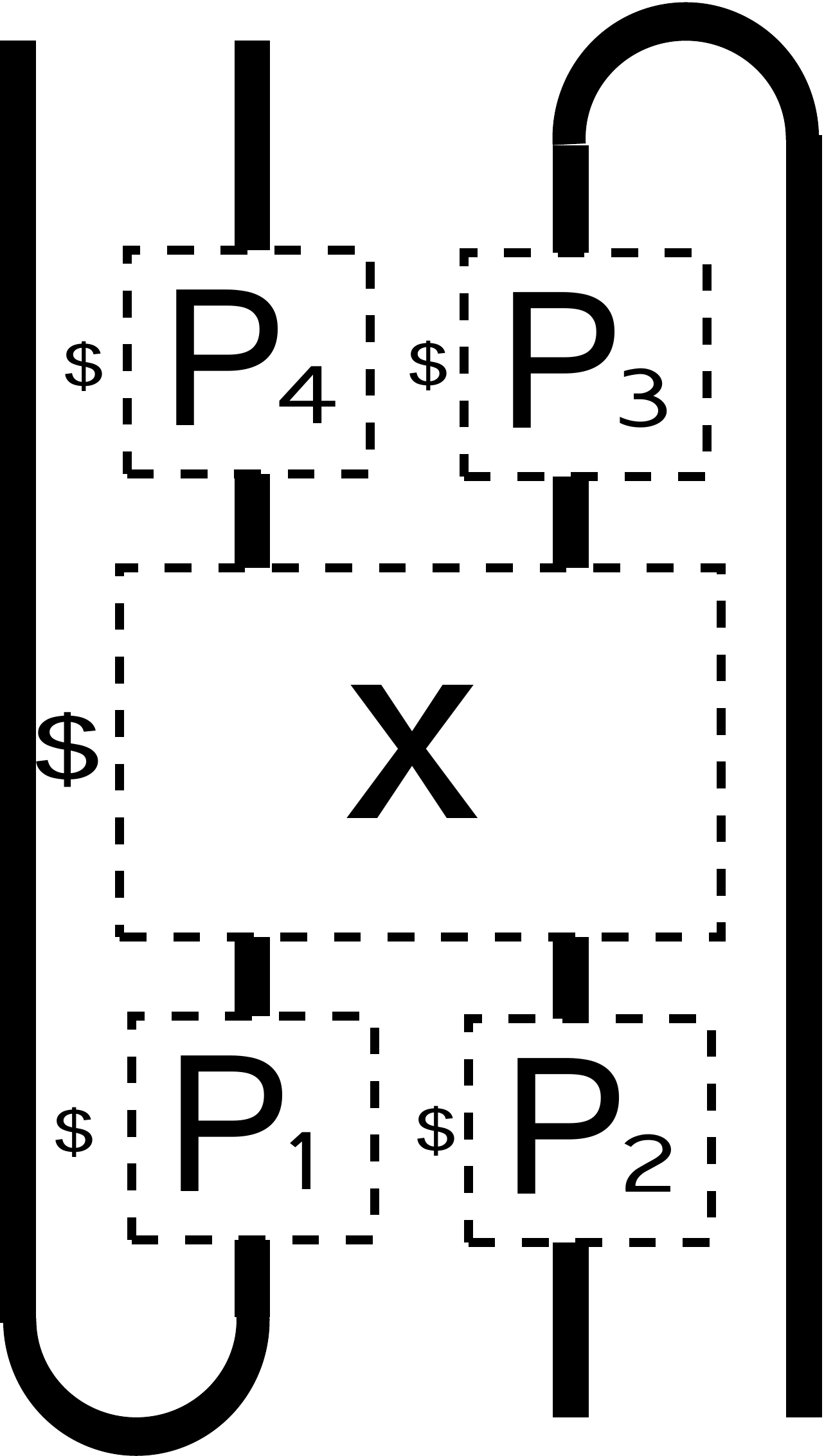}.$$

\begin{proposition}
For any $x$ in $\mathscr{P}_{P_1\otimes P_2}^{P_3 \otimes P_4}$, we have $\|\mathcal{F}(x)\|_2=\|x\|_2.$
\end{proposition}

\begin{proof}
It follows from the sphericality of subfactor planar algebras.
\end{proof}

In the rest of the section, we will write the diagram
$\grb{1234}$ as $\gra{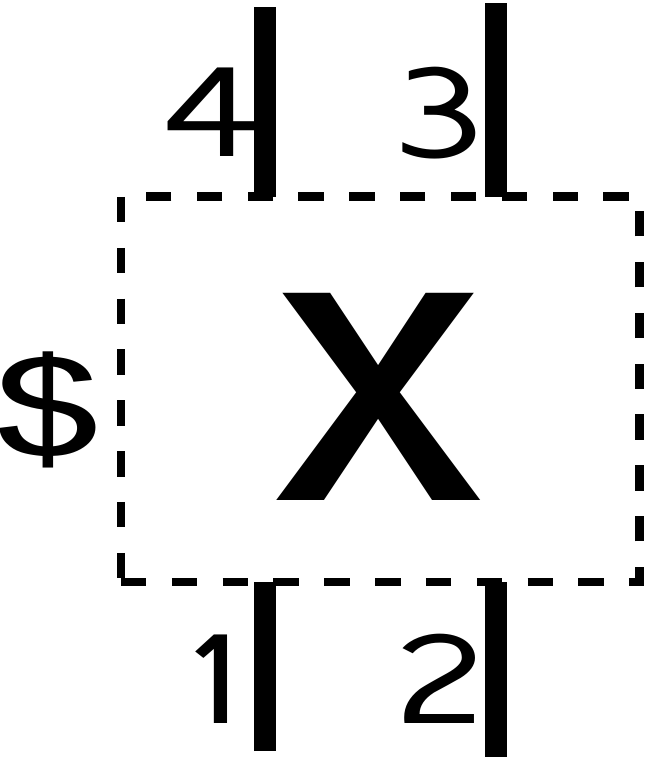}$, or simply as $\gra{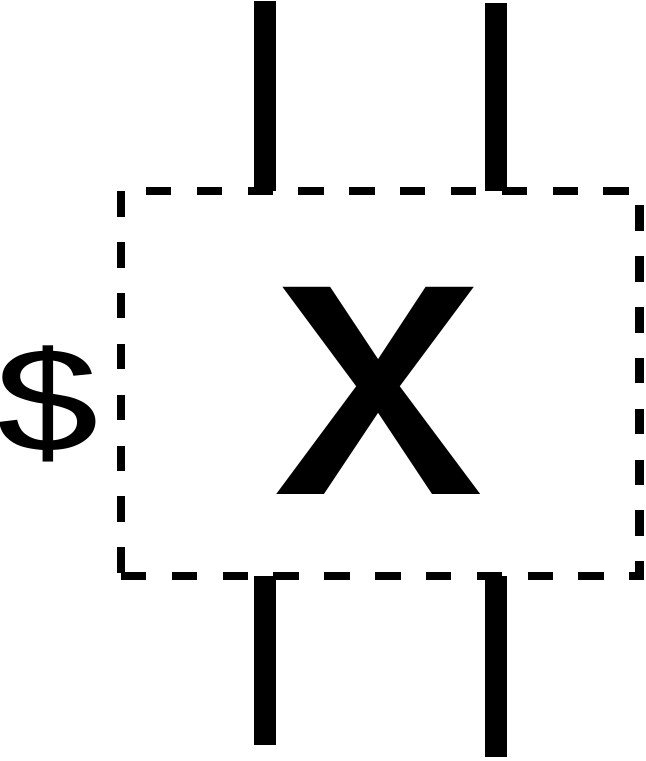}$ if there is no confusion.
By this notation, we can see that our method also works for a planar algebra with more kinds of regions and strings. The planar algebra should be $C^*$ and spherical. We will not give a formal definition here and skip the details.

Let us define
$$\delta_i=\min\{tr_{n_i}(P) | P \text{ is a nonzero subprojection of } P_i\}, 1\leq i \leq 4;$$
$$\delta_0=\max\{\sqrt{\delta_1\delta_3},\sqrt{\delta_2\delta_4}\}.$$

When the planar algebra is irreducible and $P_i$ is a through string, for $1\leq i \leq 4$, we have $\delta_i=\delta_0$, and $\mathscr{P}_{P_1\otimes P_2}^{P_4 \otimes P_3}$ is the 2-box space.

\begin{theorem}\label{young1g}
For $x\in \mathscr{P}_{P_1\otimes P_2}^{P_4 \otimes P_3}$, we have
$$\|\mathcal{F}(x)\|_p\leq \left(\frac{1}{\delta_0}\right)^{1-\frac{2}{p}}\|x\|_q,\quad 2\leq p<\infty,\frac{1}{p}+\frac{1}{q}=1.$$
\end{theorem}

The proof is almost identical to that of Theorem \ref{Young1}. Here we give a proof for the case $p=\infty,q=1$, similar to Proposition \ref{tr1}.

\begin{proof}
If $x$ is a rank-one partial isometry $v$ in $\mathscr{P}_{P_1\otimes P_2}^{P_4 \otimes P_3}$, we have
\begin{equation}\label{e1}
\grb{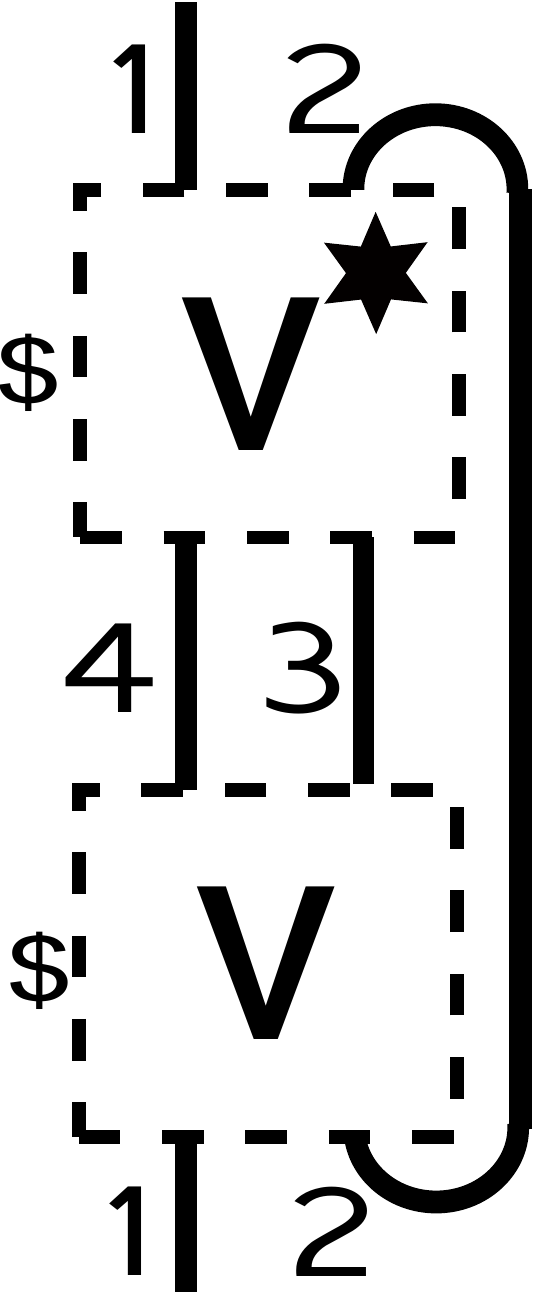}\leq \frac{\|v\|_1}{\delta_1} \quad \grb{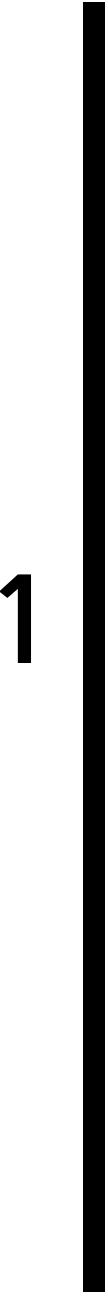}.
\end{equation}
Then
$$\left(\grb{inev1}\right)^2\leq \frac{\|v\|_1}{\delta_1} \grb{inev1}.$$
Note that $\grb{inev1}$ is a positive operator, and its range projection is contained in the projection $\grb{inev2}$, we obtain
$$ \grb{inev1} \leq \frac{\|v\|_1}{\delta_1} \grb{inev2}.$$
Adding a cap to the left side, we have
$$ \grb{inev3} \leq \frac{\|v\|_1}{\delta_1} \grb{inev4}.$$
Note that
\begin{equation}\label{e3}
\graa{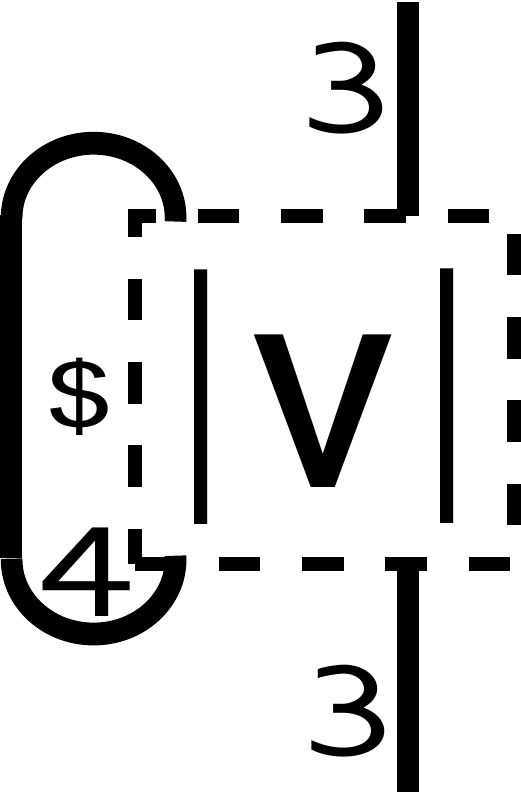}\leq \frac{\|v\|_1}{\delta_3} \quad \graa{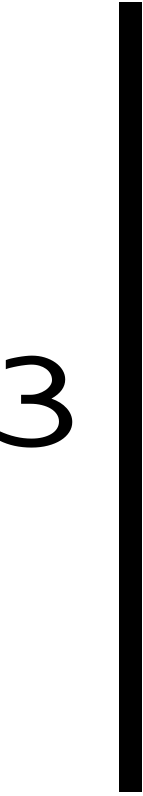}.
\end{equation}
So
$$\grb{inev3} \leq \frac{\|v\|_1^2}{\delta_1\delta_3} \grb{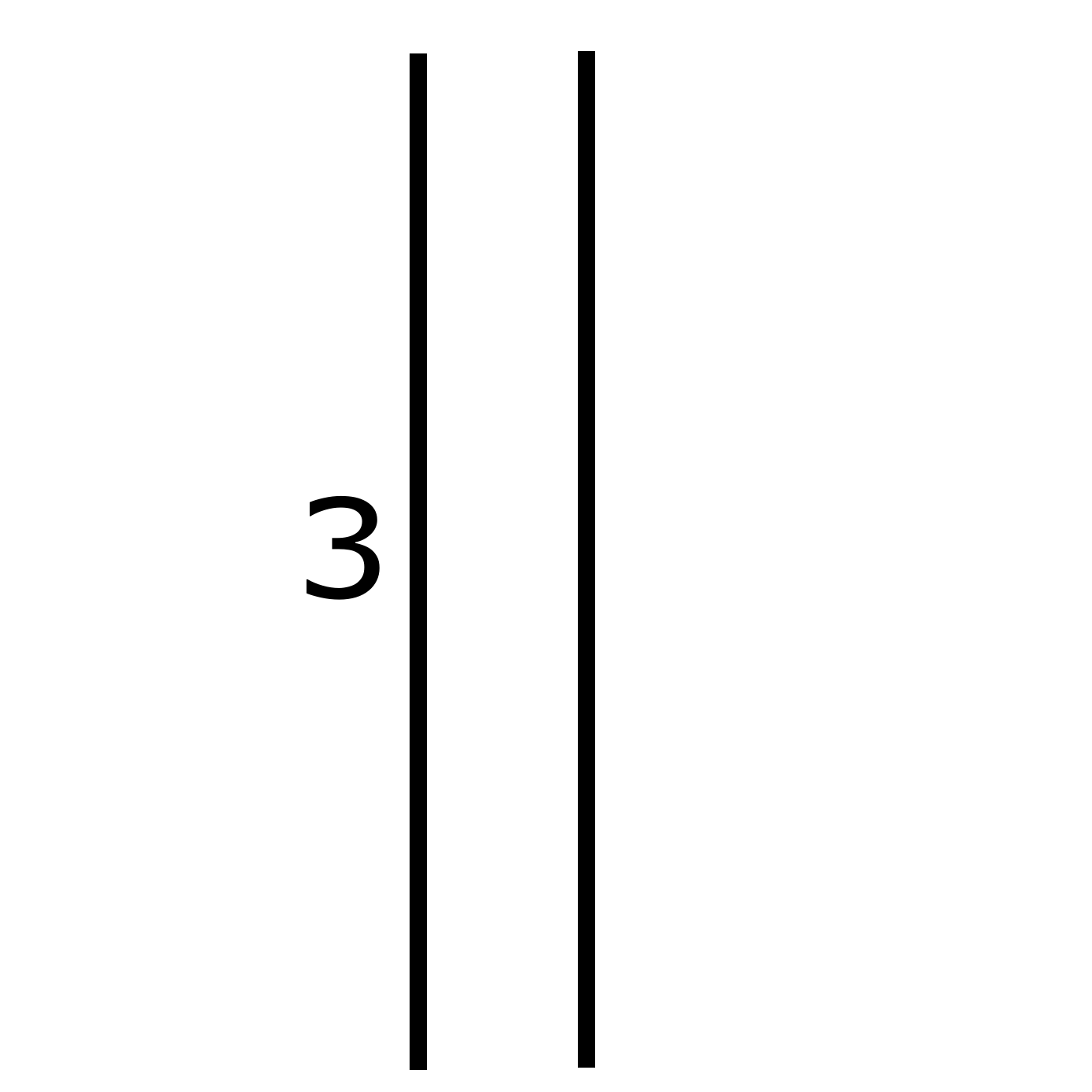}.$$
That is,
$$(\mathcal{F}^{-1}(v))^*(\mathcal{F}^{-1}(v))\leq (\frac{\|v\|_1^2}{\delta_1\delta_3}) 1.$$
Taking the norm on both sides, we have
$$\|(\mathcal{F}^{-1}(v))^*(\mathcal{F}^{-1}(v))\|_\infty\leq \frac{\|v\|_1^2}{\delta_1\delta_3}.$$
Hence
$$\|\mathcal{F}^{-1}(v)\|_\infty\leq \frac{\|v\|_1}{\sqrt{\delta_1\delta_3}}.$$
Recall that $\mathcal{F}(v)=(\mathcal{F}^{-1}(v^*))^*$ and $\|v\|_1=\|v^*\|_1$, we obtain that
$$\|\mathcal{F}(v)\|_\infty\leq \frac{\|v\|_1}{\sqrt{\delta_1\delta_3}}.$$
By symmetry, we have
$$\|\mathcal{F}(v)\|_\infty\leq \frac{\|v\|_1}{\sqrt{\delta_2\delta_4}}.$$
Therefore
$$\|\mathcal{F}(v)\|_\infty\leq \frac{\|v\|_1}{\delta_0},$$
where $\displaystyle \delta_0=\max\{\sqrt{\delta_1\delta_3},\sqrt{\delta_2\delta_4}\}.$

If $x$ is arbitrary in $\mathscr{P}_{2,\pm}$, let $x=\sum_k\lambda_k v_k$ be the rank-one decomposition, then $\|x\|_1=\sum_k\lambda_k \|v_k\|_1$.
We have proved that $\displaystyle \|\mathcal{F}(v_k)\|_\infty\leq \frac{\|v_k\|_1}{\delta_0}$ for the rank-one partial isometry $v_k$, so
\begin{equation}\label{e5}
\|\mathcal{F}(x)\|_\infty\leq \sum_k \lambda_k\|\mathcal{F}(v_k)\|_\infty \leq \sum_k \lambda_k\frac{\|v_k\|_1}{\delta_0}=\frac{\|x\|_1}{\delta_0}.
\end{equation}
\end{proof}

\begin{notation}
We call this general version of 2-box $x$ extremal if it satisfies the equality
$$\|\mathcal{F}(x)\|_\infty =\frac{1}{\delta_0}\|x\|_1$$
for this general Fourier transform.
\end{notation}

\begin{lemma}\label{nonir}
Suppose $x\in \mathscr{P}_{P_1\otimes P_2}^{P_4 \otimes P_3}$ is extremal. If $\delta_0=\sqrt{\delta_1\delta_3}$, then there are subprojections $Q_i$ of $P_i$ with trace $\delta_i$, for $i=1,3$, such that
$x=(Q_1\otimes P_2)x(P_4 \otimes Q_3)$;
If $\delta_0=\sqrt{\delta_2\delta_4}$, then there are subprojections $Q_i$ of $P_i$ with trace $\delta_i$, for $i=2,4$, such that
$x=(P_1\otimes Q_2)x(Q_4 \otimes P_3)$.
\end{lemma}

\begin{proof}
Suppose $x$ is extremal.
When $\delta_0=\sqrt{\delta_1\delta_3}$, if $x$ is a rank-one partial isometry, then Equation (\ref{e1}), (\ref{e3}) hold in the proof of \ref{young1g}.
So there are subprojections $Q_i$ of $P_i$ with trace $\delta_i$, for $i=1,3$, such that
$x=(Q_1\otimes P_2)x(P_4 \otimes Q_3)$.

If $x$ is arbitrary, let $x=\sum_k\lambda_k v_k$ be the rank-one decomposition, then each $v_k$ is a rank-one extremal partial isometry.
Without loss of generality, we take $v_1$ and $v_2$, then there are subprojections $Q_i,Q_i'$ of $P_i$ with trace $\delta_i$, for $i=1,3$,
such that $v_1=(Q_1\otimes 1)v_1(1 \otimes Q_3)$; $v_2=(Q_1'\otimes 1)v_2(1 \otimes Q_3')$. It is enough to show that $Q_i=Q_i'$.
Since $x$ is extremal, from Equation (\ref{e5}) in the proof of \ref{young1g}, we have
$$\|\mathcal{F}(v_1)+\mathcal{F}(v_2)\|_\infty=\|\mathcal{F}(v_1)\|_\infty+\|\mathcal{F}(v_2)\|_\infty.$$
Taking the square on both sides,
$$\|(\mathcal{F}(v_1)+\mathcal{F}(v_2))^*(\mathcal{F}(v_1)+\mathcal{F}(v_2))\|_\infty=(\|\mathcal{F}(v_1)\|_\infty+\|\mathcal{F}(v_2)\|_\infty)^2.$$
So
$$\||\mathcal{F}(v_1)|^2+|\mathcal{F}(v_2)|^2\|_\infty=\|\mathcal{F}(v_1)\|_\infty^2+\|\mathcal{F}(v_1)\|_\infty^2.$$
Then
$$\|\|\mathcal{F}(v_1)\|_\infty^2 (\overline{Q_1}\otimes P_4)+\|\mathcal{F}(v_2)\|_\infty^2(\overline{Q_1'}\otimes P_4)\|_\infty\geq\|\mathcal{F}(v_1)\|_\infty^2+\|\mathcal{F}(v_1)\|_\infty^2.$$
So
$$\|\|\mathcal{F}(v_1)\|_\infty^2\overline{Q_1}+\|\mathcal{F}(v_2)\|_\infty^2\overline{Q_1'}\|_\infty\geq\|\mathcal{F}(v_1)\|_\infty^2+\|\mathcal{F}(v_1)\|_\infty^2.$$
That implies $Q_1=Q_1'$. Similarly $Q_3=Q_3'$.

The proof is similar if $\delta_0=\sqrt{\delta_2\delta_4}$.
\end{proof}

For $x\in \mathscr{P}_{P_1\otimes P_2}^{P_4 \otimes P_3}, y\in \mathscr{P}_{\overline{P_2}\otimes P_1'}^{\overline{P_3} \otimes P_4'}$, let us define the coproduct $*_{P_2}^{P_3}$, simply denoted by $*$, as follows
$$x*y=\grb{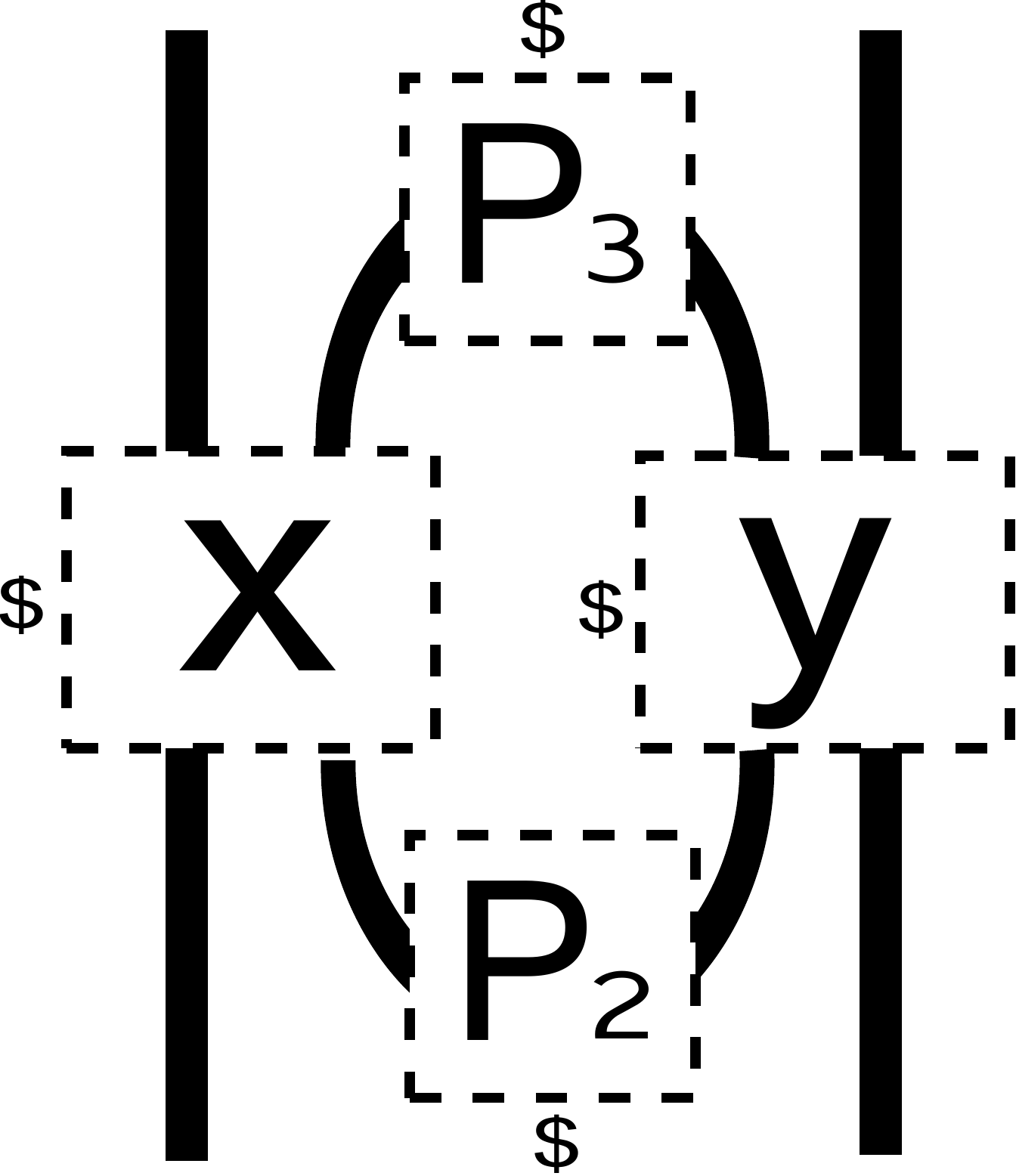}.$$

By a similar argument in Section \ref{Inequ} and the technique in the proof of Theorem \ref{young1g}, we have the following Young's inequality.
\begin{theorem}\label{UPnbox}

For any $x$ in $\mathscr{P}_{P_1\otimes P_2}^{P_4 \otimes P_3}$, $y$ in $\mathscr{P}_{\overline{P_2}\otimes P_1'}^{\overline{P_3} \otimes P_4'}$, we have
$$\|x*y\|_r\leq\left(\frac{1}{\sqrt{\delta_2\delta_3}}\right)^{\frac{1}{r}} \left(\frac{1}{\sqrt{\delta_1'\delta_4'}}\right)^{1-\frac{1}{q}} \left(\frac{1}{\sqrt{\delta_1\delta_4}}\right)^{1-\frac{1}{p}}\|x\|_q\|y\|_p,$$
where $\frac{1}{p}+\frac{1}{q}=\frac{1}{r}+1$, and $\delta_1',\delta_4'$ are defined in a similar way for $P_1', P_4'$ respectively.
\end{theorem}

This is Main Theorem \ref{mainthm3}.

Repeating the arguments in Section \ref{SectionNUP}, \ref{SectionMin}, we have the following results.
\begin{theorem}\label{minmain3}
For a nonzero element $x$ in $\mathscr{P}_{P_1\otimes P_2}^{P_3 \otimes P_4}$,
\begin{itemize}
\item[(1)] $\mathcal{S}(x)\mathcal{S}(\mathcal{F}(x))\geq \delta_0^2;$
\item[(2)] $H(|x|^2)+H(|\mathcal{F}(x)|^2)\geq \|x\|_2(2\log\delta_0-4\log\|x\|_2);$
\end{itemize}

Moreover, the following are equivalent

\begin{itemize}
\item[(1)] $\mathcal{S}(x)\mathcal{S}(\mathcal{F}(x))= \delta_0^2;$
\item[(2)] $H(|x|^2)+H(|\mathcal{F}(x)|^2)=\|x\|_2(2\log\delta_0-4\log\|x\|_2);$
\item[(3)] $x$ is an extremal bi-partial isometry;
\end{itemize}

\end{theorem}

Furthermore, if $\delta_1\delta_3=\delta_2\delta_4$, then by Lemma \ref{nonir}, the projection $P_i$ on the boundary of an extremal bi-partial isometry could be replaced by a minimal projection $Q_i$ with trace $\delta_i$, for $1\leq i\leq 4$.
Then the element $B$ constructed in Theorem \ref{minmain2} is a 2-box in an irreducible subfactor planar algebra, so the definition and properties of the biprojeciton $B$ are inherited.
By a similar argument in Section \ref{SectionMin}, we have the following results.

\begin{theorem}
With above notations and assumptions,
for a nonzero element $x$ in $\mathscr{P}_{P_1\otimes P_2}^{P_3 \otimes P_4}$, $\delta_1\delta_3=\delta_2\delta_4$, the following statements are equivalent:

\begin{itemize}
\item[(1)] $\mathcal{S}(x)\mathcal{S}(\mathcal{F}(x))= \delta_0^2;$
\item[(2)] $H(|x|^2)+H(|\mathcal{F}(x)|^2)=\|x\|_2(2\log\delta_0-4\log\|x\|_2);$
\item[(3)] $x$ is an extremal bi-partial isometry;
\item[(3')] $x$ is a partial isometry and $\mathcal{F}^{-1}(x)$ is extremal;
\item[(4)] $x$ is a bi-shift of a biprojection.
\end{itemize}

\end{theorem}

\begin{theorem}
Suppose $\mathscr{P}$ is an irreducible subfactor planar algebra.
For a nonzero element $x\in\mathscr{P}_{n,\pm}$, we have
$$\prod_{k=0}^{n-1} \mathcal{S}(\mathcal{F}^k(x))\geq\delta^n;$$
$$\sum_{k=0}^{n-1} H(|\mathcal{F}^k(x)|^2)\geq \|x\|_2(n\log\delta-2n\log\|x\|_2).$$
\end{theorem}

\begin{proof}
Considering $P_1,P_3$ as one through string, $P_2, P_4$ as $n-1$ through strings in Theorem \ref{minmain3},
then $\delta_0\geq\sqrt{\delta_1\delta_3}=\delta$ and $\mathcal{S}(x)\mathcal{S}(\mathcal{F}(x))\geq \delta^2$.
So
$$\prod_{k=0}^{2n-1} \mathcal{S}(\mathcal{F}^k(x))\geq\delta^{2n}.$$
Note that $\mathcal{S}(\mathcal{F}^{n+k}(x))=\mathcal{S}(\mathcal{F}^k(x))$, so
$$\prod_{k=0}^{n-1} \mathcal{S}(\mathcal{F}^k(x))\geq\delta^n.$$
Similarly we have
$$\sum_{k=0}^{n-1} H(|\mathcal{F}^k(x)|^2)\geq \|x\|_2(n\log\delta-2n\log\|x\|_2).$$
\end{proof}

The equalities of the above inequalities are obtained on the Temperley-Lieb diagram with caps around the boundary, for example the 5-box
$\grs{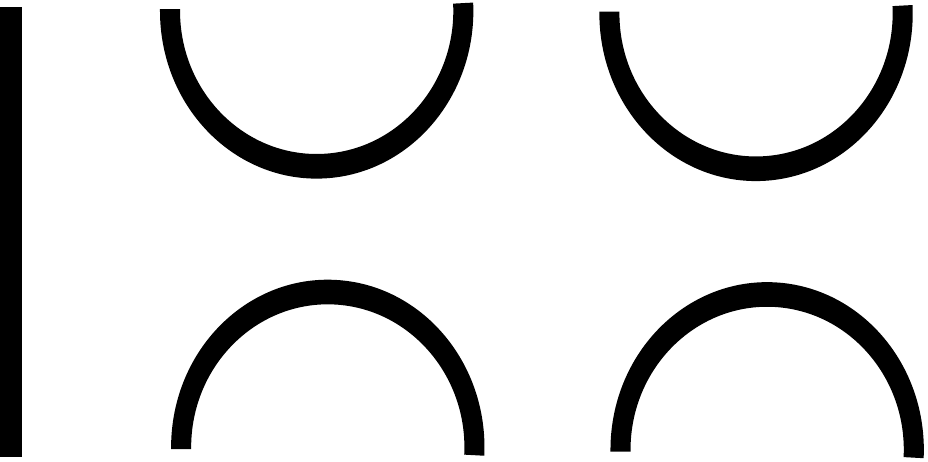}$.
If there is a biprojection, then the equality is also obtained on the Fuss-Catalan diagram with caps around.

\begin{definition}
Suppose $\mathscr{P}$ is an irreducible subfactor planar algebra.
A projection $P$ in $\mathscr{P}_{n,\pm}$ is called an $n$-projection if $\mathcal{F}^k(P)$ is a multiple of a projection, for $0\leq k\leq n-1$.
An element $x$ in $\mathscr{P}_{n,\pm}$ is called an extremal $n$-partial isometry if
$\mathcal{F}^k(x)$ is a multiple of a partial isometry and
$\displaystyle \|\mathcal{F}^{k \pm 1}(x)\|=\frac{\|\mathcal{F}^k(x)\|_1}{\delta}$, for $0\leq k\leq n-1$.
\end{definition}

\begin{question}
Whether all $n$-projections are minimizers of the uncertainty principle?
Whether all minimizers are extremal $n$-partial isometries?
Whether all minimizers are "$n$-shifts of $n$-projections"?
\end{question}

\section{Application to Group Algebras}\label{group}

Suppose $\mathcal{M}=\mathcal{N}\rtimes G$, for an outer action a finite group $G$.
Let $\mathscr{P}$ be the subfactor planar algebra of the subfactor $\mathcal{N}\subset\mathcal{M}$.
Then $\mathscr{P}_{2,+}$ is the C$^*$ algebra of functions over $G$ and $\mathscr{P}_{2,-}$ is the group C$^*$ algebra of $G$.
The value of a closed circle $\delta$ is $\sqrt{|G|}$.
The trace $tr_2$ of $\mathscr{P}_{2,+}$ is the integral of functions on $G$ over the counting measure.
The trace $tr_2$ of $\mathscr{P}_{2,-}$ is the trace of the group algebra of $G$ on its left regular representation.
For a group element $g\in G$,
the Fourier transform from $\mathscr{P}_{2,\mp}$ to $\mathscr{P}_{2,\pm}$ maps the group element $g$ to the function $\sqrt{|G|}\delta_g$; the function $\sqrt{|G|}\delta_g$ to the group element $g^{-1}$, where $\delta_g$ is the Dirac function at $g$.
The coproduct of $\mathscr{P}_{2,+}$ is a multiple of the convolution. Precisely for two functions $f_1,f_2: G\rightarrow \mathbb{C}$,
$$f_1*f_2(g_0)=\frac{1}{\sqrt{|G|}}\sum_{g\in G}f_1(g_0g^{-1})f_2(g).$$
Theorem \ref{Young2} becomes Young's inequality for the convolution of groups.
The Plancherel formula is a spherical isotopy.

\begin{proposition}\label{mingroup}
Suppose $G$ is a finite group. Take a subgroup $H$, a one dimensional representation $\chi$ of $H$, an element $g\in G$, a nonzero constant $c\in\mathbb{C}$. Then
$$x=c\sum_{h\in H}\chi(h)hg$$
is a bi-shift of a biprojection. Conversely any bi-shift of a biprojection is of this form.
\end{proposition}

\begin{proof}
For a subgroup $H$, the corresponding biprojection is $B=\sum_{h\in H}\delta_h$.
The range projection $\widetilde{B}$ of $\mathcal{F}(B)$ is $\displaystyle \frac{1}{|H|}\sum_{h\in H}h$.
By a direct computation, we have that a right shift of $B=\sum_{h\in H}\delta_h$ is $Bg=\sum_{h\in H}\delta_{hg}$ for some $g\in G$.
Let $Q$ be a right shift of $\widetilde{B}$. Then $Q$ is an element in the group algebra of $H$, and it minimizes the uncertainty principle.
Note that $\widetilde{B}$ is a central minimal projection in the group algebra of $H$, so its right shift $Q$ is also a rank-one projection. Then
$$\frac{tr_2(hQ)}{tr_2{(Q)}}, h\in H,$$ is a one dimensional representation of $H$. Take $\chi$ to be its contragredient representation, then $\displaystyle \chi(h)=\frac{tr_2(h^{-1}Q)}{tr_2{Q}}$, and $\displaystyle Q=\frac{1}{|H|}\sum_{h\in H}\chi(h)h$.
Take $x=\sum_{h\in H}\chi(h)hg$. Then $\mathcal{R}(x)=Q$ and  $\mathcal{R}(\mathcal{F}(x))=Bg$. So $x$ is a bi-shift of a biprojection.
By Theorem \ref{uniqueness}, any bi-shift of a biprojection is of this form.
\end{proof}

\begin{remark}
Note that $\chi$ is the pull back of a character of $H/[H,H]$, where $[H,H]$ is the commutator subgroup.
\end{remark}

By Theorem \ref{squarerelation}, \ref{minmain2}, we have the following corollary.

\begin{corollary}
Take a finite group $G$ acting on its left regular representation, a subset $S$ of $G$, an operator $x=\sum_{g\in S} \omega_g g$, $|\omega_g|=1$, $\forall$ $g$. Then

(1) $x$ is a bi-shift of a biprojection $\iff$ $x$ is extremal.

Furthermore, if $x=\sum_{g\in S} g$, then

(2) $S$ is a coset $\iff$ $x$ is extremal $\iff$ $\|x\|_1=|S|$.

(3) $S$ is a subgroup $\iff$ $x$ is positive.
\end{corollary}

\section{Applications to Group Actions}

The spin model \cite{Jon} of an $n$-dimensional vector space $V$ with an orthonormal basis $S$ is a subfactor planar algebra, denoted by $Spin$.
Its $2$-box space $Spin_{2,+}$ is the $C^*$ algebra $End(V)$; $Spin_{2,-}$ is the $C^*$ algebra of functions on the set $S\times S$.
For a matrix $A=(a_{ij})_{i,j\in S}$ in $Spin_{2,+}$, $tr_2(A)=\sum_{i\in S} a_{ii}$;
For a function $f(i,j)=f_{ij}$ in $Spin_{2,-}$, $tr_2(f)=\frac{1}{n}\sum_{i,j\in S} f_{ij}$.
The Fourier transform $\mathcal{F}:Spin_{2,\pm}\rightarrow Spin_{2,\mp}$ is
$$\mathcal{F}(A)(i,j)=\sqrt{n}a_{ij}; ~\mathcal{F}(f)=(\frac{f_{ji}}{\sqrt{n}})_{i,j\in S}.$$
Let $A=(a_{ij}),B=(b_{ij})$ be $n\times n$ matrices. The Hadamard product of $A$ and $B$ is the matrix $C=(c_{ij})$ (denoted by $A\circ B$) given by
$c_{ij}=a_{ij}a_{ij}$.
The coproduct $A*B$ is realised as $\sqrt{n}A\circ B$.

If there is a group action $G$ on $S$, then the fixed point algebra of $Spin$ under the induced group action of $G$ is also a subfactor planar algebra, denoted by $\mathscr{P}$. Moreover,
$\mathscr{P}_{2,+}$ consists of $S\times S$ matrices commuting with the action of $G$. Let $\delta_0$ be the minimal value of the trace of a nonzero minimal projection in $\mathscr{P}_{1,\pm}$. Then $\displaystyle \delta_0=\frac{n_0}{\sqrt{n}}$, where $n_0$ is the minimal cardinality of $G$-orbits of $S$.
Then we obtain the Hausdorff-Young inequality, Young's inequality, the uncertainty principles, and the characterizations of minimizers for elements in $\mathscr{P}_{2,+}$. For example, we have the following inequality for Hadamard product.

\begin{corollary}\label{youngmatrix}
If $A,B$ are two $S\times S$ matrixes commuting with a group action of $G$ on $S$, then we have
$$\|A\circ B\|_r\leq \frac{1}{n_0^2}\|A\|_p\|B\|_q,$$
where $\frac{1}{p}+\frac{1}{q}=\frac{1}{r}+1$ and $n_0$ is the minimal value of the number of elements in a orbit of $S$ under the action of $G$.
\end{corollary}

If the group action is transitive, then $\mathscr{P}$ is irreducible and it could be viewed as a group subgroup subfactor planar algebra. If $S=G$, and the action of $G$ on $G$ is the group multiplication, then $\mathscr{P}$ becomes the group subfactor planar algebra of $G$.

\section*{Appendix}
Suppose $B$ is a biprojection, and $\tilde{B}$ is the range projection of $\mathcal{F}(B)$.
Let $gB$, $Bg$ be a left and a right shift of $B$ respectively.
Let $h\tilde{B}$, $\tilde{B}h$ be a left and a right shift of $\tilde{B}$ respectively.

The point to construct a bi-shift of the biprojection $B$ is to make sure the pair of biprojections $B$ and $\widetilde{B}$ are adjacent to the same corner of the 2-box.
There are eight forms of bi-shifts of biprojections as follows,

\begin{center}
\grc{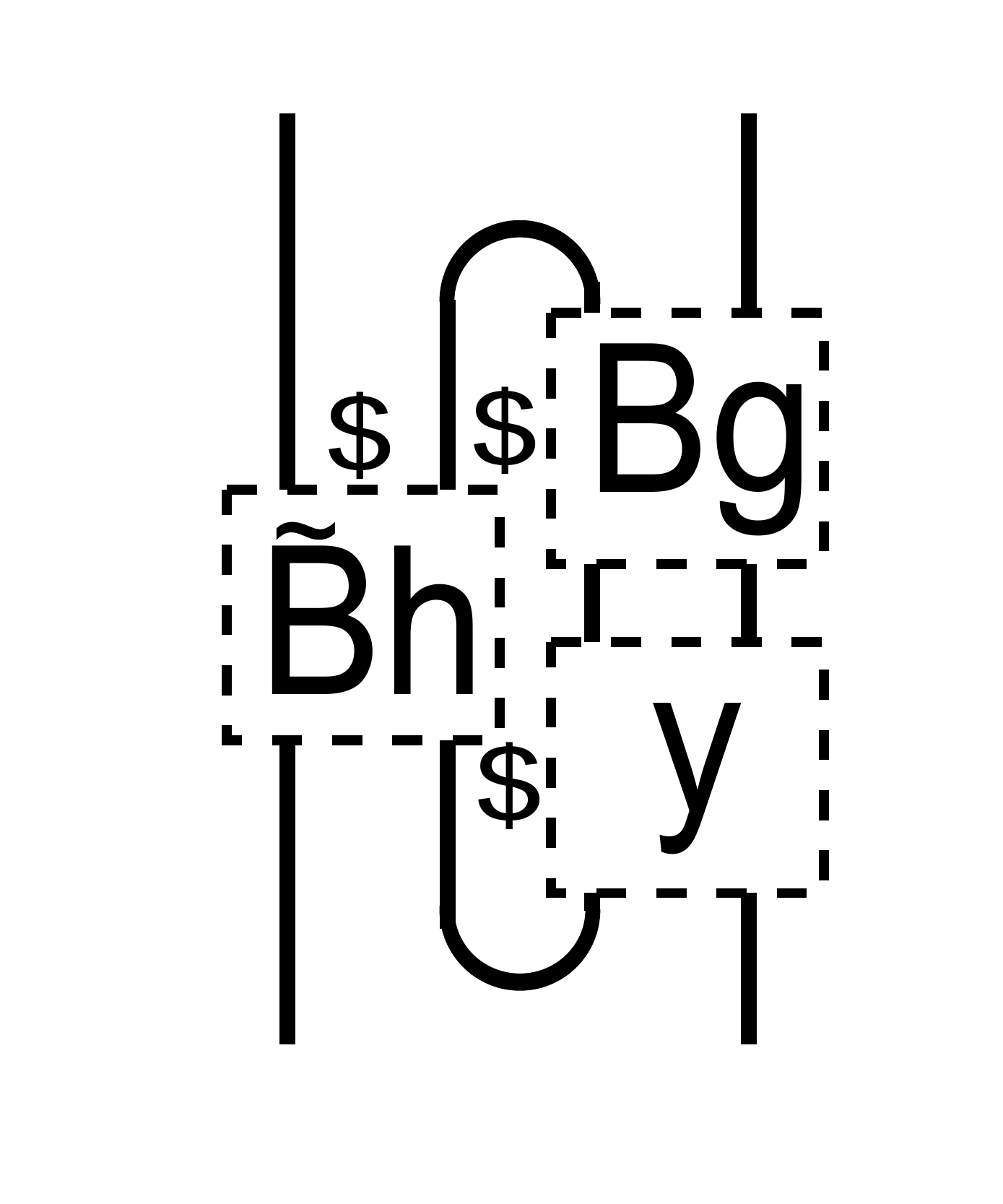} \grc{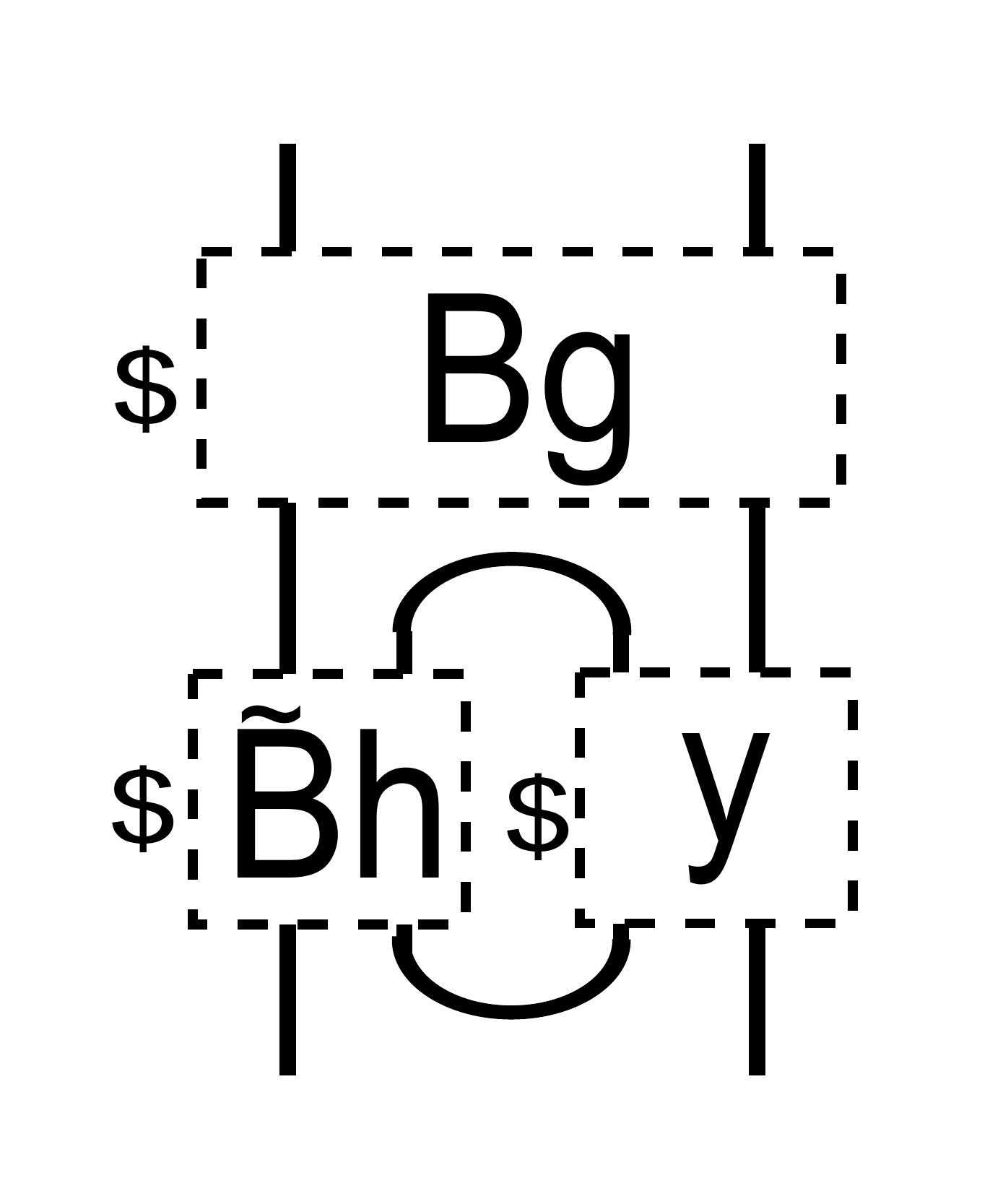} \grc{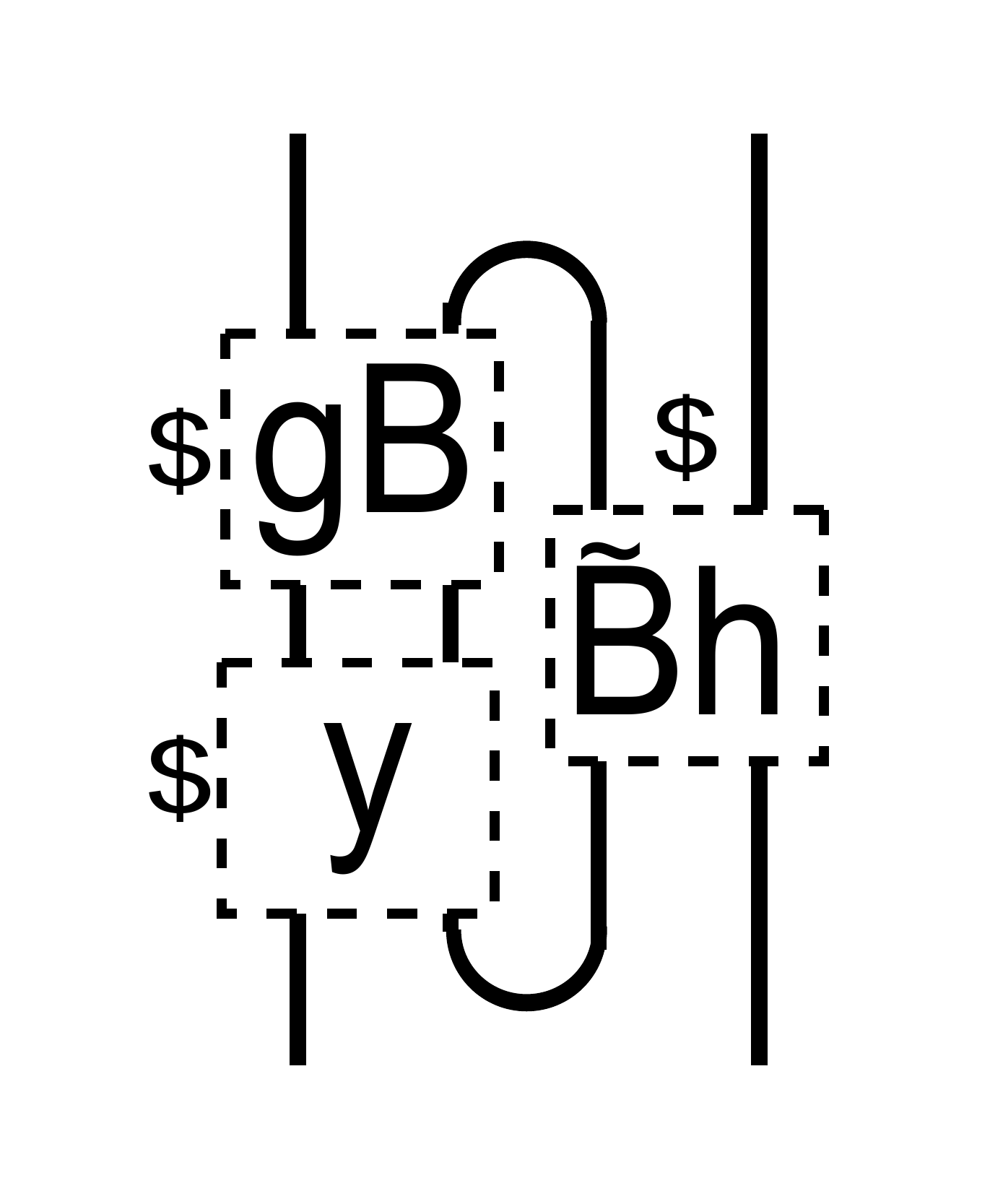} \grc{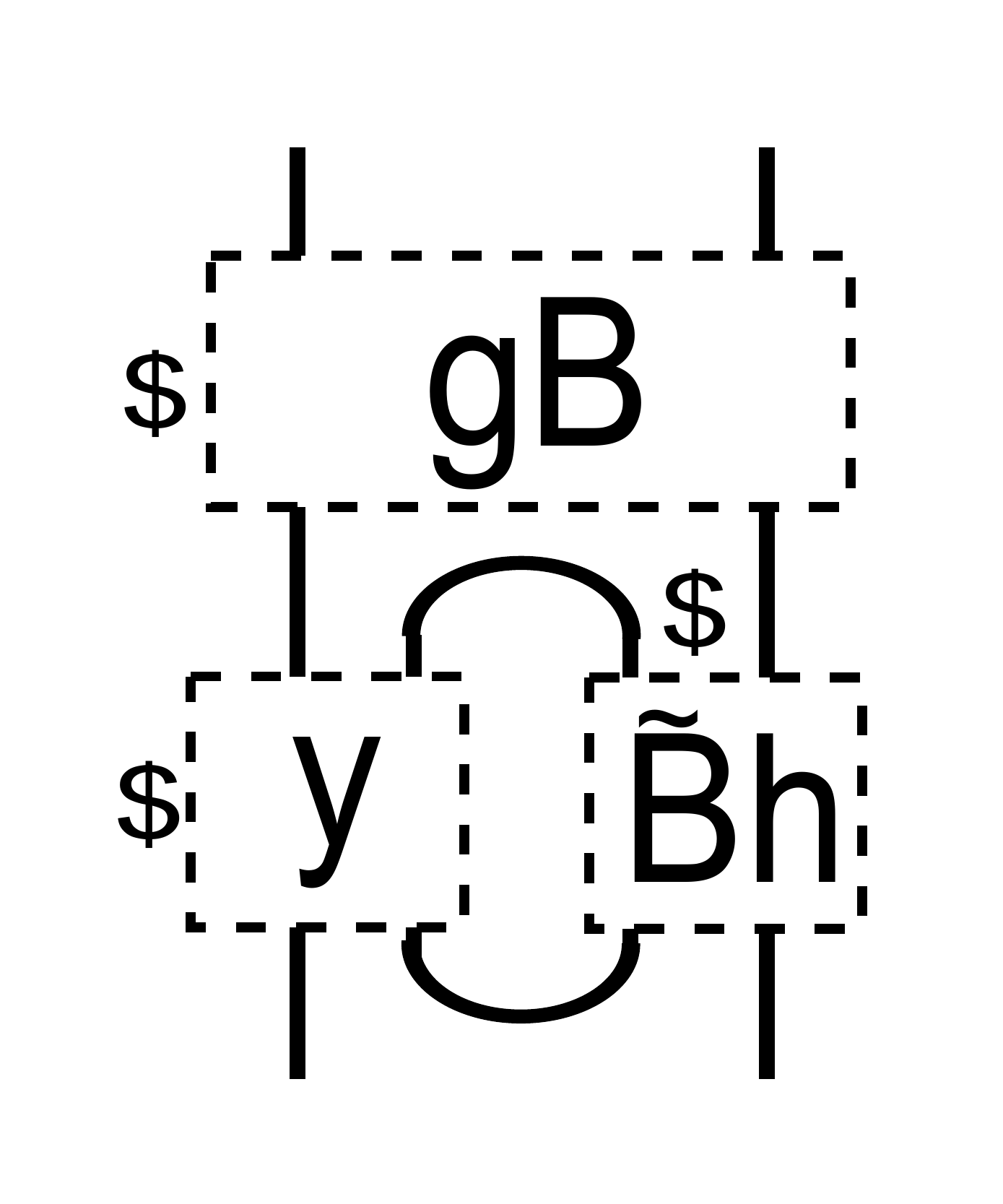}

\grc{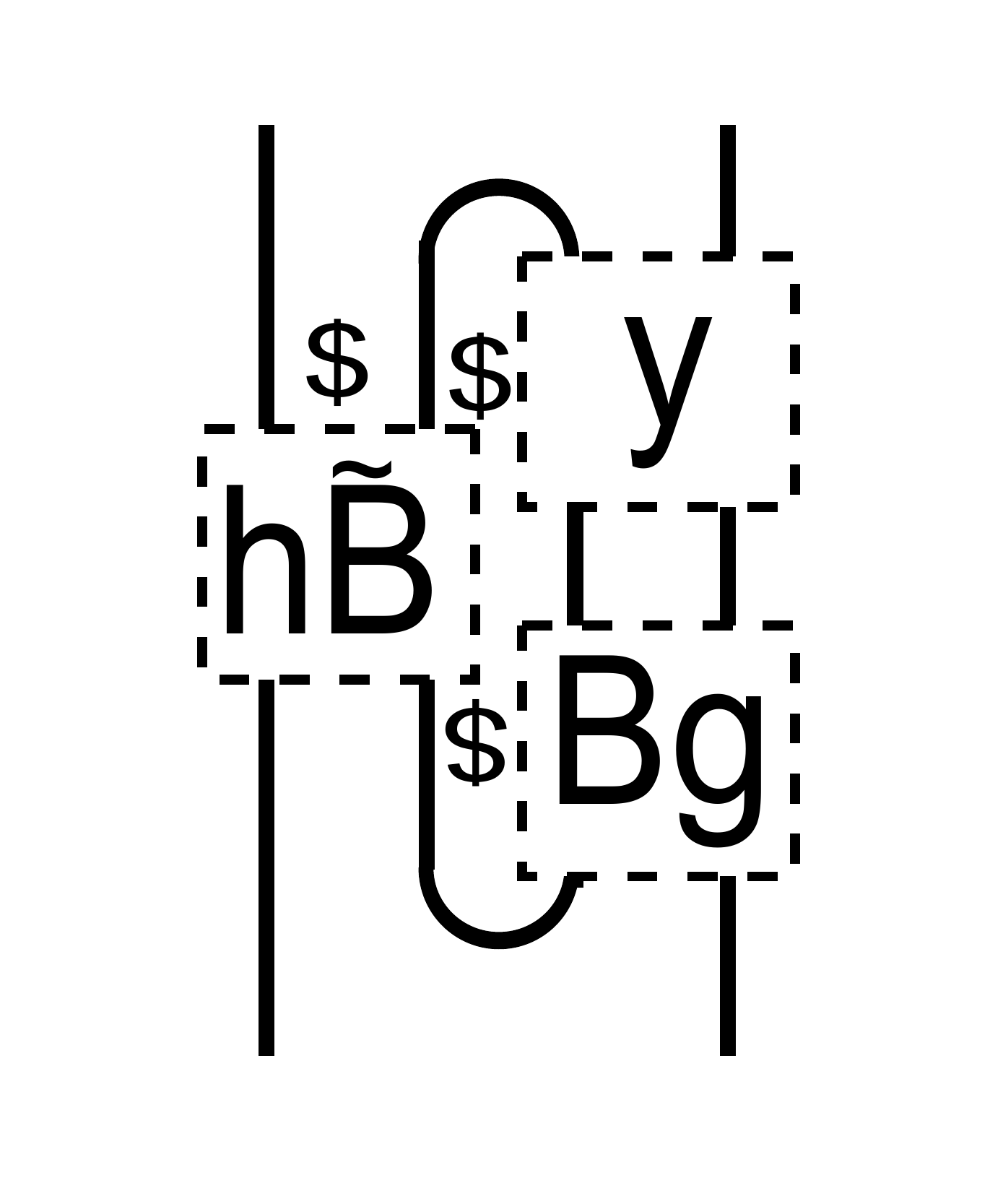} \grc{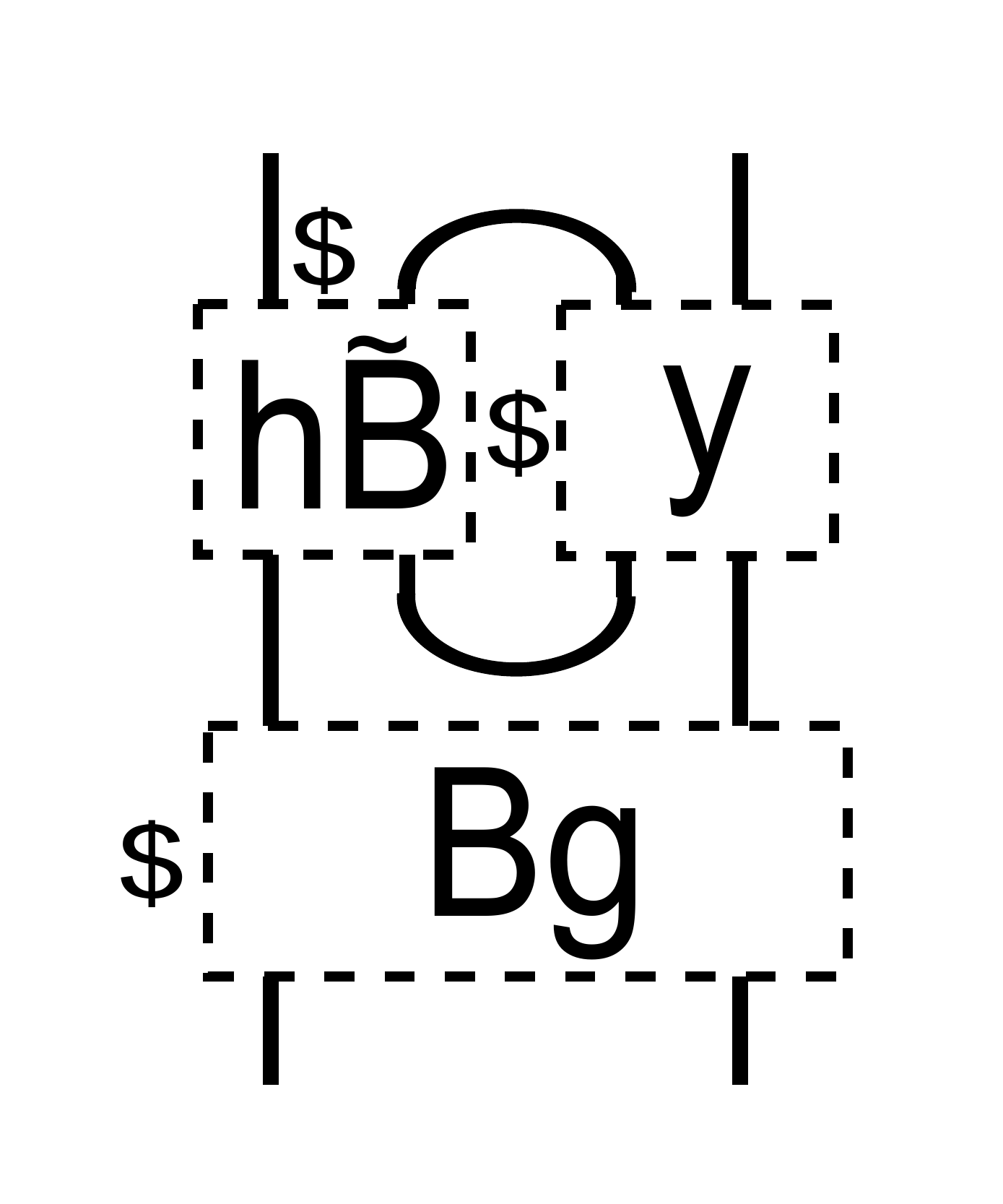} \grc{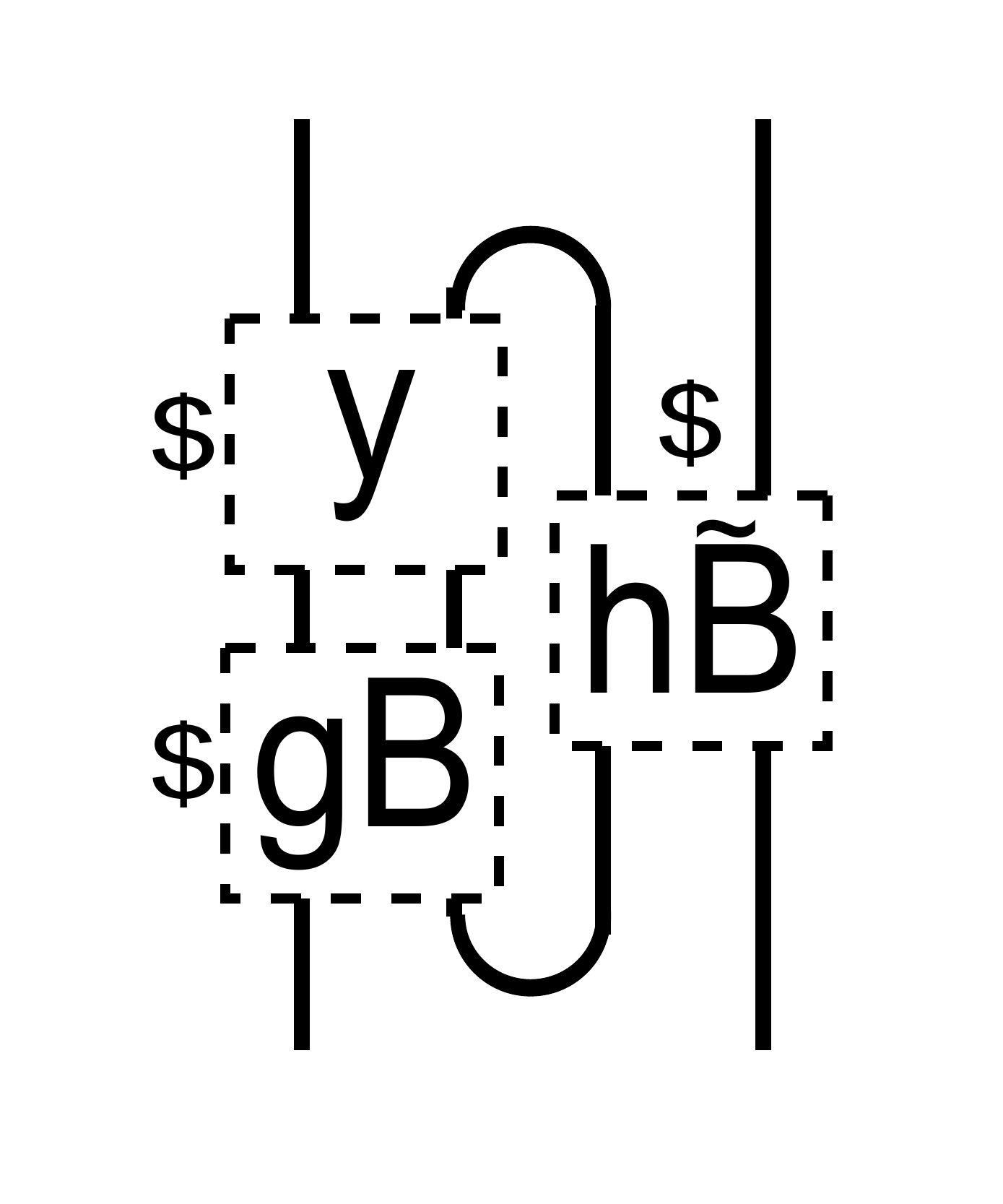} \grc{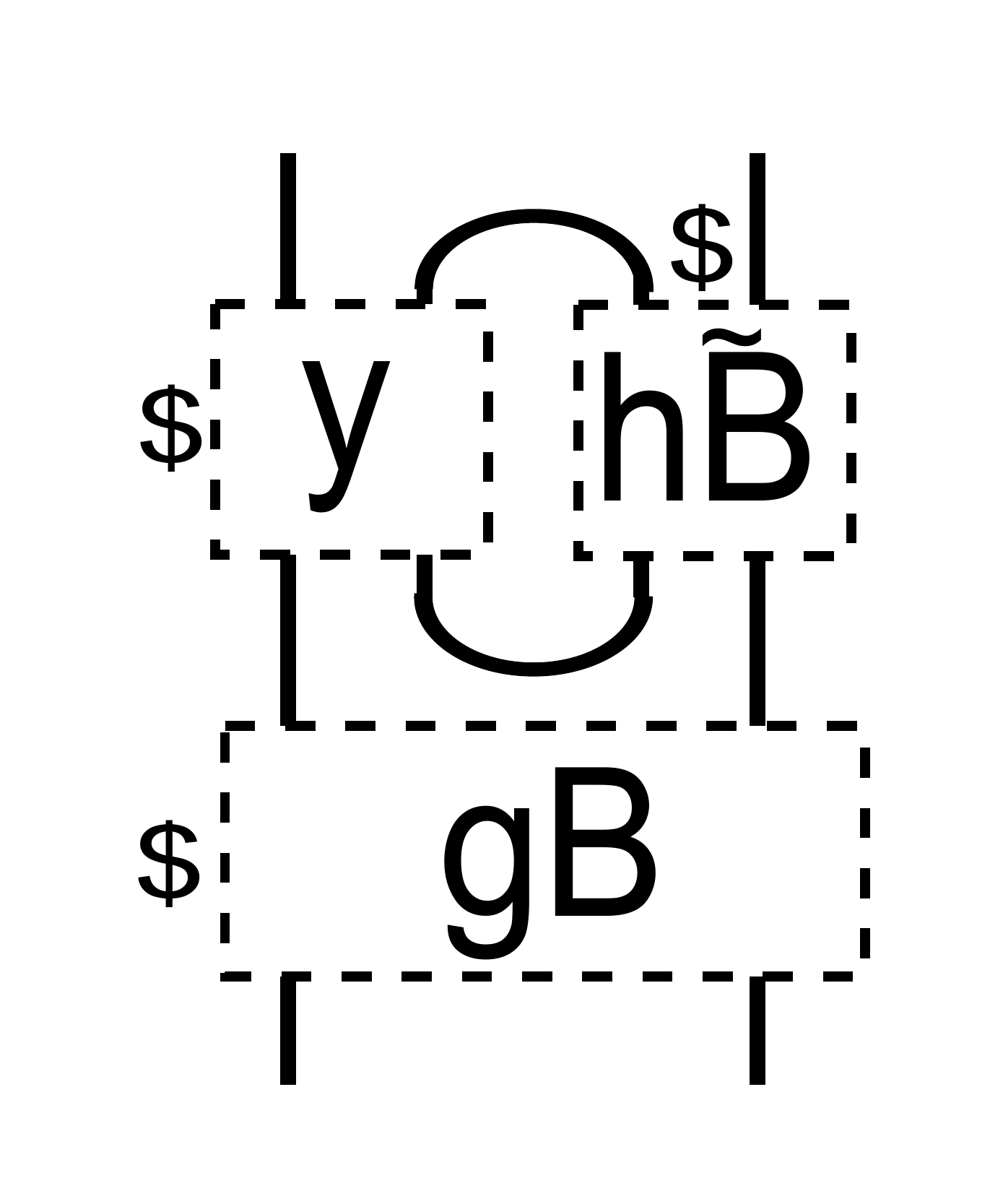}.
\end{center}

Although the definition of bishifts of biprojections is given by the first form.
By similar arguments, the proof of Theorem \ref{minmain2} also works for the other forms of bishifts of biprojections.
Consequently any form of bi-shift of a biprojection could be expressed as any other form of bi-shift of a possibly different biprojection. Therefore it does not matter what we choose as the definition of bishifts of biprojections.


\begin{thebibliography}{99}

\bibitem{Be75} W. Beckner,
     {Inequalities in fourier analysis,}
     Annals of Math. \textbf{102}(1975), 159-182.

\bibitem{Bis94} D. Bisch,
     {A note on intermediate subfactors,}
     Pacific J. Math. \textbf{163}(1994), 201-216.

\bibitem{Bis94e} D. Bisch,
     {An example of an irreducible subfactor of the hyperfinite II1 factor with rational, noninteger index,}
     J. reine angew. Math. \textbf{455}(1994), 21--34.

\bibitem{BJ97a} D. Bisch and V. Jones,
     {Algebras associated to intermediate subfactors,}
     Invent. Math. \textbf{128}(1997), 89-157.

\bibitem{Burn03} M. Burn,
     {Subfactor, planar algebras and rotations}
     Ph.D. thesis at the University of California, Berkeley, 2003, arXiv:1111.1362.


\bibitem{Con80} A. Connes,
    {On the spatial theory of von Neumann algebras.}
    J. Funct. Anal. \textbf{35} (1980), no. 2, 153-164.

\bibitem{CrKa14} J. Crann and M. Kalantar,
    {An uncertainty principle for unimodular quantum groups}
    ArXiv:1404.1276v1 [math-ph], 2014.

\bibitem{CRT} E. Candes, J. Romberg and T. Tao,
     {Robust uncertainty principles: exact signal reconstruction from highly incomplete frequency information}
     IEEE Transactions on Information Theory. \textbf{52}(2006), 489-509.


\bibitem{DeCoTh} A. Dembo, T. M. Cover and J. A. Thomas,
     {Information theoritc inequalities,}
     IEEE Transactions on Information Theory, \textbf{37}(1991),1501-1518.

\bibitem{DoSt} D.L. Donoho and P.B. Stark,
     {Uncertainty principles and signal recovery,}
     SIAM J. Appl. Math. \textbf{49}(1989), 906-931.


\bibitem{EnNe96} M. Enock and R. Nest,
     {Irreducible inclusions of factors, multiplicative unitaries and Kac algebras,}
     Journal of Func. Anal. \textbf{137}(1996), 466-543.

\bibitem{GGI} D. Goldstein, R. Guralnick, and I. Isaacs,
     {Inequalities for finite group permutation modules,}
     Trans. Amer. Math. Soc. \textbf{357}(2005), 4017-4042.

\bibitem{Hardy} G.H. Hardy,
     {A theorem concerning Fourier transforms,}
     Journal of the London Mathematical Society, \textbf{8}(1933), no.3, 227-231.

\bibitem{Hirsch} I.I. Hirschman,
     {A note on entropy,}
     Amer. Jour. Math. \textbf{79}(1957), 152-156.

\bibitem{Hor} L. H\"{o}rmander,
     {Notions of convexity,}
     Birkh$\ddot{a}$user, 1994.

\bibitem{Jon} V. Jones,
     {Planar algebras, I,}
     New Zealand Journal of Math. QA/9909027.

\bibitem{Jon83} V. Jones,
     {Index for subfactors,}
     Invent. Math. \textbf{72}(1983), 1-25.

\bibitem{Jon12} V. Jones,
     {Quadratic tangles in planar algebras,}
     Duke Math. J. \textbf{161}(2012), 2257-2295.

\bibitem{JonPen} V. Jones, D. Penneys,
     {Infinite index subfactors and the GICAR categories,}
     preprint, 2013


\bibitem{KLS} V. Kodiyalam, Z. Landau and V. Sunder,
     {The planar algebra associated to a Kac algebra,}
     Proc. Indian Acad. Sci. (Math Sci.) \textbf{113}(2003), no.1, 15-51.

\bibitem{Ko84} H. Kosaki,
     {Applications of the complex interpolation method to a von Neumann algebra: noncommutative $L^p$-spaces}
     J. Funct. Anal. 56(1984), 29-78.


\bibitem{Liu13} Z. Liu,
     {Exchange relation planar algebras of small rank,}
     arXiv:1308.5656v2.


\bibitem{Ocn88} A.~Ocneanu,
   {Quantized groups, string algebras and {G}alois theory for
  algebras}, Operator algebras and applications, Vol.\ 2, London Math. Soc.
  Lecture Note Ser., vol. 136, Cambridge Univ. Press, Cambridge, 1988,
  pp.~119--172.

\bibitem{OzPr} M. \"{O}zaydm and T. Przebinda,
     {An entropy-based uncertainty principle for a locally compact abelian group,}
     J. Funct. Anal. \textbf{215}(2004), 241-252.

\bibitem{Pop94} S. Popa,
     {Classification of amenable subfactor of type II,}
     Acta Math. \textbf{172}(1994), 352-445.

\bibitem{Pop95} S. Popa,
     {An axiomatization of the lattice of higher relative commutants,}
      Invent. Math. \textbf{120} (1995), 237--252.

\bibitem{Sm90} K. T. Smith,
     {The uncertainty principle on groups,}
     SIAM J. Appl. Math. \textbf{50}(1990), 876-882.

\bibitem{Sz94} W. Szymanski,
     {Finite index subfactors and Hopf algebra crossed products}
     Proc. Amer. Math. Soc. \textbf{120}(1994), no. 2, 519-528.

\bibitem{Tao05} T. Tao,
     {An uncertainty principle for cyclic groups of prime order,}
     Math. Res. Lett. \textbf{12} (2005), no. 1. 121-127.

\bibitem{Wen87} H. Wenzl,
     {On sequences of projections,}
     C.R. Math. Rep. Acad. Sci. Canada. \textbf{9}(1987), no. 1, 5-9.


\bibitem{Xu07} Q. Xu,
     {Operator spaces and noncommutative $L_p$: The part on noncommutative $L_p$-spaces,}
     Lectures in the Summer school on "Banach spaces and operator spaces".

\end{thebibliography}
\end{document}